\documentclass[11pt]{amsart}

\theoremstyle{plain}
\newtheorem{theorem}{Theorem}
\numberwithin{theorem}{section}
\newtheorem{lemma}[theorem]{Lemma}

\newtheorem{corollary}[theorem]{Corollary}

\theoremstyle{definition}

\newtheorem{remark}[theorem]{Remark}

\newtheorem{acknowledgments}{Acknowledgments\ignorespaces}

\newcommand{\C}{{\mathbb C}}

\newcommand{\Z}{{\mathbb Z}}
\newcommand{\Q}{{\mathbb Q}}
\renewcommand{\P}{{\mathbb P}}
\newcommand{\s}{{\mathbb S}}

              \newcommand{\J}{{\mathcal J}}

\begin{document}
\title[Symmetries of exotic $\C\P^2$]
{Symplectic symmetries of a hypothetical exotic $\C\P^2$: a potential construction or a nonexistence proof} 
\author{Weimin Chen}
\thanks{}
\date{\today}
\maketitle

\begin{abstract}
This paper is primarily concerned with the existence of a symplectic $4$-manifold, homeomorphic but not diffeomorphic to $\C\P^2$, which admits a nontrivial finite order symplectomorphism. We reduced the existence of such an exotic $\C\P^2$ to a question in contact geometry. More concretely, we showed that there are precisely $12$ plumbed manifolds, each equipped with a canonical contact structure associated to the plumbing, such that an exotic $\C\P^2$ with a nontrivial finite order symplectomorphism exists if and only if one of the $12$ contact plumbed manifolds admits a certain specific symplectic filling. As the main technical result, we completely determined the order and local representations of a pseudofree symplectic $\Z_p$-action for an odd prime $p$ on a rational homology $\C\P^2$ with positive canonical class. Our result suggests that potentially, there are symplectic 
$\Z_p$-actions which are not a finite automorphism of a fake projective plane. Moreover, we introduced a construction, which, under the minimal assumptions on the symplectic fillings of the $12$ contact plumbed manifolds, produces a symplectic $4$-manifold on the Bogomolov-Miyaoka-Yau line via a reverse-engineering process.
\end{abstract}

\tableofcontents

\section{Introduction}

An interesting question in $4$-manifold topology asks whether there is a symplectic $4$-manifold homeomorphic but not diffeomorphic to $\C\P^2$. This question may be regarded as the symplectic analogue of the classical conjecture of Severi (cf. \cite{Sev}), which was only resolved as a corollary of the celebrated work of S.T. Yau on Calabi's conjecture (cf. \cite{Y}). In other words, the answer to the above $4$-manifold question is negative in the K\"{a}hler category due to Yau's theorem. On the other hand, it has also been a challenging problem to construct an exotic smooth structure on $\C\P^2$. See \cite{C22} for some proposal for constructing a symplectic exotic $\C\P^2$ and \cite{C2} for some recent progress in some aspects of \cite{C22}. Note that by the seminal work of Taubes \cite{T}, a symplectic exotic $\C\P^2$ is characterized by having a positive canonical class (i.e. $c_1(K)\cdot [\omega]>0$). So in the symplectic case, the exoticness is relatively easy to detect.

In this paper, we shall consider a slightly modified version of the above question, i.e., we are investigating the existence of a symplectic $4$-manifold homeomorphic but not diffeomorphic to $\C\P^2$, which admits a nontrivial finite order symplectomorphism. The extra symmetry assumption allows us to approach the existence question via the minimal resolution of the symplectic quotient orbifold (see \cite{C11} for a general discussion), an idea first exploited by Keum in the algebraic setting to construct a fake projective plane (cf. \cite{Keum1}). In the symplectic setting, however, we shall employ a completely different strategy and construction, relying on the technique of pseudoholomorphic curves and results from symplectic and contact geometry and gauge theory. 

Central to our approach to the exotic $\C\P^2$ problem are $12$ plumbed manifolds, denoted by $Y_\Gamma$, where $\Gamma$ is one of the $12$ weighted graphs listed in Figure 1. Each graph $\Gamma$ determines a symplectic plumbing $P_\Gamma$ and a symplectic divisor $D_\Gamma$ with normal crossings, whose regular neighborhood is modeled by $P_\Gamma$, and with this understood, $Y_\Gamma=-\partial P_\Gamma$ as oriented manifolds. We remark that the symplectic structure on $P_\Gamma$ depends on the areas of the symplectic surfaces in the symplectic divisor $D_\Gamma$, which are not part of the data contained in the graph $\Gamma$. 

There are two structures associated to $Y_\Gamma$ which are essential to our construction. First,
there is a canonical contact structure $\xi_{LM}$ on $Y_\Gamma$, which we will refer to as the
{\bf Li-Mak contact structure} of the plumbed manifold $Y_\Gamma$. Indeed, observe that for each 
$\Gamma$ in Figure 1, the intersection matrix of $\Gamma$ is nonsingular and is not negative-definite. By the work of Li and Mak \cite{LM}, one can arrange the areas of the symplectic surfaces in the symplectic divisor $D_\Gamma$ to obey a certain positivity condition (i.e.,  GS-criterion, see \cite{GS}), such that the symplectic plumbing $P_\Gamma$ has a concave contact boundary 
$(Y_\Gamma,\xi_{LM})$, where $\xi_{LM}$ is uniquely determined by $\Gamma$ up to a contact 
isotopy (cf. \cite{LM}, Sec. 2.3.1).  

Secondly, observe that each $\Gamma$ in Figure 1 is given as $\Gamma_p$ or $\Gamma_{p,i}$
for some $i$, where $p$ is an odd prime number. With this understood, the other structure associated 
to $Y_\Gamma$ is a natural surjective homomorphism $\rho: \pi_1(Y_\Gamma)\rightarrow \Z_p$,
which is obtained as follows. Note that in $\Gamma$, there is a unique vertex, the central vertex, which is characterized by the property that after removing the vertex and the edges connecting to it, the graph $\Gamma$ becomes a disjoint union of three linear chains of spheres. Moreover, each linear chain of spheres corresponds to the exceptional divisors in the minimal resolution of some isolated quotient singularity of order $p$. Thus if we symplectically blow down the three linear chains of spheres in the symplectic plumbing $P_\Gamma$, we obtain a symplectic $4$-orbifold with concave contact boundary $(Y_\Gamma,\xi_{LM})$, which we will denote by $\overline{P}_\Gamma$. With this understood, a key fact is that $\overline{P}_\Gamma$ is a global quotient. More concretely, there is a symplectic $4$-manifold $\tilde{P}_\Gamma$ with concave contact boundary $(\tilde{Y}_\Gamma,\tilde{\xi}_{LM})$, which is equipped with a pseudofree symplectic $\Z_p$-action, such that 
$\overline{P}_\Gamma=\tilde{P}_\Gamma/\Z_p$ is the symplectic quotient orbifold (cf. Lemma 4.4). It follows immediately that there is a $p$-fold covering $\tilde{Y}_\Gamma\rightarrow Y_\Gamma$ such that the contact structure $\tilde{\xi}_{LM}$ is the pull-back of $\xi_{LM}$. 
This said, $\rho: \pi_1(Y_\Gamma)\rightarrow \Z_p$ is the corresponding surjective homomorphism.


With the preceding understood, we now state our first theorem.

\begin{theorem}
(1) Let $\Gamma$ be any given graph in Figure 1, which is given as $\Gamma_p$ or 
$\Gamma_{p,i}$. Assume $W_\Gamma$ is a strong symplectic filling of the contact plumbed manifold $(Y_\Gamma,\xi_{LM})$. We consider the following conditions on $W_\Gamma$.
\begin{itemize}
\item [{(i)}] There is a surjective homomorphism $\bar{\rho}: \pi_1(W_\Gamma)\rightarrow \Z_p$ extending the homomorphism $\rho: \pi_1(Y_\Gamma)\rightarrow \Z_p$, i.e., $\rho=\bar{\rho}\circ i_\ast$ where $i_\ast: \pi_1(Y_\Gamma)\rightarrow \pi_1(W_\Gamma)$. Let $\tilde{W}_\Gamma\rightarrow W_\Gamma$ be the $p$-fold covering associated to $\bar{\rho}: \pi_1(W_\Gamma)\rightarrow \Z_p$. Note that $\tilde{W}_\Gamma$ is a strong symplectic filling of 
$(\tilde{Y}_\Gamma,\tilde{\xi}_{LM})$.
\item [{(ii)}] $H_1(W_\Gamma,\Q)=0$ and $i_\ast: H_2(Y_\Gamma,\Q)\rightarrow H_2(W_\Gamma,\Q)$ is an isomorphism. 
\item [{(iii)}] $H_1(\tilde{W}_\Gamma,\Q)=0$.
\item [{(iv)}] The image of the kernel of $i_\ast: \pi_1(\tilde{Y}_\Gamma)\rightarrow  \pi_1(\tilde{P}_\Gamma)$ under $i_\ast: \pi_1(\tilde{Y}_\Gamma)\rightarrow  \pi_1(\tilde{W}_\Gamma)$
normally generates $\pi_1(\tilde{W}_\Gamma)$ and the kernel of $i_\ast: \pi_1(\tilde{Y}_\Gamma)\rightarrow  \pi_1(\tilde{W}_\Gamma)$ maps onto $\pi_1(\tilde{P}_\Gamma)$ under $i_\ast: \pi_1(\tilde{Y}_\Gamma)\rightarrow  \pi_1(\tilde{P}_\Gamma)$.
\end{itemize}
If $W_\Gamma$ obeys (i) and (ii), then the symplectic $4$-manifold $M_\Gamma:=
\tilde{W}_\Gamma\cup_{(\tilde{Y}_\Gamma,\tilde{\xi}_{LM})} \tilde{P}_\Gamma$, which comes equipped with a pseudofree symplectic $\Z_p$-action, having $W_\Gamma\cup_{(Y_\Gamma,\xi_{LM})} \overline{P}_\Gamma$ as the symplectic quotient orbifold, is of a positive canonical class, with Euler characteristic $\chi(M_\Gamma)=3$ and signature $Sign(M_\Gamma)=1$. In particular, 
$M_\Gamma$ is a $4$-manifold on the Bogomolov-Miyaoka-Yau line. Furthermore, $M_\Gamma$ is a 
$\Q$-homology $\C\P^2$ if and only if (iii) is obeyed, and $M_\Gamma$ is an exotic $\C\P^2$ if and only if both (iii) and (iv) are obeyed. 

(2) Conversely, suppose there is a symplectic $\Q$-homology $\C\P^2$ with positive canonical class, denoted by $M$, which admits a pseudofree symplectic $\Z_p$-action for an odd prime $p$. Then there is a weighted graph $\Gamma$
in Figure 1, where $\Gamma=\Gamma_p$ or $\Gamma_{p,i}$ for some $i$, such that the corresponding contact plumbed manifold $(Y_\Gamma,\xi_{LM})$ admits a strong symplectic filling $W_\Gamma$ obeying (i)-(iii), and $M$ can be identified with the $4$-manifold $M_\Gamma$. Moreover, if $M$ is an exotic $\C\P^2$, then (iv) is also obeyed by $W_\Gamma$. 
\end{theorem}

\begin{remark}
Theorem 1.1 offers a potential construction of a symplectic exotic $\C\P^2$; indeed, if there is a weighted graph $\Gamma$ in Figure 1, such that the corresponding contact plumbed manifold 
$(Y_\Gamma,\xi_{LM})$ admits a strong symplectic filling $W_\Gamma$ obeying (i)-(iv) (e.g., $W_\Gamma$ obeys (i)-(ii) where $\bar{\rho}: \pi_1(W_\Gamma)\rightarrow \Z_p$ is isomorphic, cf. Corollary 4.6), then the corresponding $4$-manifold $M_\Gamma$ is a symplectic exotic $\C\P^2$. On the other hand, a symplectic $\Z_p$-action of prime order on an exotic $\C\P^2$ must be pseudofree with $p>2$ (cf. \cite{C}, Section 4). It follows easily that Theorem 1.1 also offers a possible nonexistence proof, i.e., if one can show that for any of the graphs $\Gamma$ in Figure 1, $(Y_\Gamma,\xi_{LM})$
does not admit any strong symplectic fillings obeying (i)-(iv), then there is no symplectic exotic 
$\C\P^2$ with a nontrivial finite order symplectomorphism. 
\end{remark}

\begin{picture}(500,610)
\put(90,560){\circle*{5}}
\put(320,560){\circle*{5}}
\put(60,560){\circle*{5}}
\put(120,560){\circle*{5}}
\put(75,590){\circle*{5}}
\put(105,590){\circle*{5}}
\put(60,575){\circle*{5}}
\put(120,575){\circle*{5}}
\put(60,560){\line(1,0){30}}
\put(60,560){\line(0,1){15}}
\put(60,560){\line(1,0){30}}
\put(90,560){\line(-2,1){30}}
\put(90,560){\line(-1,2){15}}
\put(90,560){\line(1,2){15}}
\put(90,560){\line(2,1){30}}
\put(90,560){\line(1,0){30}}
\put(75,590){\line(1,0){30}}
\put(120,560){\line(0,1){15}}
\put(90,555){\makebox(0,0)[t]{$[7],6$}}
\put(60,555){\makebox(0,0)[tr]{$-2$}}
\put(120,555){\makebox(0,0)[tl]{$-2$}}
\put(60,578){\makebox(0,0)[br]{$-2$}}
\put(120,578){\makebox(0,0)[bl]{$-2$}}
\put(75,595){\makebox(0,0)[b]{$-2$}}
\put(105,595){\makebox(0,0)[b]{$-2$}}
\put(90,530){\makebox(0,0)[t]{$\Gamma_3$}}
\put(290,560){\circle*{5}}
\put(260,560){\circle*{5}}
\put(350,560){\circle*{5}}
\put(380,560){\circle*{5}}
\put(305,575){\circle*{5}}
\put(335,575){\circle*{5}}
\put(305,590){\circle*{5}}
\put(335,590){\circle*{5}}
\put(320,560){\line(-1,1){15}}
\put(320,560){\line(1,1){15}}
\put(320,562){\line(-1,0){30}}
\put(320,558){\line(-1,0){30}}
\put(320,562){\line(1,0){30}}
\put(320,558){\line(1,0){30}}
\put(290,560){\line(-1,0){30}}
\put(350,560){\line(1,0){30}}
\put(305,590){\line(1,0){30}}
\put(305,590){\line(0,-1){15}}
\put(335,590){\line(0,-1){15}}
\put(320,555){\makebox(0,0)[t]{$[3],2$}}
\put(320,530){\makebox(0,0)[t]{$\Gamma_5$}}
\put(290,555){\makebox(0,0)[tr]{$-3$}}
\put(260,555){\makebox(0,0)[tr]{$-2$}}
\put(350,555){\makebox(0,0)[tl]{$-3$}}
\put(380,555){\makebox(0,0)[tl]{$-2$}}
\put(300,575){\makebox(0,0)[r]{$-2$}}
\put(355,575){\makebox(0,0)[r]{$-2$}}
\put(300,590){\makebox(0,0)[r]{$-2$}}
\put(355,590){\makebox(0,0)[r]{$-2$}}
\put(90,500){\circle*{5}}
\put(90,485){\circle*{5}}
\put(90,470){\circle*{5}}
\put(90,455){\circle*{5}}
\put(60,455){\circle*{5}}
\put(30,455){\circle*{5}}
\put(0,455){\circle*{5}}
\put(120,455){\circle*{5}}
\put(150,455){\circle*{5}}
\put(180,455){\circle*{5}}
\put(290,455){\circle*{5}}
\put(260,455){\circle*{5}}
\put(320,455){\circle*{5}}
\put(350,455){\circle*{5}}
\put(380,455){\circle*{5}}
\put(305,470){\circle*{5}}
\put(335,470){\circle*{5}}
\put(305,485){\circle*{5}}
\put(335,485){\circle*{5}}
\put(305,500){\circle*{5}}
\put(335,500){\circle*{5}}
\put(90,430){\makebox(0,0)[t]{$\Gamma_{7,1}$}}
\put(320,430){\makebox(0,0)[t]{$\Gamma_{7,2}$}}
\put(90,450){\makebox(0,0)[t]{$[1],0$}}
\put(320,450){\makebox(0,0)[t]{$[3],2$}}
\put(90,455){\line(1,0){30}}
\put(90,455){\line(-1,0){30}}
\put(320,455){\line(1,0){30}}
\put(320,455){\line(-1,0){30}}
\put(120,455){\line(1,0){30}}
\put(60,455){\line(-1,0){30}}
\put(350,455){\line(1,0){30}}
\put(290,455){\line(-1,0){30}}
\put(150,455){\line(1,0){30}}
\put(30,455){\line(-1,0){30}}
\put(320,455){\line(1,1){15}}
\put(320,455){\line(-1,1){15}}
\put(90,455){\line(0,1){15}}
\put(90,470){\line(0,1){15}}
\put(90,485){\line(0,1){15}}
\put(305,500){\line(0,-1){30}}
\put(335,500){\line(0,-1){30}}
\put(305,500){\line(1,0){30}}
\put(95,500){\makebox(0,0)[l]{$-2$}}
\put(95,485){\makebox(0,0)[l]{$-2$}}
\put(95,470){\makebox(0,0)[l]{$-3$}}
\put(120,450){\makebox(0,0)[tl]{$-3$}}
\put(150,450){\makebox(0,0)[tl]{$-2$}}
\put(180,450){\makebox(0,0)[tl]{$-2$}}
\put(60,450){\makebox(0,0)[tr]{$-3$}}
\put(30,450){\makebox(0,0)[tr]{$-2$}}
\put(0,450){\makebox(0,0)[tr]{$-2$}}
\put(350,450){\makebox(0,0)[tl]{$-2$}}
\put(380,450){\makebox(0,0)[tl]{$-4$}}
\put(290,450){\makebox(0,0)[tr]{$-2$}}
\put(260,450){\makebox(0,0)[tr]{$-4$}}
\put(300,470){\makebox(0,0)[r]{$-2$}}
\put(300,485){\makebox(0,0)[r]{$-2$}}
\put(300,500){\makebox(0,0)[r]{$-2$}}
\put(355,470){\makebox(0,0)[r]{$-2$}}
\put(355,485){\makebox(0,0)[r]{$-2$}}
\put(355,500){\makebox(0,0)[r]{$-2$}}
\put(90,395){\circle*{5}}
\put(90,375){\circle*{5}}
\put(30,390){\circle*{5}}
\put(30,375){\circle*{5}}
\put(30,360){\circle*{5}}
\put(60,395){\circle*{5}}
\put(90,375){\line(-4,1){60}}
\put(90,375){\line(0,1){20}}
\put(120,395){\circle*{5}}
\put(150,390){\circle*{5}}
\put(150,375){\circle*{5}}
\put(150,360){\circle*{5}}
\put(90,375){\line(4,1){60}}
\put(90,375){\line(4,-1){60}}
\put(90,375){\line(-4,0){60}}
\put(90,370){\makebox(0,0)[tr]{$[2],1$}}
\put(90,345){\makebox(0,0)[t]{$\Gamma_{7,3}$}}
\put(60,395){\line(1,0){60}}
\put(30,390){\line(0,-1){30}}
\put(150,390){\line(0,-1){30}}
\put(60,400){\makebox(0,0)[br]{$-2$}}
\put(90,400){\makebox(0,0)[b]{$-2$}}
\put(120,400){\makebox(0,0)[bl]{$-3$}}
\put(24,390){\makebox(0,0)[r]{$-3$}}
\put(24,375){\makebox(0,0)[r]{$-2$}}
\put(24,360){\makebox(0,0)[r]{$-2$}}
\put(155,390){\makebox(0,0)[l]{$-3$}}
\put(155,375){\makebox(0,0)[l]{$-2$}}
\put(155,360){\makebox(0,0)[l]{$-2$}}
\put(320,390){\circle*{5}}
\put(320,370){\circle*{5}}
\put(290,390){\circle*{5}}
\put(260,390){\circle*{5}}
\put(230,390){\circle*{5}}
\put(350,390){\circle*{5}}
\put(380,390){\circle*{5}}
\put(290,360){\circle*{5}}
\put(290,380){\circle*{5}}
\put(350,370){\circle*{5}}
\put(380,370){\circle*{5}}
\put(320,390){\line(0,-1){20}}
\put(230,390){\line(1,0){150}}
\put(320,369){\line(-3,-1){30}}
\put(320,368){\line(-3,1){30}}
\put(290,360){\line(0,1){20}}
\put(320,372){\line(-3,1){30}}
\put(320,372){\line(1,0){30}}
\put(320,368){\line(1,0){30}}
\put(350,370){\line(1,0){30}}
\put(328,365){\makebox(0,0)[t]{$[1],0$}}
\put(320,345){\makebox(0,0)[t]{$\Gamma_{7,4}$}}
\put(230,395){\makebox(0,0)[br]{$-2$}}
\put(260,395){\makebox(0,0)[br]{$-2$}}
\put(290,395){\makebox(0,0)[br]{$-2$}}
\put(320,395){\makebox(0,0)[b]{$-2$}}
\put(350,395){\makebox(0,0)[bl]{$-2$}}
\put(380,395){\makebox(0,0)[bl]{$-2$}}
\put(285,383){\makebox(0,0)[tr]{$-4$}}
\put(290,355){\makebox(0,0)[tr]{$-2$}}
\put(380,365){\makebox(0,0)[tl]{$-2$}}
\put(350,365){\makebox(0,0)[tl]{$-4$}}
\put(90,320){\circle*{5}}
\put(90,305){\circle*{5}}
\put(90,290){\circle*{5}}
\put(90,275){\circle*{5}}
\put(60,275){\circle*{5}}
\put(30,275){\circle*{5}}
\put(0,275){\circle*{5}}
\put(120,275){\circle*{5}}
\put(150,275){\circle*{5}}
\put(180,275){\circle*{5}}
\put(90,250){\makebox(0,0)[t]{$\Gamma_{7,5}$}}
\put(90,270){\makebox(0,0)[t]{$[1],0$}}
\put(90,273){\line(-1,0){30}}
\put(90,277){\line(-1,0){30}}
\put(90,273){\line(1,0){30}}
\put(90,277){\line(1,0){30}}
\put(92,275){\line(0,1){15}}
\put(88,275){\line(0,1){15}}
\put(90,320){\line(0,-1){30}}
\put(60,275){\line(-1,0){60}}
\put(120,275){\line(1,0){60}}
\put(95,320){\makebox(0,0)[l]{$-2$}}
\put(95,305){\makebox(0,0)[l]{$-2$}}
\put(95,290){\makebox(0,0)[l]{$-3$}}
\put(120,270){\makebox(0,0)[tl]{$-3$}}
\put(150,270){\makebox(0,0)[tl]{$-2$}}
\put(180,270){\makebox(0,0)[tl]{$-2$}}
\put(60,270){\makebox(0,0)[tr]{$-3$}}
\put(30,270){\makebox(0,0)[tr]{$-2$}}
\put(0,270){\makebox(0,0)[tr]{$-2$}}
\put(320,290){\circle*{5}}
\put(350,290){\circle*{5}}
\put(380,290){\circle*{5}}
\put(410,290){\circle*{5}}
\put(290,290){\circle*{5}}
\put(260,290){\circle*{5}}
\put(230,290){\circle*{5}}
\put(200,290){\circle*{5}}
\put(350,310){\circle*{5}}
\put(380,310){\circle*{5}}
\put(320,310){\circle*{5}}
\put(290,310){\circle*{5}}
\put(350,310){\line(0,-1){20}}
\put(200,290){\line(1,0){210}}
\put(290,310){\line(1,0){90}}
\put(410,285){\makebox(0,0)[t]{$-4$}}
\put(200,285){\makebox(0,0)[t]{$-2$}}
\put(230,285){\makebox(0,0)[t]{$-2$}}
\put(260,285){\makebox(0,0)[t]{$-2$}}
\put(290,285){\makebox(0,0)[t]{$-2$}}
\put(320,285){\makebox(0,0)[t]{$-3$}}
\put(350,285){\makebox(0,0)[t]{$-1$}}
\put(380,285){\makebox(0,0)[t]{$-3$}}
\put(290,315){\makebox(0,0)[b]{$-2$}}
\put(320,315){\makebox(0,0)[b]{$-2$}}
\put(350,315){\makebox(0,0)[b]{$-3$}}
\put(380,315){\makebox(0,0)[b]{$-2$}}
\put(320,265){\makebox(0,0)[t]{$\Gamma_{11,1}$}}
\put(90,175){\makebox(0,0)[t]{$\Gamma_{11,2}$}}
\put(120,220){\circle*{5}}
\put(90,220){\circle*{5}}
\put(60,220){\circle*{5}}
\put(150,220){\circle*{5}}
\put(60,220){\line(1,0){90}}
\put(-30,200){\line(1,0){120}}
\put(120,200){\line(1,0){60}}
\put(118,200){\line(0,1){20}}
\put(122,200){\line(0,1){20}}
\put(120,202){\line(-1,0){30}}
\put(120,198){\line(-1,0){30}}
\put(60,225){\makebox(0,0)[b]{$-2$}}
\put(90,225){\makebox(0,0)[b]{$-2$}}
\put(120,225){\makebox(0,0)[b]{$-3$}}
\put(150,225){\makebox(0,0)[b]{$-2$}}
\put(60,195){\makebox(0,0)[t]{$-2$}}
\put(90,195){\makebox(0,0)[t]{$-3$}}
\put(120,195){\makebox(0,0)[t]{$-1$}}
\put(150,195){\makebox(0,0)[t]{$-4$}}
\put(30,195){\makebox(0,0)[t]{$-2$}}
\put(0,195){\makebox(0,0)[t]{$-2$}}
\put(-30,195){\makebox(0,0)[t]{$-2$}}
\put(180,195){\makebox(0,0)[t]{$-3$}}
\put(120,200){\circle*{5}}
\put(150,200){\circle*{5}}
\put(180,200){\circle*{5}}
\put(90,200){\circle*{5}}
\put(60,200){\circle*{5}}
\put(30,200){\circle*{5}}
\put(0,200){\circle*{5}}
\put(-30,200){\circle*{5}}
\put(350,180){\circle*{5}}
\put(350,195){\circle*{5}}
\put(350,210){\circle*{5}}
\put(350,225){\circle*{5}}
\put(350,180){\circle*{5}}
\put(380,180){\circle*{5}}
\put(410,180){\circle*{5}}
\put(440,180){\circle*{5}}
\put(320,180){\circle*{5}}
\put(290,180){\circle*{5}}
\put(260,180){\circle*{5}}
\put(230,180){\circle*{5}}
\put(200,180){\circle*{5}}
\put(170,180){\circle*{5}}
\put(320,160){\makebox(0,0)[t]{$\Gamma_{13,1}$}}
\put(170,180){\line(1,0){270}}
\put(350,180){\line(0,1){45}}
\put(355,195){\makebox(0,0)[l]{$-5$}}
\put(355,210){\makebox(0,0)[l]{$-2$}}
\put(355,225){\makebox(0,0)[l]{$-2$}}
\put(170,175){\makebox(0,0)[t]{$-2$}}
\put(200,175){\makebox(0,0)[t]{$-2$}}
\put(230,175){\makebox(0,0)[t]{$-2$}}
\put(260,175){\makebox(0,0)[t]{$-2$}}
\put(290,175){\makebox(0,0)[t]{$-2$}}
\put(320,175){\makebox(0,0)[t]{$-3$}}
\put(350,175){\makebox(0,0)[t]{$-1$}}
\put(380,175){\makebox(0,0)[t]{$-2$}}
\put(410,175){\makebox(0,0)[t]{$-3$}}
\put(440,175){\makebox(0,0)[t]{$-3$}}
\put(320,130){\circle*{5}}
\put(320,100){\circle*{5}}
\put(320,70){\circle*{5}}
\put(320,115){\circle*{5}}
\put(320,85){\circle*{5}}
\put(350,70){\circle*{5}}
\put(380,70){\circle*{5}}
\put(410,70){\circle*{5}}
\put(440,70){\circle*{5}}
\put(290,70){\circle*{5}}
\put(260,70){\circle*{5}}
\put(230,70){\circle*{5}}
\put(200,70){\circle*{5}}
\put(320,45){\makebox(0,0)[t]{$\Gamma_{19}$}}
\put(320,65){\makebox(0,0)[t]{$[1],0$}}
\put(320,130){\line(0,-1){60}}
\put(200,70){\line(1,0){240}}
\put(325,130){\makebox(0,0)[l]{$-3$}}
\put(325,115){\makebox(0,0)[l]{$-2$}}
\put(325,100){\makebox(0,0)[l]{$-3$}}
\put(325,85){\makebox(0,0)[l]{$-2$}}
\put(350,65){\makebox(0,0)[tl]{$-2$}}
\put(380,65){\makebox(0,0)[tl]{$-3$}}
\put(410,65){\makebox(0,0)[tl]{$-2$}}
\put(440,65){\makebox(0,0)[tl]{$-3$}}
\put(290,65){\makebox(0,0)[tr]{$-2$}}
\put(260,65){\makebox(0,0)[tr]{$-3$}}
\put(230,65){\makebox(0,0)[tr]{$-2$}}
\put(200,65){\makebox(0,0)[tr]{$-3$}}
\put(90,100){\circle*{5}}
\put(150,115){\circle*{5}}
\put(150,100){\circle*{5}}
\put(150,85){\circle*{5}}
\put(150,70){\circle*{5}}
\put(30,115){\circle*{5}}
\put(30,100){\circle*{5}}
\put(30,85){\circle*{5}}
\put(30,70){\circle*{5}}
\put(90,130){\circle*{5}}
\put(60,130){\circle*{5}}
\put(30,130){\circle*{5}}
\put(120,130){\circle*{5}}
\put(90,95){\makebox(0,0)[t]{$[1],0$}}
\put(90,50){\makebox(0,0)[t]{$\Gamma_{13,2}$}}
\put(30,130){\line(1,0){90}}
\put(30,100){\line(1,0){120}}
\put(30,115){\line(0,-1){45}}
\put(150,115){\line(0,-1){45}}
\put(90,130){\line(0,-1){30}}
\put(30,135){\makebox(0,0)[b]{$-2$}}
\put(60,135){\makebox(0,0)[b]{$-2$}}
\put(90,135){\makebox(0,0)[b]{$-2$}}
\put(120,135){\makebox(0,0)[b]{$-4$}}
\put(25,115){\makebox(0,0)[r]{$-4$}}
\put(25,100){\makebox(0,0)[r]{$-2$}}
\put(25,85){\makebox(0,0)[r]{$-2$}}
\put(25,70){\makebox(0,0)[r]{$-2$}}
\put(155,115){\makebox(0,0)[l]{$-4$}}
\put(155,100){\makebox(0,0)[l]{$-2$}}
\put(155,85){\makebox(0,0)[l]{$-2$}}
\put(155,70){\makebox(0,0)[l]{$-2$}}
\put(200,15){\makebox(0,0)[t]{Figure 1: Weighted graphs of the symplectic plumbings}}
\end{picture}

More generally, we observe that under the minimal assumptions, i.e., $W_\Gamma$ obeying (i) and (ii), we obtain in Theorem 1.1 a symplectic $4$-manifold $M_\Gamma$ which is on the Bogomolov-Miyaoka-Yau line. It is currently an open question as whether there are any examples of symplectic $4$-manifolds on the Bogomolov-Miyaoka-Yau line which are not algebraic surfaces. Note that by Yau's solution of Calabi's conjecture \cite{Y}, such algebraic surfaces must be quotients of the unit complex ball in $\C^2$. 

Finally, we point out that by Theorem 1.1, a pseudofree symplectic $\Z_p$-action of prime order on a 
$\Q$-homology $\C\P^2$ with positive canonical class can be realized only if there is a $\Gamma=
\Gamma_p$ or $\Gamma_{p,i}$ for some $i$ in Figure 1, such that $(Y_\Gamma,\xi_{LM})$ is strong fillable. In this regard, it is known that for $\Gamma=\Gamma_3, \Gamma_{7,1}$ and 
$\Gamma_{7,5}$, $(Y_\Gamma,\xi_{LM})$ admits a strong symplectic filling (cf. Theorem 4.1),
which also obeys (i)-(iii) of Theorem 1.1. For the rest of the graphs in Figure 1, in general we do not even know whether $(Y_\Gamma,\xi_{LM})$ is tight or not. 

\vspace{2mm}

The main technical ingredient for the proof of Theorem 1.1 is a complete classification of the possible orders and local representations of a pseudofree symplectic $\Z_p$-action of odd prime order $p$ on a $\Q$-homology $\C\P^2$ with positive canonical class, which is stated in Theorem 1.3 below.
Since a symplectic $\Z_p$-action of prime order on an exotic $\C\P^2$ must be pseudofree with $p>2$ (cf. \cite{C}), by specializing in the simply connected case in Theorem 1.3, we obtain a complete classification of all possible prime order symplectic $\Z_p$-actions on a hypothetical symplectic exotic $\C\P^2$.

\begin{theorem}
Let $(M,\omega)$ be a $\Q$-homology $\C\P^2$ with $c_1(K)\cdot [\omega]>0$, where $K$ denotes the canonical line bundle of $(M,\omega)$, such that $M$ is equipped with a pseudofree smooth $\Z_p$-action preserving the symplectic structure $\omega$, where $p>2$ is a prime number. Then $p=3$, $5$, $7$, $11$, $13$, or $19$. Moreover, let $\J$ denote the set of $\Z_p$-invariant, $\omega$-compatible almost complex structures on $M$ which are integrable near the fixed point set of the $\Z_p$-action. Then for any generic $J\in\J$, there is a $\Z_p$-invariant $J$-holomorphic curve $C$ in $M$ which contains all $3$ fixed points of the $\Z_p$-action and is embedded in the complement of the fixed point set. Further information about $C$ as well as the local representations of the $\Z_p$-action at the fixed point set are stated below:
\begin{itemize}
\item [{(1)}] For $p=3$, $C$ is of genus $25$, has a nodal singularity at each of the $3$ fixed points, and 
$c_1(K)=\frac{1}{2}C$. The rotation numbers at each of the $3$ fixed points is $(1,2)$ {\em(}for any generator 
of $\Z_p${\em)}. 
\item [{(2)}] For $p=5$, $C$ is of genus $23$, has a nodal singularity at one of the fixed points and has two embedded local branches with a tangency of order $2$ at each of the other two fixed points, and 
$c_1(K)=\frac{1}{2}C$. For some generator of $\Z_p$, the rotation numbers is $(1,4)$ at the nodal singularity and are $(4,3)$ and $(1,2)$ respectively at the other two fixed points. 
\item [{(3)}] For $p=7$, there are five subcases: 
\begin{itemize}
\item [{(i)}] $C$ is embedded, of genus $10$, and $c_1(K)=C$. The rotation numbers at the $3$ fixed points are $(1,3)$, $(6,2)$, and $(4,5)$ for some generator of $\Z_p$.
\item [{(ii)}] $C$ is of genus $27$, has a nodal singularity at one of the fixed points and is embedded near the other two fixed points, and $c_1(K)=\frac{1}{2}C$. For some generator of $\Z_p$, the rotation numbers is $(1,6)$ at the nodal singularity and are $(6,5)$ and $(1,2)$ respectively at the other two fixed points.
\item [{(iii)}] $C$ is of genus $23$, has a simple cusp singularity at one of the fixed points, a nodal singularity at another, and a reducible singularity at the third fixed point, where one of the local branches is embedded and the other local branch contains a simple cusp, with the intersection number of the two local branches equaling $2$, 
and $c_1(K)=\frac{1}{2}C$. For some generator of $\Z_p$, the rotation numbers is $(2,3)$ at the simple cusp, $(2,6)$ at the nodal singularity, and $(4,6)$ at the reducible singularity.
\item [{(iv)}] $C$ is of genus $19$, with $c_1(K)=\frac{1}{2}C$, and at one of the fixed points, $C$ has three embedded local branches, two of which are tangent of order $2$ intersecting the third local branch transversely,
while at another fixed point, $C$ has two embedded local branches which are tangent of order $2$, and at the last fixed point, $C$ has a cuspidal singularity whose link is a $(3,4)$-torus knot. For some generator of $\Z_p$, the rotation numbers are $(1,4)$, $(6,3)$, and $(3,4)$ respectively at the first, second, and third fixed point.
\item [{(v)}] $C$ is of genus $19$, with $c_1(K)=\frac{1}{2}C$, and at each of the $3$ fixed points, $C$ has two embedded local branches which are tangent of order $3$. For some generator of $\Z_p$, the rotation numbers are $(1,3)$, $(6,2)$, and $(4,5)$ respectively at the $3$ fixed points. 
\end{itemize}
\item [{(4)}] For $p=11$, there are two subcases: 
\begin{itemize}
\item [{(i)}] $C$ is of genus $5$, has a simple cusp singularity at one of the fixed points and is 
embedded near the other two fixed points, with $c_1(K)=\frac{3}{2}C$. The rotation numbers at the $3$ fixed points are $(2,3)$, $(2,10)$, and $(3,1)$ respectively for some generator of $\Z_p$. 
\item [{(ii)}] $C$ is of genus $15$, with $c_1(K)=\frac{1}{2}C$, and is embedded at one of the fixed points, has two embedded local branches at another, which are tangent of order $5$, and has two tangent local branches at the third fixed point, where each local branch contains a simple cusp singularity. For some generator of $\Z_p$, the rotation numbers are $(1,3)$, $(10,2)$, and $(8,9)$ respectively at the first, second, and third fixed point.
\end{itemize}
\item [{(5)}] For $p=13$, there are two subcases: 
\begin{itemize}
\item [{(i)}] $C$ is embedded, of genus $6$,  with $c_1(K)=\frac{3}{2}C$. The rotation numbers at the $3$ fixed points are $(1,6)$, $(1,8)$ and $(6,5)$ for some generator of $\Z_p$.
\item [{(ii)}] $C$ is of genus $19$, has a cuspidal singularity of $\delta$-invariant $3$ at each of the fixed points, 
with $c_1(K)=\frac{1}{2}C$. The rotation numbers at the $3$ fixed points are $(3,4)$, $(10,1)$, and $(9,12)$ for some generator of $\Z_p$.
\end{itemize}
\item[{(6)}] For $p=19$, $C$ is embedded, of genus $28$, with $c_1(K)=\frac{1}{2}C$. The rotation numbers at the $3$ fixed points are $(1,12)$, $(18,11)$, and $(7,8)$ respectively for some generator of $\Z_p$.
\end{itemize}
\end{theorem}

\begin{remark}
(1) Primary examples of a symplectic $\Q$-homology $\C\P^2$ with positive canonical class, which are also the only known examples, are the fake projective planes, which are completely classified by Prasad and Yeung \cite{PY} following Yau's solution of Calabi's conjecture \cite{Y} (see also Cartwright and Steger \cite{CS} for the computer-assisted group theoretic computations). Based on their classification of the fake projective planes, Keum \cite{Keum2} further classified the quotient space and local representations of a finite automorphism group of a fake projective plane. In particular, it was shown that the prime order automorphisms are of order either $3$ or $7$. Regarding symplectic $\Q$-homology $\C\P^2$ with positive canonical class as the symplectic analog of the fake projective planes, Theorem 1.3 suggests that potentially, there are symplectic $\Z_p$-actions on a symplectic fake projective plane which are not algebraic.  

More concretely, note that in Cases (2), (4), (5), and (6) of Theorem 1.3, the order $p\neq 3$ or $7$. On the other hand, concerning the action near the fixed points, a symplectic $\Z_p$-action in Theorem 1.3 is said of {\bf holomorphic} local representations if its local representations are the same as the local representations of a $\Z_p$-automorphism of a fake projective plane. A symplectic $\Z_p$-action is said of {\bf linear} local representations if its local representations are the same as that of a linear 
$\Z_p$-action on $\C\P^2$. One can check easily that holomorphic local representations are also linear. With this understood, in Theorem 1.3 the symplectic $\Z_p$-action has holomorphic local representations in Cases (1), (3)(i), and (3)(v). (In fact, the local representations in Cases (3)(i) and (3)(v) are the same; the difference lies in the invariant curves. In all other cases, the symplectic $\Z_p$-actions differ in either the order $p$ or the local representations.) Moreover, it has linear local representations (but not holomorphic) in Cases (2), (3)(ii), (3)(iv), (4)(ii), (5)(ii), and (6). In the remaining cases, i.e., Cases (3)(iii), (4)(i), and (5)(i), the symplectic $\Z_p$-action has exotic (i.e., nonlinear) local representations. We remark that none of the symplectic $\Z_p$-actions in Theorem 1.3 can be further eliminated by the existing obstructions (however, it is very likely that there are cases that can be ruled out as a consequence of Theorem 1.1). 

It follows easily that a $\Z_p$-automorphism of a fake projective plane, if regarded as a symplectic 
$\Z_p$-action on a $\Q$-homology $\C\P^2$, must belong to one of the following three cases:  (1), (3)(i), or (3)(v). It is clear that any $\Z_3$-automorphism of a fake projective plane belongs to the unique Case (1), although it is interesting to note that, according to the computations in \cite{CS}, there are at least six fake projective planes which admits a $\Z_3$-automorphism (cf. \cite{PY1}). However, for the case of $p=7$, the situation is less straightforward. There is a $\Z_7$-automorphism of a fake projective plane which could belong to either Case (3)(i) or (3)(v), and there is also a 
$\Z_7$-automorphism which only belongs to Case (3)(v). See Theorem 3.9 for more details. 

(2) One crucial assertion in Theorem 1.3 is that in all the cases of the group action, the invariant curve $C$ contains all $3$ fixed points of the action, and is embedded in the complement of the fixed-point set. This has the following important consequence: in the minimal resolution $X$ of the symplectic quotient orbifold, the proper transform of the quotient curve $C/\Z_p$ connects the exceptional divisors from the resolution of the $3$ quotient singularities to form a {\it connected} symplectic divisor $D$ with normal crossings. With this understood, the $12$ weighted graphs in Figure 1 are the corresponding plumbing graphs for the $12$ cases of group actions listed in Theorem 1.3, where the graphs are labeled as either $\Gamma_p$ or $\Gamma_{p,i}$, with $p$ being the order of the group action, and in the latter case $\Gamma_{p,i}$, with $i$ indicating the subcase of the group actions of order $p$ in the listing in Theorem 1.3, e.g., $\Gamma_{7,5}$ is the plumbing graph corresponding to the group action in Case (3)(v). 

By the work of Li and Mak \cite{LM}, after changing the symplectic structure via a deformation if necessary, a regular neighborhood of the symplectic divisor $D$, which is modeled by the symplectic plumbing $P_\Gamma$ according to the weighted graph $\Gamma$, has a concave contact boundary 
$(Y_\Gamma, \xi_{LM})$. This allows us to decompose the minimal resolution $X$ as a union of two pieces, one is the regular neighborhood of the symplectic divisor $D$, a standard piece, and the other is its complement in $X$, which is a strong symplectic filling of the contact plumbed manifold 
$(Y_\Gamma, \xi_{LM})$. This canonical decomposition of the minimal resolution of the quotient 
orbifold established a fundamental connection between the group actions under consideration and the symplectic fillings of the $12$ plumbed manifolds $(Y_\Gamma, \xi_{LM})$, which plays the key role in Theorem 1.1. Indeed, Theorem 1.1(2) follows readily from this decomposition, and Theorem 1.1(1) follows from a corresponding reverse-engineering procedure. 

\end{remark}

The paper is organized as follows. Section 2 is entirely occupied by the proof of Theorem 1.3, the long and technical part of this work. The main input here is Taubes's work \cite{T} applied in an equivariant or orbifold setting as developed in \cite{C1, C}. A very crucial ingredient is also a theorem of Edmonds and Ewing (cf. \cite{EE}) concerning the rotation numbers of the $\Z_p$-actions under consideration.  This is summarized in Lemma 2.1. We then use various techniques to regulate the set of $J$-holomorphic curves, which is invariant under the group action and is representing the class $c_1(2K)$, that we obtained from applying Taubes \cite{T}. This includes choosing a generic $\Z_p$-invariant $J$ so that we have the non-negativity of the virtual dimensions of the moduli spaces of orbifold $J$-holomorphic curves. We combine this with the genus formula and the adjunction inequality for singular curves, the constraint on the $\delta$-invariant of a cuspidal singularity from the rotation numbers of the group action, and a key constraint on the local intersection number of an embedded $J$-holomorphic disc with the $J$-holomorphic curves at a fixed point (cf. Lemma 2.7), which is derived based on Taubes \cite{T} in an equivariant setting. Finally, there is also the Global Quotient Criterion which is used in the last step. The result is the existence of an invariant curve containing all $3$ fixed points, a classification of its singularities, and the determination of the group order and the corresponding local representations of the group action. 

\vspace{2mm}

In Section 3, we examine the minimal resolution of the symplectic quotient orbifold. For each of the cases in Theorem 1.3, we determine the structure of the associated symplectic divisor in the minimal resolution, with the corresponding plumbing graphs listed in Figure 1. This is the content of Theorem 3.3, which also contains an expression for the canonical class of the minimal resolution as a linear
combination of the symplectic surfaces in the symplectic divisor. The expression is then used to obtain information about the symplectic Kodaira dimension and minimality of the minimal resolution. We also include an analysis of prime order automorphisms of fake projective planes from the point of view of this paper. In Section 4, we start with a brief review of the relevant results of Li and Mak \cite{LM}, ending with Theorem 4.1, asserting that a necessary condition for a given case of the $\Z_p$-actions in Theorem 1.3 to be realized is that the corresponding contact plumbed manifold 
$(Y_\Gamma,\xi_{LM})$ is strongly symplectically fillable. Then we establish a reverse-engineering procedure needed for a proof of Theorem 1.1, ending with the proof of Theorem 1.1. 

\section{Existence and classification of invariant curves}

This section is devoted to a proof of Theorem 1.3. To set up the stage, let $(M,\omega)$ be a symplectic $\Q$-homology $\C\P^2$ such that $c_1(K)\cdot [\omega]>0$, where $K$ denotes the canonical line bundle of $(M,\omega)$. Furthermore, we assume $(M,\omega)$ is equipped with a pseudofree symplectic $G$-action where $G=\Z_p$ for an odd prime $p$. It is clear that the $G$-action is trivial in rational homology, so by Lefschetz fixed point theorem (cf. \cite{AS}), the fixed point set of the $G$-action consists of $3$ isolated points. 

We let $\J$ be the set of $G$-invariant, $\omega$-compatible almost complex structures, such that each $J\in\J$
further obeys the following condition: $J$ is integrable near each of the $3$ fixed points. The space $\J$ is a nonempty, contractible Fr\'{e}chet space. 

First of all, we shall look at the local representations of the $G$-action at the fixed points. To this end, we fix
a $J\in \J$ so that for each fixed point $q$, $(T_q M,J_q)$ is a complex vector space.  The induced $G$-action on
$(T_q M,J_q)$ is completely determined by a pair of unordered integers $(a,b)$ where $0<a,b<p$, after we fix a generator $g\in G$, where $a,b$ are simply the weights of the action of $g$ on $(T_q M,J_q)$. Note that for any
$k\in\Z$, the weights of the action of $g^k$ are simply $ka \pmod{p}$ and $kb \pmod{p}$. The pair $(a,b)$ is called the {\bf rotation numbers} at the fixed point $q$, which is clearly independent of the choice of $J\in\J$ (but it depends on the choice of $g\in G$).

\begin{lemma}
Fixing any generator $g\in G$ and any fixed point $q$, there is a pair of integers $0<a<b<p$ such that 
$(a,b)$ is the rotation numbers at $q$, and the rotation numbers at the other two fixed points are given in the following four cases:
\begin{itemize}
\item [{(i)}] $(p-a,b-a)$ and $(p-b,p+a-b)$. 
\item [{(ii)}] $(a, p+a-b)$ and $(p-b, p+a-b)$.
\item [{(iii)}] $(p-a,b-a)$ and $(b, b-a)$.
\item [{(iv)}] $(a, p+a-b)$ and $(b, b-a)$.
\end{itemize}
In particular, we observe that if $(a,b)=(1,2)$, then (1) there is a fixed point $\tilde{q}\neq q$, such that for any
rotation numbers $(\tilde{a},\tilde{b})$ at $\tilde{q}$, $\tilde{a}+\tilde{b}=p$, and (2) with respect to either $g$ or 
$g^{-1}$, $(1,2)$ is also the rotation numbers at a fixed point $q^\prime \neq q$. 
\end{lemma}

\begin{proof}
Let $(a_i,b_i)$, for $i=1,2,3$, be the rotation numbers at the $3$ fixed points, where without loss of generality, we assume $(a_1,b_1)=(a,b)$. Then by a theorem of Edmonds and Ewing (cf. \cite{EE}), up to a simultaneous change of sign, the rotation numbers at the other two fixed points are given by the rotation numbers of a corresponding linear $G$-action on $\C\P^2$, which in particular implies that $a\neq b$. If we assume $a<b$ without loss of generality, then the rotation numbers of the linear action are $(p-a,b-a)$ and $(p-b,p+a-b)$ respectively at the other
two fixed points. Consequently, either $(a_2,b_2)$ or $(p-a_2,p-b_2)$ equals $(p-a,b-a)$ and  either $(a_3,b_3)$ or $(p-a_3,p-b_3)$ equals $(p-b,p+a-b)$ respectively. The lemma follows immediately. 

\end{proof}

We remark that case (i) of Lemma 2.1 corresponds precisely to the case of local representations of a linear action
on $\C\P^2$.

\subsection{Existence of invariant curves}

The following lemma was proved in \cite{C} except for the assertion concerning the fixed points.

\begin{lemma}
For any $J\in\J$, there is a finite set of irreducible $J$-holomorphic curves $\{C_i\}$ with multiplicities $\{n_i\}$,
such that $c_1(2K)=\sum_i n_i C_i$. Furthermore, the set $\cup_i C_i$ is invariant under the $G$-action and the $3$ fixed points of $G$ are all contained in $\cup_i C_i$.
\end{lemma}

\begin{proof}
Let $d(2K)$ be the virtual dimension of the $G$-invariant Seiberg-Witten moduli space associated to
the $G$-equivariant complex line bundle $2K$. Then the key to the proof is to show $d(2K)\geq 0$, 
as explained in \cite{C}, Section 4. The point is that, under this condition, a wall-crossing argument implies that the corresponding $G$-equivariant Seiberg-Witten invariant is nonzero in the Taubes chamber, and as a consequence, $c_1(2K)=\sum_i n_i C_i$ for a finite set of $J$-holomorphic curves $C_i$ such that $\cup_i C_i$ is invariant under the $G$-action and contains any fixed point $q$ whose rotation numbers $(a,b)$ obeys $2(a+b)\neq 0\pmod{p}$, which is equivalent to $a+b \neq 0\pmod{p}$ as $p$ is odd. 

We need to explain why the $J$-holomorphic curves $C_i$ can be so chosen such that $\cup_i C_i$ contains all
$3$ fixed points. To this end, we need to recall the calculation of $d(2K)$. Let $(a_i,b_i)$, for $i=1,2,3$, be the rotation numbers at the $3$ fixed points. Then $d(2K)$ is given by the following formula: 
$$
d(2K)=\frac{1}{p}(c_1(2K)^2- c_1(2K)\cdot c_1(K)+\sum_{i=1}^3 \sum_{x=1}^{p-1} 
\frac{2(\mu_p^{-2(a_i+b_i)x}-1)}{(1-\mu_p^{-a_ix})(1-\mu_p^{-b_ix})}),
$$
where $\mu_p:=\exp(2\pi i/p)$. See \cite{C1}, Appendix A, also Lemma 3.3 in \cite{C0}. In the present case,
with $c_1(K)^2=9$, it was shown in \cite{C}, Section 4, that $d(2K)$ is given by the following formula
$$
d(2K)=\frac{1}{p}(18+6(p-1)-2\sum_{i=1}^3 (\delta(a_i+b_i,b_i)+\delta (b_i,a_i))),
$$
where for any $0\leq c\leq p$, $0<d<p$, $\delta(c,d)$ is the unique solution to $c-d\delta=0 \pmod{p}$ where
$0\leq \delta<p$ (see Lemma 4.4 in \cite{C}). 

With the preceding understood, in order to evaluate $d(2K)$, we assume without loss of generality that
$(a_1,b_1)=(a,b)$, and $(a_2,b_2)$ and $(a_3,b_3)$ are given as in Lemma 2.1. Furthermore, for an illustration
we assume we are in case (i). Then we observe the following: 
\begin{itemize}
\item [{(1)}] $b-a \delta(b_1,a_1)=0 \pmod{p}$.
\item [{(2)}] $b-a-(p-a) \delta(b_2,a_2)= 0\pmod{p}$. 
\item [{(3)}] $p+a-b-(p-b) \delta (b_3,a_3)=0 \pmod{p}$. 
\item [{(4)}] $a+b- b\delta(a_1+b_1, b_1)=0 \pmod{p}$.
\item [{(5)}] $p-2a+b-(b-a) \delta(a_2+b_2, b_2)=0 \pmod{p}$. 
\item [{(6)}] $2p+a-2b -(p+a-b)\delta (a_3+b_3,b_3)=0 \pmod{p}$.
\end{itemize}
Note that (1) and (2) implies $\delta(b_1,a_1)+\delta(b_2,a_2)=1\pmod{p}$, (5) and (6) implies that
$\delta(a_2+b_2,b_2)+\delta(a_3+b_3,b_3)=3 \pmod{p}$, and finally,  (3) and (4) implies that 
$\delta(a_1+b_1,b_1)+\delta(b_3,a_3)=2 \pmod{p}$. It follows easily that 
$$
\delta(b_1,a_1)+\delta(b_2,a_2)=p+1,
$$
$$
\delta(a_2+b_2,b_2)+\delta(a_3+b_3,b_3)=\left \{\begin{array}{ll} 
0 & \mbox{ if $a_i+b_i=p$ for both $i=2,3$} \\
3 & \mbox{ if exactly for one of $i=2,3$, $a_i+b_i=p$ } \\
p+3 & \mbox{ if $a_i+b_i\neq 0 \pmod{p}$ for $i=2,3$ }\\
\end{array} \right. 
$$
and 
$$
\delta(a_1+b_1,b_1)+\delta(b_3,a_3)=\left \{\begin{array}{ll} 
2 & \mbox{ if $a_1+b_1=p$}\\
p+2 & \mbox{ if $a_1+b_1\neq 0\pmod{p}$.}
\end{array} \right. 
$$
We further observe that if $a_i+b_i=p$ for both $i=2,3$, then $b-2a=0\pmod{p}$ and $a-2b=0\pmod{p}$, which
implies that $p=3$ must be true, and $a_1+b_1=p$ as well. On the other hand, if $a_1+b_1=p$ and for one of 
$i=2,3$, $a_i+b_i=p$, $p=3$ must also be true, and $a_i+b_i=p$ for the other $i=2,3$ as well. 
With this understood, if we let $k$ be the number of fixed points whose rotation numbers $(a_i,b_i)$
obeys $a_i+b_i=p$, then either $k=3$ with $p=3$, or $k\leq 1$. It follows easily that when $k=3$, we have
$\sum_{i=1}^3 (\delta(a_i+b_i,b_i)+\delta (b_i,a_i))=p+3$, and as a result 
$$
d(2K)=\frac{1}{p}(18+6(p-1)-2(p+3))=\frac{1}{p}(6+4p)=6=2k,
$$
as in this case we have $p=3$. If $k=1$, we have $\sum_{i=1}^3 (\delta(a_i+b_i,b_i)+\delta (b_i,a_i))=2p+6$,
and 
$$
d(2K)=\frac{1}{p}(18+6(p-1)-2(2p+6))=2=2k,
$$
and if $k=0$, we have $\sum_{i=1}^3 (\delta(a_i+b_i,b_i)+\delta (b_i,a_i))=3p+6$, and $d(2K)=0$. 

With the preceding understood, we note that if $d(2K)=2k>0$, then we can arrange so that $\cup_i C_i$ contains
any chosen set of $k$ distinct points in $M$ (cf. Taubes \cite{T}). So when $k>0$, we may simply choose the $k$ distinct points to be the $k$ fixed points whose rotation numbers $(a_i,b_i)$ obeys $a_i+b_i=p$. This shows that we can always arrange so that $\cup_i C_i$ contains all $3$ fixed points. 

For the other cases, i.e., (ii)-(iv), the proof is identical. We simply make the following observation: if we change 
$c,d$ to $p-c,p-d$, the number $\delta(c,d)$ remains unchanged, i.e., $\delta(c,d)=\delta(p-c,p-d)$, and furthermore,
$a_i+b_i=p$ if and only if $(p-a_i)+(p-b_i)=p$. 

\end{proof}

To proceed further, let $L\in H_2(M)/Tor$ be the generator such that $\omega (L)>0$. Then we can speak of
the degree of a homology or cohomology class as the multiplicity of the class with respect to $L$. In particular,
$c_1(2K)$ is of degree $6$, and if in $c_1(2K)=\sum_i n_i C_i$, we let $d_i>0$ be the degree of $C_i$, then
one has the constraint $\sum_i n_i d_i=6$. 

In what follows, we shall choose $J\in\J$ to be generic, which has the effect that the virtual dimensions of the moduli spaces of orbifold $J$-holomorphic curves\footnote{The Fredholm theory of orbifold $J$-holomorphic curves is based on the theory of maps between orbifolds developed in \cite{C000}. The notion of maps formulated in \cite{C000} corresponds to the ad hoc notion of  ``a good (V-manifold) map equipped with an isomorphism class of compatible systems" in \cite{CR}. In particular, the proof of non-negativity of the virtual dimension, which follows by a standard argument, requires the corresponding structural theorem for the maps established in \cite{C000}.} in the symplectic $4$-orbifold $M/G$ are non-negative as long as the corresponding moduli space is nonempty (cf. \cite{C0}). We shall review the general formula for the virtual dimension of the moduli spaces after proving the next two lemmas, as the lemmas only involve the following two special cases, i.e., either the orbifold $J$-holomorphic curves contain no essential orbifold points or the orbifold $J$-holomorphic curves are spherical.

\begin{lemma}
Let $J\in\J$ be generic. Then each $J$-holomorphic curve $C\in\{C_i\}$, where $c_1(2K)=\sum_i n_i C_i$, must be invariant under the $G$-action. 
\end{lemma}

\begin{proof}
Suppose to the contrary that there is a $C\in\{C_i\}$ which is not invariant under the $G$-action. Then 
$\{g\cdot C|g\in G\}$ is a set of $p=|G|$ distinct $J$-holomorphic curves, which is a subset of $\{C_i\}$. 
It follows easily that the degree $d$ of $C$ is at most $2$.

On the other hand, let $f: \Sigma\rightarrow M$ be a $J$-holomorphic parametrization of $C$, and let $\pi: M\rightarrow M/G$ be the quotient map. Then $\pi\circ f: \Sigma\rightarrow M/G$ is a $J$-holomorphic parametrization of $\pi(C)$ because $C$ is not $G$-invariant and the order $p=|G|$ is prime. 
Let $g_\Sigma$ be the genus of $\Sigma$. Then the complex virtual dimension $d_C$ of the moduli space 
of $J$-holomorphic curves in the orbifold $M/G$ that contains $\pi(C)$ is given by
$$
d_C=c_1(T(M/G))\cdot \pi(C)+(g_\Sigma-1)=c_1(TM)\cdot C+(g_\Sigma-1)=-3d+ (g_\Sigma-1).
$$
(See \cite{CR}, Proposition 3.2.5). Since $J\in\J$ is generic, one has $d_C\geq 0$. On the other hand, by the adjunction inequality, we have $g_\Sigma\leq \frac{1}{2}(C\cdot C+c_1(K)\cdot C+2)=\frac{1}{2}(d^2+3d+2)$. Then $d_C\geq 0$ implies $d\geq 3$, which is a contradiction to $d\leq 2$ we obtained earlier.

\end{proof}

\begin{lemma}
Let $J\in\J$ be generic. Then each $J$-holomorphic curve $C\in\{C_i\}$, where $c_1(2K)=\sum_i n_i C_i$, must
be non-spherical, i.e., $C$ cannot be parametrized by a $J$-holomorphic map from $\s^2$.
\end{lemma}

\begin{proof}
Assume to the contrary that there is a $C\in\{C_i\}$ which is parametrized by a $J$-holomorphic map 
$f: \s^2\rightarrow M$. By Lemma 2.3, $C$ is invariant under the $G$-action, which implies that there is an induced $G$-action on the domain $\s^2$, with respect to which $f$ is $G$-equivariant. Let $x,\tilde{x}\in \s^2$ be the two fixed points and let $q=f(x)$, $\tilde{q}=f(\tilde{x})$ be the corresponding fixed points in $M$. Furthermore, let 
$g\in G$ be the generator which acts on $\s^2$ as a rotation of angle $2\pi/p$ at the fixed point $x$. With this understood, let $(a,b)$, $a<b$, be the rotation numbers at $q$ with respect to the generator $g\in G$. Finally, 
let $(\tilde{a},\tilde{b})$ be the rotation numbers at the fixed point $\tilde{q}$ with respect to the generator 
$\tilde{g}:=g^{-1}\in G$, which acts on $\s^2$ as a rotation of angle $2\pi/p$ at the fixed point $\tilde{x}$.

With the preceding understood, the complex virtual dimension $d_C$ for the component of the moduli space of 
orbifold $J$-holomorphic curves in the quotient orbifold $M/G$ that contains the curve $C/G$ is given by the following formula
$$
d_C=\frac{c_1(TM)\cdot C}{p}+1-\frac{a+b}{p}-\frac{\tilde{a}+\tilde{b}}{p}=1-\frac{3d+a+b+\tilde{a}+\tilde{b}}{p}, 
$$
where $d>0$ is the degree of $C$. (See \cite{C0}, Lemma 1.10, also \cite{CR}, Lemma 3.2.4).
Note that since $J\in\J$ is generic, one has $d_C\geq 0$. 

We shall derive a contradiction by examining the values of $(\tilde{a},\tilde{b})$ using Lemma 2.1. First,
if $\tilde{q}=q$, then $(\tilde{a},\tilde{b})=(p-a,p-b)$, which implies $d_C<0$, a contradiction. Next, suppose 
the rotation numbers at $\tilde{q}$ with respect to the generator $g$ is $(a_2,b_2)$ in Lemma 2.1. Then 
$(\tilde{a},\tilde{b})$ equals either $(p-a, b-a)$ or $(a, p+a-b)$. In either case $d_C<0$, also a contradiction. Lastly, assume the rotation numbers at $\tilde{q}$ with respect to the generator $g$ is $(a_3,b_3)$ in Lemma 2.1. Then 
$(\tilde{a},\tilde{b})$ equals either $(p-b, p+a-b)$ or $(b, b-a)$. In the former case, $d_C<0$, and in the latter,
$d_C=\frac{1}{p}(p-3d-3b)\geq 0$, which implies that $p=3d+3b$ as $d_C\in\Z$. But this contradicts the assumption that $p$ is prime, as $d>0$ and $b>0$. This finishes the proof of the lemma. 

\end{proof}

\subsection{Determination of the group orders}

To proceed further, let $C\in \{C_i\}$ be an irreducible component with degree $d$, and let $f: \Sigma\rightarrow M$ be a $J$-holomorphic parametrization of $C$. Since $C$ is $G$-invariant by Lemma 2.3, there is an induced $G$-action on $\Sigma$ such that $f$ is $G$-equivariant. Moreover, $\Sigma\neq\s^2$ by Lemma 2.4. Let $g_\Sigma>0$ be the genus of $\Sigma$, $g_{|\Sigma/G|}$ be the genus of the underlying Riemann surface of the $2$-orbifold 
$\Sigma/G$, and let $k$ be the number of the fixed points of the $G$-action on $\Sigma$. Then the following holds true:
$$
2-2g_{|\Sigma/G|}-k=\frac{1}{p}(2-2g_\Sigma-k).
$$
We should point out that either $k=0$, i.e., the $G$-action on $\Sigma$ is free, or $k\geq 2$, i.e., there are at least $2$ fixed points on $\Sigma$  (see \cite{C2}, Lemma 3.1). 

Let $d_C\in\Z$ be the complex virtual dimension of the moduli space of orbifold $J$-holomorphic curves in the orbifold $M/G$ which contains $C/G$. To state the formula of $d_C$ in the case of $k\geq 2$, we let $z_j$, where $j=1,2,\cdots,k$, be the fixed points on $\Sigma$, and let $g_j\in G$ be the element which acts in a neighborhood 
of $z_j$ as a counterclockwise rotation by an angle $2\pi/p$. Then (cf. \cite{CR}, Proposition 3.2.5)
$$
d_C=\frac{1}{p} c_1(TM)\cdot C+ (g_{|\Sigma/G|}-1)+k-\sum_{j=1}^k \frac{m_{j,1}+m_{j,2}}{p},
$$
where for each $j$, $(m_{j,1},m_{j,2})$ is the rotation numbers at the fixed point $f(z_j)$ with respect to $g_j\in G$. Note that under the assumption that $J\in\J$ is generic, $d_C\geq 0$.

With the proceeding understood, it proves to be useful to observe some rough estimate on the genus $g_\Sigma$ in terms of the degree $d$ of the $J$-holomorphic curve $C$. To this end, we label the $3$ fixed points in $M$ by $q_i$,  for $i=1,2,3$, and for each $i$, let $k_i\geq 0$ be the number of fixed points $z_j$ such that $f(z_j)=q_i$. Moreover, for each $i$ we define $\delta_i$ as follows: if $k_i=1$ and $C$ is singular at $q_i$, we let $\delta_i$ be the $\delta$-invariant of the cuspidal singularity $q_i\in C$; otherwise, let $\delta_i=0$. Then we observe the following adjunction inequality 
$$
g_\Sigma\leq \frac{1}{2}(d^2+3d+2)-\frac{1}{2}\sum_{i=1}^3 k_i(k_i-1)-\sum_{i=1}^3 \delta_i. 
$$
This follows easily from the following genus formula: for each $i$ with $k_i>1$, let $z_{j,i}$, $j=1,2,\cdots,k_i$,
be the fixed points in $\Sigma$ such that $f(z_{j,i})=q_i$, let $C_{j,i}$ be the local branch of $C$ at $q_i$ which is parametrized by a neighborhood of $z_{j,i}$ under $f$, and let $\delta_{j,i}$ be the corresponding $\delta$-invariant 
($\delta_{j,i}=0$ if $C_{j,i}$ is embedded), then 
$$
g_\Sigma= \frac{1}{2}(d^2+3d+2)-\frac{1}{2}\sum_{k_i>1}\sum_{j\neq j^\prime}C_{j,i}\cdot C_{j^\prime,i}-\sum_{k_i>1}\sum_{j=1}^{k_i} \delta_{j,i}-\sum_{i=1}^3 \delta_i - x\cdot p,
$$
where $C_{j,i}\cdot C_{j^\prime,i}$ is the local intersection number and $x\cdot p$, for $0\leq x\in\Z$, is the contribution from the singularities of $C$ lying in the complement of the fixed points $\{q_i|i=1,2,3\}$; in particular, when $p$ is relatively large, $x=0$. Note that  $C_{j,i}\cdot C_{j^\prime,i}\geq \alpha_{j,i}\alpha_{j^\prime,i}$ where
$\alpha_{j,i}$, $\alpha_{j^\prime,i}$ are the multiplicities of the corresponding singularities.

As a consequence, we have the following rough estimate:
$$
2-2g_\Sigma-k\geq -d^2-3d- \#\{k_i|k_i=1\}\geq -d^2-3d-3. 
$$
On the other hand, observe that since $g_\Sigma> 0$ (cf. Lemma 2.4), $2-2g_\Sigma-k\leq 0$ with ``$=$" if and only if $k=0$ and $g_\Sigma=1$, in which case one also has $g_{|\Sigma/G|}=1$. But this can be ruled out by the
fact that $d_C\geq 0$. Consequently, it follows easily from
$$
2-2g_{|\Sigma/G|}-k=\frac{1}{p}(2-2g_\Sigma-k)\neq 0, \mbox{ and } 2-2g_{|\Sigma/G|}-k\in \Z, 
$$
that $d^2+3d+3\geq p$. In particular, since $d\leq 6$, we see immediately that $p<57$. This fundamental finiteness for the order $p$ makes it possible for a classification of the possible symplectic $\Z_p$-actions on $M$ as stated in Theorem 1.3. We summarize our discussions in the following lemma.

\begin{lemma}
The prime order $p<57$. Moreover, 
$$
2-2g_\Sigma-k=-px\geq -(d^2+3d+3)\geq -57,
$$ 
where $0<x=2g_{|\Sigma/G|}-2+k\in\Z$.
\end{lemma}

Next, we make some useful observations concerning the local structure of a singularity of $C$ that occurs 
at a fixed point, in connection with the rotation numbers at the fixed point. Since we require $J\in\J$ to be integrable near the fixed points, $C$ is given by germs of holomorphic curves near a fixed point. In case of a cuspidal singularity of $C$ or a local branch of $C$, let the following be a Puiseux parametrization of the singularity:
$$
z_1=z^\alpha, \; z_2=c_1z^{\beta_1}+c_2z^{\beta_2}+\cdots 
$$
where each $c_j\neq 0$, and $1<\alpha<\beta_1<\beta_2< \cdots$. With this understood, the characteristic 
sequence of the singularity, $(\alpha; \beta_{j_1}, \beta_{j_2},\cdots)$ where $\{\beta_{j_k}\}$ is a subsequence of
$\{\beta_j\}$, is characterized by the condition that $\alpha>gcd(\alpha,\beta_{j_1})>gcd(\alpha,\beta_{j_1},\beta_{j_2})>\cdots > 1$, with $\alpha$ being the multiplicity of the singularity. It is well-known that the multiplicity sequence of the singularity, particularly the $\delta$-invariant, is completely determined by the characteristic sequence; more precisely, we have
$$
\text{$\delta$-invariant} =\frac{1}{2} \sum_{k=1}^m (\beta_{j_k}-1)(r_{k-1}-r_k)=\frac{1}{2}((\alpha-1)(\beta_{j_1}-1)+
\sum_{k=2}^m (r_{k-1}-1)(\beta_{j_k}-\beta_{j_{k-1}})),
$$
where $(\alpha; \beta_{j_1}, \beta_{j_2},\cdots, \beta_{j_m})$ is the characteristic 
sequence of the singularity, and $r_0=\alpha$, $r_k=gcd(r_{k-1}, \beta_{j_k})$ for $k\geq 1$, with $r_m=1$.
See \cite{BK, Milnor}. 

With this understood, we observe that if $g\in G$ is the generator such that 
$g\cdot z=\exp(2\pi/p)z$, and if $(a,b)$ is the rotation numbers with respect to $g\in G$ at the fixed point, then
$$
a=\alpha \pmod{p},\;\;\;  b=\beta_1=\beta_2=\cdots \pmod{p}.
$$
The following are some quick consequences: (i) if there is a $\beta_j$ such that $\beta_j=b$, then it must be 
$\beta_1=b$, and $\beta_j>b$ for all $j>1$, (ii) if $\alpha$ is a prime number, then one can always arrange so that 
$(\alpha,\beta_1)$ is the characteristic sequence; in particular, the $\delta$-invariant $\delta\geq \frac{1}{2}(a-1)(b-1)$, with ``$=$" if and only if $(a,b)=(\alpha,\beta_1)$. 

The following fact will be frequently used: in the case of a nodal singularity of $C$ at a fixed point, we observe that the tangent plane of each local branch must be invariant under the $G$-action. It follows easily that the rotation numbers for the two local branches must take the form $(m_{1,1},m_{1,2})=(1,m_1)$, $(m_{2,1},m_{2,2})=(1,m_2)$, where $m_1m_2=1\pmod{p}$; in particular, note that $g_2=g_1^{m_2}$. 

Finally, there is one more constraint, which we will call the {\bf Global Quotient Criterion}, stemming from the fact that the orbifold curve $C/G$ is a global quotient. See \cite{C2}, Lemma 3.1, for more details.

\vspace{2mm}

{\bf Global Quotient Criterion:} {\it Recall that $z_j$, for $j=1,2,\cdots,k$, are the fixed points on 
$\Sigma$, each assigned with a generator $g_j\in G$, with rotation numbers $(m_{j,1},m_{j,2})$. Assume without
loss of generality that $C$ contains all $3$ fixed points $q_1,q_2,q_3$, and $f(z_1)=q_1$. Moreover, applying
Lemma 2.1, where we write $(a,b)=(m_{1,1},m_{1,2})$, and assuming case (i) is true, we let $(p-a,b-a)$,
$(p-b,p+a-b)$ be the rotation numbers at $q_2,q_3$ for the generator $g_1$ without loss of generality. Then
let $c_1=1$, and for each $j>1$, there is a $c_j\in\Z$ such that if $f(z_j)=q_1$ (resp.  $f(z_j)=q_2$,
$f(z_j)=q_3$), then 
$$
(m_{j,1},m_{j,2})=(c_j a,c_j b)  \; (\mbox{resp.  }(c_j (p-a), c_j(b-a)), (c_j (p-b), c_j(p+a-b)) )\pmod{p}.
$$
Note that $g_j=g_1^{c_j}$ for each $j$. With this understood, $g_1g_2g_3\cdots g_k=1$ in $G$ implies the following constraint on the rotation numbers  $(m_{j,1},m_{j,2})$:
$$
c_1+c_2+c_3+\cdots+c_k=0 \pmod{p}. 
$$
For the cases (ii)-(iv) of Lemma 2.1, the discussion is completely analogous, and we leave the details to the reader.
}

\vspace{2mm}

With the proceeding understood, pick a $J\in\J$ such that $c_1(2K)=\sum_i n_i C_i$ by Lemma 2.2. 

\begin{lemma}
Let $J\in\J$ be generic. If there is a $C\in \{C_i\}$ of degree $d$ such that $d\leq 3$, then $7\leq p\leq 13$.
\end{lemma}

\begin{proof}
To show $p\geq 7$,  the key observation is that for each $j=1,2,\cdots,k$, $m_{j,1}\neq m_{j,2}$, so that $m_{j,1}+m_{j,2}\geq 3$ (see Lemma 2.1). On the other hand, $g_\Sigma\leq \frac{1}{2}(d^2+3d+2)$ by the adjunction inequality, so that 
$g_{|\Sigma/G|}-1\leq \frac{1}{2p}(d^2+3d)-k (\frac{1}{2}-\frac{1}{2p})$. It follows easily that
$$
d_C\leq \frac{d^2-3d}{2p}+k(\frac{1}{2}-\frac{5}{2p}). 
$$
With the assumption $d\leq 3$, and under the condition $p\leq 5$, we see immediately that $d_C\geq 0$ implies
that $d=3$ and $p=5$ must be true. Moreover, $m_{j,1}+m_{j,2}=3$, i.e., $(m_{j,1},m_{j,2})=(1,2)$, must be
true for each of the fixed point $f(z_j)\in C$. However, with $p=5$, this contradicts Lemma 2.1.  Hence $p\geq 7$. 

On the other hand, assume to the contrary that $p\geq 17$. Then by Lemma 2.5, $d^2+3d+3\geq p$, which implies that $d=3$ must be true, and $p=17$ or $19$. Furthermore, note that $d^2+3d+3<2p$, so that $2-2g_\Sigma-k=-p$ must be the only possibility in Lemma 2.5. Note that $2-2g_{|\Sigma/G|}-k=-1$, which implies that $k=3$ and $g_{|\Sigma/G|}=0$ as $k\geq 2$. In what follows, we shall rule out $p=17$ or $19$ by a case by case analysis. 

Consider first the case of $p=19$. From $2-2g_\Sigma-k=-p$ where $k=3$, it follows easily that
$g_\Sigma=9$, which, with $d=3$, implies that $C$ either has a cuspidal singularity of $\delta$-invariant $1$ (i.e., a simple cusp) at a fixed point or a nodal singularity at a fixed point, say $q_1$. In the former case, the rotation numbers $(m_{1,1},m_{1,2})$ at $q_1$ must be $(2,3)$, while the rotation numbers at the other two fixed
points $q_2,q_3$ must take the form of $(1,m_2)$ and $(1,m_3)$, as $C$ is embedded near $q_2,q_3$. With this understood, we note that
$$
d_C=-\frac{3d}{p}+ (g_{|\Sigma/G|}-1)+3-\sum_{j=1}^3 \frac{m_{j,1}+m_{j,2}}{p}
=\frac{1}{19}(22-m_2-m_3).
$$
With both $m_2,m_3\geq 2$, and $d_C\geq 0$ and $d_C\in\Z$, it follows easily that $m_2+m_3=22$. 
On the other hand, we apply Lemma 2.1 with $(a_1,b_1)=(a,b)=(2,3)$, we see the rotation numbers at the other two fixed points are $(a_2,b_2)=(17, 1)$ or $(2, 18)$, $(a_3,b_3)=(16, 18)$ or $(3, 1)$. It follows that $(1,m_2)$
is either $(1,9)$ or $(1,17)$, and $(1,m_3)$ is either $(1,3)$ or $(1,13)$. Now with $m_2+m_3=22$, we must have
$(1,m_2)=(1,9)$ and $(1,m_3)=(1,13)$. However, one can check easily that the Global Quotient Criterion fails here.  Hence the former case is eliminated. 

For the latter case where there is a nodal singularity at $q_1$, the rotation numbers $(m_{j,1},m_{j,2})$ takes the form of $(1,m_1)$ and $(1,m_2)$ for the two local branches. As $k=3$, there is another fixed point contained in $C$, where rotation numbers $(m_{j,1},m_{j,2})$ takes the form $(1,m_3)$. 
With this understood, observe that the fact that $d_C\geq 0$ and $d_C\in\Z$ implies that $m_1+m_2+m_3=26$. 
Now without loss of generality we assume $m_1\leq m_2$, Then observe the fact that $m_1m_2=1\pmod{19}$, $g_2=g_1^{m_2}$, which implies that $g_3=g_1^{-1-m_2}$. To proceed further, note that at the fixed points $q_2,q_3$, the rotation numbers with respect to $g_1\in G$ are $(18, m_1-1)$ or $(1,20-m_1)$ and 
$(19-m_1, 20-m_1)$ or $(m_1,m_1-1)$. Then $g_3=g_1^{-1-m_2}$ implies that $(1,m_3)$ must be one of the following: $(18-m_2, 19+m_1-m_2)$, $(1+m_2, m_2-m_1)$, $(1+m_1, 19+m_1-m_2)$, or $(18-m_1, m_2-m_1)$,
where $1<m_1<m_2$. One can check easily that any of the above choices for $(1,m_3)$ either contradicts 
$m_1 m_2=1\pmod{19}$, or $m_1+m_2+m_3=26$, or both. Hence the latter case is also eliminated. This completes the discussion for $p=19$. 

For $p=17$, $2-2g_\Sigma-k=-p$ with $k=3$ implies that $g_\Sigma=8$. With $d=3$, it follows easily
that $C$ is singular, either with (i) a cuspidal singularity of $\delta$-invariant $2$, or (ii) two simple cusp singularities, or (iii) one simple cusp singularity and one nodal singularity, all occurring at the fixed points. In case (i), the rotation numbers $(m_{j,1},m_{j,2})$ take the form of $(2,5)$, $(1,m_2)$, and $(1,m_3)$, where by a similar calculation with the virtual dimension $d_C$, $m_2+m_3=16$ must hold true. On the other hand, using Lemma 2.1, with 
$(a,b)=(2,5)$, one finds that $(1,m_2)$ is either $(1,7)$ or $(1,5)$, and $(1,m_3)$ is either $(1,4)$ or $(1,13)$. It is easy to see that none of the above choices yields $m_2+m_3=16$, hence (i) is ruled out. In case (ii), the rotation numbers $(m_{j,1},m_{j,2})$ take the form of $(2,3)$, $(2,3)$, and $(1,m_3)$. Using Lemma 2.1, one finds that this is not possible. Hence (ii) is also ruled out. Finally, consider case (iii). The rotation numbers $(m_{j,1},m_{j,2})$ take the form of $(1,m_1)$, $(1,m_2)$, and $(2,3)$, where the first two occur at the nodal point and the last one at the simple cusp singularity. With this understood, we note that $g_2=g_1^{m_2}$ and $g_3=g_1^{-1-m_2}$, which implies that $m_1m_2=1\pmod{17}$ and $(2,3)$ must be one of the following: $(16-m_2, 17+m_1-m_2)$, $(1+m_2, m_2-m_1)$, $(1+m_1, 17+m_1-m_2)$, or $(16-m_1, m_2-m_1)$, where $1<m_1<m_2$.
In fact, only $(2,3)=(16-m_1,m_2-m_1)$ is possible, but it is a contradiction to $m_1m_2=1\pmod{17}$. Thus (iii) is also eliminated, and consequently, $p\neq 17$. 

The proof of Lemma 2.6 is completed.

\end{proof}

It remains to examine the possibility that $c_1(2K)$ is represented by a single, degree $d=6$ $J$-holomorphic curve $C$, which is $G$-invariant and contains all $3$ fixed points $q_1,q_2,q_3$ (this is the case when the assumption of Lemma 2.6 does not hold); in particular, $k\geq 3$. For this case, the following lemma plays a key role in our analysis.

\begin{lemma}
For any of the fixed point $q_i$, $i=1,2,3$, suppose $(m_{1},m_{2})$ is the rotation numbers at $q_i$ with respect to some generator $g\in G$. Then there exists a pair of non-negative integers $a_1,a_2$, with $a_1+a_2>0$ and obeying 
$$
(a_1+2)m_{1}+(a_2+2)m_{2}=0 \pmod{p},
$$
such that for any local $J$-holomorphic disk $D$ containing the fixed point $q_i$, the local intersection number
of $D$ and $C$ at $q_i$, denoted by $int_D(q_i)$,  is at least $a_1+a_2$, i.e., 
$$
int_D(q_i)\geq a_1+a_2.
$$
In particular, if $C$ is embedded near $q_i$, and let $(1,m)$, where $m>1$, be the rotation numbers at $q_i$, then $p>5$, and either $2+3m=p$, or $2+3m=2p$, or $3+2m=p$.
\end{lemma}

\begin{proof}
The proof of the lemma is identical to that of Lemma 1.6 in \cite{C}, where instead of having the canonical line bundle $K$ as in \cite{C}, we have $2K$ here.  We refer the reader to \cite{C} for details. Regarding the special case where $C$ is embedded near $q_i$, we note that $int_D(q_i)=1$ for a generic $D$. Thus $a_1+a_2=1$ must be true, i.e., either $a_1=0$ and $a_2=1$, or $a_1=1$ and $a_2=0$. The lemma follows immediately. 

\end{proof}

We note that in Lemma 2.7, if the generator $g$ is changed to another one, i.e., $g^k$ for some $k\in\Z$, the rotation numbers $(m_{1},m_{2})$ will be changed accordingly, i.e., to $(km_{1},km_{2})$ mod $p$. It is easy to see that the conditions on $a_1,a_2$ remain unchanged. In particular, the constraint $int_D(q_i)\geq a_1+a_2$ from Lemma 2.7 does not depend on the choice of the rotation numbers $(m_{1},m_{2})$ involved in the argument. 

\begin{lemma}
Let $J\in\J$ be generic. Suppose in $c_1(2K)=\sum_i n_i C_i$ from Lemma 2.2, $\{C_i\}$ consists of a single component $C$ of degree $d=6$. Then $p\leq 19$, and $p\neq 17$. 
\end{lemma}

\begin{proof}
We shall divide $p=17$, or $p\geq 23$ into a few cases and rule out each of the cases using different strategies 
and arguments.

\vspace{2mm}

{\bf (1):} $p\geq 37$. We begin by observing that with $d=6$, $d^2+3d+3=57<2p$, so that 
$2-2g_\Sigma-k=-p$ is the only possibility in Lemma 2.5. As such, one has $g_{|\Sigma/G|}=0$ and $k=3$. Since $C$ contains all $3$ fixed points, it follows easily that $C$ has only cuspidal singularities. Furthermore, $p\geq 37$ implies that $g_\Sigma\geq 18$, so that by the adjunction inequality, the sum of the $\delta$-invariants of the singularities of $C$ can not exceed $10$. 

With the preceding understood, we observe that if the $\delta$-invariant of a singularity of $C$ is less than $6$, the multiplicity of the singularity must be either $2$ or $3$, which is prime. It follows, with $p\geq 37$, that its rotation numbers $(m_{j,1},m_{j,2})$ must equal the characteristic sequence $(\alpha,\beta_1)$, therefore it must be
one of the following: $(2,3),(2,5),(2,7),(2,9), (2,11), (3,4), (3,5)$. Applying Lemma 2.7 to any of these choices,
it is easy to check that except for the case of $(2,11)$ with $p=37$, the sum $a_1+a_2$ is always greater than the multiplicity of the singularity. By Lemma 2.7, the local intersection number $int_D(q_i)$ is greater than the 
multiplicity of the singularity for any local $J$-holomorphic disc $D$. But for a generic local embedded $J$-holomorphic disc $D$, $int_D(q_i)$ equals the multiplicity of the singularity, by which we arrive at a contradiction. 
On the other hand, the case of $(2,11)$ with $p=37$ can be ruled out as follows: with $p=37$ and $k=3$, one has
$g_\Sigma=18$, which implies by the genus formula that the sum of the $\delta$-invariants of the singularities of $C$ equals $10$. It follows easily that $C$ must have two singularities, each with characteristic sequence 
$(\alpha,\beta_1)=(2,11)$ and $\delta$-invariant equaling $5$, and at the third fixed point $C$ is embedded. In particular, the rotation numbers $(m_{j,1},m_{j,2})$ at two of the fixed points are both $(2,11)$, which violates Lemma 2.1, because if we start out in Lemma 2.1 with $(a,b)=(2,11)$, the other $(2,11)$ can not be realized as rotation numbers at any of the two remaining fixed points for any generator $g\in G$. It follows easily that, 
for $p\geq 37$, $C$ must have at most one cuspidal singularity with a $\delta$-invariant between $6$ and $10$. 

To proceed further, we observe that if $p\geq 47$, then $g_\Sigma\geq 23$, which implies that the $\delta$-invariant of the singularity is less than $6$ by the genus formula, unless $C$ is embedded with $g_\Sigma=28$. But in this case, $p=57$ which is not prime. Consequently, it remains to examine the cases $p=37, 41$ or $43$, where $C$ has a unique cuspidal singularity. 

Let $p=43$. In this case, the $\delta$-invariant of the singularity of $C$ equals $7$. There are several possibilities for the rotation numbers at the fixed point: $(2,15)$ or $(3,8)$ if the multiplicity of the singularity equals $2$ or $3$ which are prime. Both possibilities can be easily eliminated by Lemma 2.7 as we argued before. On the other hand, it is easy to see that the multiplicity is at most $4$, and when it is $4$, the multiplicity sequence of the singularity
is $(4,2,1,1)$. But this is not realizable, hence $p=43$ is ruled out. 

Let $p=41$.  In this case, the $\delta$-invariant equals $8$. The multiplicity of the singularity is at most $4$. It is easy to see that $3$ and $4$ can not be realized. With this understood, the rotation numbers must be $(2,17)$ with the multiplicity being $2$. But this case can be ruled out by Lemma 2.7. Hence $p=41$ is eliminated. 

Finally, let $p=37$.  In this case, the $\delta$-invariant equals $10$, and the multiplicity of the singularity is at most $5$. For the cases where the multiplicity is a prime number, the rotation numbers at the fixed point is either 
$(2, 21)$, or $(3,11)$, or $(5,6)$. If the multiplicity equals $4$, then it is easy to see that the multiplicity sequence of the singularity must be $(4,2,2,2,2, 1,1)$, from which we see that the characteristic sequence of the singularity must be $(4,6,11)$. But this is not possible because one can not have two $\beta_j$'s (i.e., $6$ and $11$) which are both
less than $p=37$. On the other hand, the case $(2,21)$ can be eliminated using Lemma 2.7. For the remaining 
$(3,11)$ and $(5,6)$, we note that $C$ is embedded near the other two fixed points. The corresponding rotation numbers $(1,m_2)$ and $(1,m_3)$ can be determined using Lemma 2.7, which must be either $(1,24)$ or $(1,17)$. One can check easily that for either $(3,11)$ or $(5,6)$, these rotation numbers are not possible by Lemma 2.1. This eliminates $p=37$, and the analysis of Case (1) is completed.

\vspace{2mm}

{\bf (2):} $p=23,29$ or $31$. First of all, we note that in Lemma 2.5, for $p=29$ or $31$, $2-2g_\Sigma-k=-p$ is still the only possibility, but for $p=23$, $2-2g_\Sigma-k=-2p$ is also possible. Moreover, in the latter case, one has
$k=4$ and $g_{|\Sigma/G|}=0$ instead. In particular, $k=4$ implies that one of the singularity of $C$ will not be
cuspidal, i.e., $C$ has two local branches at the singularity. 

Let $p=31$. Then $2-2g_\Sigma-k=-p$ with $k=3$ implies that $g_\Sigma=15$, from which it follows that the 
sum of the $\delta$-invariants of the singularities of $C$ equals $13$. With this understood, it suffices to examine the following $4$ different scenarios, (i)-(iv).

{\bf (i):} $C$ has a singularity with $\delta$-invariant less than $7$ but not equal to $4$ or $6$. It is easy to check that the following are the possible rotation numbers at the fixed point: $(2,3)$, $(2,5)$, $(2,7)$, $(2,11)$, $(3,4)$. All these possibilities can be eliminated using Lemma 2.7. 

{\bf (ii):} $C$ has a singularity with $\delta$-invariant $4$. It follows easily that, since (i) is not possible, $C$ must have another singularity with $\delta$-invariant $9$, and at the remaining fixed point, $C$ is embedded. Let us examine the singularity with $\delta$-invariant $9$ more closely. It is easy to see that the multiplicity of the singularity is at most $4$. For the case of multiplicity being $2$ or $3$, it is easy to see that the rotation numbers at the fixed point are $(2,19)$ and $(3,10)$ respectively. One can check easily that these possibilities can be ruled out by Lemma 2.7. So the multiplicity must be $4$. With this understood, it follows easily that the characteristic sequence of the singularity is either $(4,7)$ or $(4,6,9)$. The latter case can be ruled out by noting that there are two $\beta_j$'s (i.e., $6$ and $9$) which are both less than $p=31$. In the former case, the rotation numbers at the fixed point is $(4,7)$, which can be ruled out by Lemma 2.7. 

{\bf (iii):} $C$ has a singularity with $\delta$-invariant $6$. It is easy to see that $C$ must have one more singularity with $\delta$-invariant $7$. Regarding the singularity with $\delta$-invariant $6$, the possible rotation numbers at the fixed point are $(2,13)$, $(3,7)$, or $(4,5)$, where $(2,13)$, $(3,7)$ can be eliminated using Lemma 2.7. For the singularity with $\delta$-invariant $7$, we find that if the multiplicity of the singularity is $4$, then the multiplicity sequence must be $(4,2,1,1)$, which can not be realized. 
If the multiplicity is $2$ or $3$, then the rotation numbers at the corresponding fixed point must be $(2,15)$ and
$(3,8)$ respectively, where $(2,15)$ can be ruled out by Lemma 2.7. Now to see (iii) is not possible, we simply observe that $(4,5)$ and $(3,8)$ cannot be both realized as rotation numbers at two different fixed points by 
Lemma 2.1. 

{\bf (iv):} $C$ has a unique singularity with $\delta$-invariant $13$. It is easy to see that the multiplicity of the singularity is at most $5$. For the primes $2,3$ and $5$, the characteristic sequences are of the form $(\alpha,\beta_1)$. With this understood, $\alpha=5$ can not be realized as $\frac{1}{2}(\alpha-1)(\beta_1-1)=13$ has no solution $\beta_1\in\Z$. On the other hand, for $\alpha=2$ and $3$, $\beta_1=27$ and $14$ respectively. The corresponding rotation numbers at the fixed point are $(2,27)$ and $(3,14)$ respectively. If the multiplicity equals $4$, it follows easily that the multiplicity sequence must be $(4,2,2,2,2,2,2,2,1,1)$. The corresponding characteristic 
sequence is $(4,6,17)$. However, this is not possible because there are two $\beta_j$'s (i.e., $6$ and $17$) which are both less than $p=31$. To eliminate $(2,27)$ and $(3,14)$, we note that $C$ is embedded near the other two fixed points. The corresponding rotation numbers $(1,m_2)$ and $(1,m_3)$ can be determined using Lemma 2.7, which are either $(1,20)$ or $(1,14)$. One can check easily that for either $(2,27)$ or $(3,14)$, these rotation numbers are not possible by Lemma 2.1. Hence $p=31$ is eliminated. 

Let $p=29$. In this case, the sum of the $\delta$-invariants of the singularities of $C$ equals $14$. The argument is similar, in which we consider $3$ different scenarios, (i)-(iii). 

{\bf (i):} $C$ has a singularity with $\delta$-invariant less than $8$ but not equal to $4$. One can show easily that the following are the possible rotation numbers at the fixed point: $(2,3)$, $(2,5)$, $(2,7)$, $(2,11)$, $(2,13)$, $(2,15)$, $(3,4)$, $(3,7)$, $(3,8)$, $(4,5)$. All these possibilities can be ruled out by Lemma 2.7. 

{\bf (ii):} $C$ has a singularity with $\delta$-invariant $4$. In this case, $C$ must have another singularity which has $\delta$-invariant $10$. Examining the singularity of $\delta$-invariant $10$, we note that the multiplicity of the singularity is at most $5$. For the cases of prime numbers, i.e., $2$, $3$, and $5$, the corresponding rotation numbers at the fixed point are $(2,21)$, $(3,11)$, and $(5,6)$ respectively.  All these possibilities can be ruled out by Lemma 2.7. If the multiplicity equals $4$, then it is easy to see that the multiplicity sequence of the singularity must be $(4,2,2,2,2,1,1)$. Correspondingly, the characteristic sequence is $(4,6,11)$. But this can not be realized because there are two $\beta_j$'s (i.e., $6$ and $11$) which are both less than $p=29$. 

{\bf (iii):} $C$ has a unique singularity with $\delta$-invariant $14$. Again the multiplicity of the singularity is at most $5$. If the multiplicity equals $2$, then the characteristic sequence of the singularity must be $(2,29)$. But $p=29$, so this can not be realized. If the multiplicity equals $4$, then the multiplicity sequence of the singularity is either $(4,4,2,2,1,1)$ or $(4,2,2,2,2,2,2,2,2,1,1)$. Correspondingly, the characteristic sequence is either $(4,10,11)$ or $(4,6,19)$. But both can not be realized because there are two $\beta_j$'s which are both less than $p=29$. Finally, it is easy to see that with the $\delta$-invariant being $14$, the multiplicity can not be $3$, and if the multiplicity equals $5$, the rotation numbers at the fixed point must be $(5,8)$. To eliminate this case, we note that $C$ is embedded near the other two fixed points. The corresponding rotation numbers $(1,m_2)$ and $(1,m_3)$ can be determined using Lemma 2.7, which are either $(1,9)$ or $(1,13)$. One can check easily that these rotation numbers are not possible by Lemma 2.1. Hence the case of $p=29$ is eliminated. 

Finally, let $p=23$. There are two cases in Lemma 2.5: (a) $2-2g_\Sigma-k=-2p$, (b) $2-2g_\Sigma-k=-p$. In case (a), one has $2-2g_{|\Sigma/G|}-k=-2$. With $k\geq 3$, one must have $k=4$. In case (b), $k=3$ as we observed earlier. 

{\bf (a):} In this case, $g_\Sigma=22$, and as $k=4$, we may assume without loss of generality that $k_1=2$, and $k_2=k_3=1$. In particular, at the fixed points $q_2,q_3$, $C$ can have only cuspidal singularities. Furthermore, it is easy to see that the sum of the $\delta$-invariants of the singularities at $q_2,q_3$ is at most $5$. Using Lemma 2.7, one can verify easily that $C$ is singular at most one of $q_2,q_3$, and if so, the rotation numbers at the fixed point must be $(3,4)$. On the other hand, if $C$ is embedded near $q_2$ or $q_3$, the corresponding rotation numbers must be $(1,7)$ or $(1,10)$ by Lemma 2.7. Using Lemma 2.1, one can easily check that these are not possible (either $C$ is embedded near one of $q_2,q_3$, or both). Hence (a) is eliminated. 

{\bf (b):} In this case, $g_\Sigma=11$, and $C$ has only cuspidal singularities, with the sum of the 
$\delta$-invariants equaling $17$. We first consider the case where $C$ has a singularity which has a $\delta$-invariant less than $10$. The multiplicity of the singularity is easily seen to be at most $4$. Using Lemma 2.7, one can rule out the possibility of multiplicity being $2$. Furthermore, if the multiplicity equals $3$, then the rotation numbers at the fixed point must be one of the following: $(3,4)$, $(3,7)$, $(3,8)$, and $(3,10)$; the case of $(3,5)$ can be ruled out by Lemma 2.7. If the multiplicity equals $4$, then the characteristic sequence of the singularity is either $(4,5)$, $(4,7)$, or $(4,6,9)$, where the last two cases can not be realized (the case $(4,7)$ is ruled out by Lemma 2.7). It follows easily from the above analysis that for any singularity of $C$ with a $\delta$-invariant less than 
$10$, its $\delta$-invariant must be either $3$, $6$, $7$, or $9$, With these understood, it suffices to consider the following $4$ scenarios, (i)-(iv), and eliminate them one by one. 

{\bf (i):} $C$ has a singularity with $\delta$-invariant $11$. The multiplicity of the singularity is at most $5$, and we shall examine them one by one. For the case of primes, i.e., $2,3$ and $5$, the characteristic sequence is of the form $(\alpha,\beta_1)$, where $\frac{1}{2}(\alpha-1)(\beta_1-1)=11$. For $\alpha=2$, $\beta_1=23=p$, which is not possible because $(\alpha,\beta_1)$ is also the rotation numbers at the fixed point. If $\alpha=3$, 
$\beta_1=12$, which contradicts the fact that $gcd(\alpha,\beta_1)=1$. If $\alpha=5$, there is no solution 
$\beta_1\in\Z$ to $\frac{1}{2}(\alpha-1)(\beta_1-1)=11$, hence this is also not possible. Finally, we assume the multiplicity equals $4$. Then it follows easily that the characteristic sequence must be $(4,6,13)$, which is not possible because one can not have two $\beta_j$'s (i.e., $6$ and $13$) which are both less than $p=23$. 

{\bf (ii):} $C$ has a singularity with $\delta$-invariant $3$. We note that the rotation numbers at the fixed point is $(3,4)$. With this understood, we observe that there are two possibilities: either $C$ is singular at the remaining fixed points, each of $\delta$-invariant $7$, or $C$ is singular at one of the remaining fixed points with $\delta$-invariant $14$, and is embedded near the other remaining fixed point. The latter case can be ruled out as we did in case (a), while in the former case, the rotation numbers at the $3$ fixed points are $(3,4)$, $(3,8)$, $(3,8)$ (each with respect to some generator in $G$). One can eliminate this possibility by Lemma 2.1. 

{\bf (iii):} $C$ has a singularity of $\delta$-invariant $7$. The rotation numbers at the fixed point is $(3,8)$. We only need to examine the case where $C$ is singular at one of the remaining fixed points, with a $\delta$-invariant $10$, and is embedded near the other remaining fixed point, at which the rotation numbers is either $(1,7)$ or $(1,10)$ by Lemma 2.7. Using Lemma 2.1, one can easily show that if $(3,8)$ is the rotation numbers at one of the fixed points, $(1,7)$ or $(1,10)$ can never be realized as rotation numbers at any other fixed point. 

{\bf (iv):} $C$ has a unique singularity with $\delta$-invariant $17$. Note that $C$ is embedded near
the remaining two fixed points, so the corresponding rotation numbers must be either $(1,7)$ or $(1,10)$ 
by Lemma 2.7. This case can be ruled out as we did in case (a). This eliminates the case of $p=23$, and the analysis of Case (2) is completed.

\vspace{2mm}

{\bf (3):} $p=17$. We begin with some quick but useful observations: if $C$ is embedded near a fixed point, then by Lemma 2.7 the rotation numbers at the fixed point must be either $(1,5)$ or $(1,7)$. On the other hand, by Lemma 2.1, it is easy to see that the rotation numbers $(1,5)$ or $(1,7)$ can not be realized at two different fixed points
(even with respect to different generators of $G$). As a consequence, $C$ cannot be embedded near more than one fixed points. Secondly, it also follows from Lemma 2.7 that $C$ cannot have a cuspidal singularity with $\delta$-invariant $1$ or $2$ at any fixed point, and if the $\delta$-invariant equals $3$, the rotation numbers at the fixed point must be $(3,4)$. With this understood, we note that there are $3$ cases to consider: 
(a) $2-2g_\Sigma-k=-3p$, (b) $2-2g_\Sigma-k=-2p$, and (c) $2-2g_\Sigma-k=-p$. See Lemma 2.5.

{\bf (a):} $2-2g_\Sigma-k=-3p$. In this case, $k=3$ or $5$. If $k=3$, then $g_\Sigma=25$, and $C$ has only cuspidal singularities at the fixed points, with the sum of the $\delta$-invariants equaling $3$. By the previous observations, it follows easily that $C$ has to be embedded near two of the $3$ fixed points, which is not possible. If $k=5$, then $g_\Sigma=24$. It follows from the genus formula for $g_\Sigma$ and the fact that $C$ can not be embedded near two different fixed points that $C$ must have two local branches at two of the fixed points and is embedded near the third fixed point, e.g.  $k_1=k_2=2$ and $k_3=1$ without loss of generality. Furthermore, if a local branch at one of the fixed points, say $q_1$, is singular, then by the genus formula for $g_\Sigma$ the only possibility is that the local branch has a $\delta$-invariant $1$ cuspidal singularity while the other local branch is embedded with their tangent planes being transversal at $q_1$, and $C$ must be embedded at $q_3$. As a consequence, the rotation numbers at $q_1$ associated to the two local branches must be $(2,3)$ and $(1,12)$ respectively. With the rotation numbers at $q_3$ (near where $C$ is embedded) being $(1,5)$ or $(1,7)$, one can easily check that Lemma 2.1 is violated. 

Hence both local branches at both $q_1,q_2$ must be embedded. With this understood, note that the genus formula for $g_\Sigma$ implies that the two local branches at one of $q_1,q_2$, say $q_1$, must be tangent to each other at the fixed point. As a consequence, the two local branches must determine the same rotation numbers, of the form $(1,m)$, where $m$ is the local intersection number of the two branches. It then follows from the genus formula for $g_\Sigma$ that $m<4$ must be true. To eliminate this possibility, we shall use Lemma 2.7. To see this, observe that $int_D(q_1)=2$ for a generic $D$. On the other hand, Lemma 2.7 implies that
$$
(2+a_1)+(2+a_2)m=0 \pmod{17}
$$
for some non-negative integers $a_1,a_2$ such that $0<a_1+a_2\leq 2$. It follows easily that $m=5,7,8,15$ or $16$, which contradicts $m<4$. Hence (a) is eliminated. 

{\bf (b):} $2-2g_\Sigma-k=-2p$. In this case, $k=4$, $g_{|\Sigma/G|}=0$, and $g_\Sigma=16$. With $k=4$, we assume without loss of generality that $k_1=2$, and $k_2=k_3=1$. Let $C_{1,1}, C_{2,1}$ be the two local branches of $C$ at the fixed point $q_1$, let $\delta_{1,1}, \delta_{2,1}$ be the corresponding $\delta$-invariants of the singularities contained in $C_{1,1}, C_{2,1}$, and let $\delta_2,\delta_3$ be the $\delta$-invariants of the singularities of $C$ at the fixed points $q_2,q_3$. Then it follows easily from the genus formula for $g_\Sigma=16$ that
$$
C_{1,1}\cdot C_{2,1}+\delta_{1,1}+\delta_{2,1}+\delta_2+\delta_3=28-16=12.
$$
In particular, $\delta_{1,1},\delta_{2,1},\delta_2,\delta_3\leq 11$. With this understood, in Table 1 below we list the possible $\delta$-invariants and rotation numbers at $q_2,q_3$ 
(after applying Lemma 2.7); in particular, note that $\delta$-invariant $\neq 1,2,4,5$ and $8$.

\vspace{3mm}

\centerline{
\begin{tabular}{lllllll}
\hline \vline\hspace{1mm}
$\delta$-invariant $\leq 11$ & \vline \hspace{1mm} rotation numbers  \hspace{5.8mm} \vline \\
\hline \vline\hspace{1mm}
0 &  \vline \hspace{1mm} $(1,5)$ or $(1,7)$ \hspace{11.6mm} \vline \\ 
\hline \vline\hspace{1mm} 
3 &  \vline \hspace{1mm} $(3,4)$  \hspace{26.1mm} \vline\\
\hline \vline\hspace{1mm} 
6 &  \vline \hspace{1mm} $(2,13)$, $(3,7)$ or $(4,5)$ \vline\\
\hline \vline\hspace{1mm} 
7 &  \vline \hspace{1mm} $(2,15)$  \hspace{24.2mm} \vline\\
\hline \vline\hspace{1mm} 
9 &  \vline \hspace{1mm} $(4,7)$  \hspace{26.1mm} \vline\\
\hline \vline\hspace{1mm} 
10 &  \vline \hspace{1mm} $(5,6)$  \hspace{26.1mm} \vline\\ \hline
\end{tabular}
}
\vspace{2mm}

\centerline{Table 1: $\delta$-invariants and rotation numbers at $q_2,q_3$, in Case (3)(b) with $p=17$.
}

\vspace{2mm}

With the preceding understood, there are $4$ scenarios, (i)-(iv), to consider. 

{\bf (i):} Both $\delta_{1,1},\delta_{2,1}$ are nonzero. Let $(a_1,b_1)$, $(a_2,b_2)$ be the rotation numbers associated to the local branches $C_{1,1},C_{2,1}$, where we may assume that they coincide with the characteristic sequences of the singularities. There are two possibilities: if $C_{1,1},C_{2,1}$ have the same tangent plane at $q_1$, then $(a_1,b_1)=(a_2,b_2)$ must be true, and in this case, $C_{1,1}\cdot C_{2,1}$ equals the sum of the square of all the entries in the multiplicity sequence, and if $C_{1,1},C_{2,1}$ have transverse tangent planes at 
$q_1$, then $C_{1,1}\cdot C_{2,1}=a_1a_2$ (i.e., the product of the multiplicities). With this understood, we observe
that $C_{1,1}\cdot C_{2,1}+\delta_{1,1}+\delta_{2,1}\leq 9$, as $C$ cannot be embedded at both $q_2,q_3$. It follows immediately that the multiplicity of both singularities must be $2$ (i.e., $a_1=a_2=2$). Assume one of 
$\delta_{1,1},\delta_{2,1}$ equals $1$. Then if $C_{1,1},C_{2,1}$ have the same tangent plane, we must have both 
$\delta_{1,1}=\delta_{2,1}=1$, and $C_{1,1}\cdot C_{2,1}=2^2+1^2+1^2=6$. If $C_{1,1},C_{2,1}$ have 
transverse tangent planes, then it is easy to see that we must have $(a_1,b_1)=(2,3)$, $(a_2,b_2)=(2,7)$. It follows 
that in both cases, $\delta_2+\delta_3=12-8=4$, which is easily seen not possible by Table 1. If both $\delta_{1,1},\delta_{2,1}\geq 2$, then the only possibility is that $\delta_{1,1}=2,\delta_{2,1}=3$, and $C_{1,1},C_{2,1}$ have 
transverse tangent planes, because $C_{1,1}\cdot C_{2,1}\neq 5$ and $\delta_2+\delta_3\neq 4$. With this understood, the rotation numbers $(a_1,b_1)$, $(a_2,b_2)$ associated to the two local branches $C_{1,1},C_{2,1}$ must be $(2,5)$ and $(2,7)$ respectively, which contradicts the fact that $C_{1,1},C_{2,1}$ have transverse tangent planes. Hence (i) is eliminated.

{\bf (ii):} $\delta_{1,1}=\delta_{2,1}=0$ and $C_{1,1}\cdot C_{2,1}=1$. In this case, we have $\delta_2+\delta_3=11$,
which is easily seen not possible by Table 1.

{\bf (iii):} $\delta_{1,1}=\delta_{2,1}=0$ and $m=C_{1,1}\cdot C_{2,1}>1$. In this case, the rotation numbers $(1,m)$ at $q_1$ is subject to Lemma 2.7, which implies that $m=5,7,8,15$, or $16$, and on the other hand, 
$\delta_2+\delta_3=12-m$. It follows easily that $m=5$ is the only possibility, and according to Table 1, the rotation numbers at $q_2,q_3$ must be $(1,5)$ (or $(1,7)$) and $(2,15)$ respectively. But this violates Lemma 2.1, hence
(iii) is eliminated.

{\bf (iv):} $\delta_{1,1}>0$ but $\delta_{2,1}=0$. Note that $\delta_2+\delta_3\geq 3$, which implies by the genus formula that 
$\delta_{1,1}\leq 7$. Consequently, the rotation numbers at $q_1$ that is associated to the local branch $C_{1,1}$
must be of the form $(\alpha,\beta_1)$ where $(\alpha,\beta_1)$ is the characteristic sequence of the singularity, 
and the multiplicity $\alpha=2,3$ or $4$, with $C_{1,1}\cdot C_{2,1}=\beta_1$ if $C_{1,1}, C_{2,1}$ are tangent at $q_1$ and $C_{1,1}\cdot C_{2,1}=\alpha$ if otherwise. Moreover, $(\alpha,\beta_1)$ is subject to a constraint 
from Lemma 2.7, i.e., if for any non-negative integers $a_1,a_2$ with $a_1+a_2>0$ and 
$$
(2+a_1)\alpha +(2+a_2)\beta_1=0 \pmod{17},
$$
one always has $a_1+a_2>\alpha+1=int_D(q_1)$ (here $D$ is generic), then $(\alpha,\beta_1)$ can be eliminated. Finally, we point out that
the rotation numbers at $q_1$ that is associated to the local branch $C_{2,1}$ is of the form $(1,m)$ which can be
determined completely from $(\alpha,\beta_1)$ and the fact of whether $C_{1,1}, C_{2,1}$ are tangent at $q_1$ or not. With this understood, we obtain the following list of possible rotation numbers at the fixed points, where in each case the first two rotation numbers are associated with $C_{1,1}, C_{2,1}$ and the last two rotation numbers are at $q_2$, $q_3$ respectively:
\begin{itemize}
\item [{(1)}] $(2,3)$, $(1,12)$, $(1,5)$ (or $(1,7)$), $(4,7)$.
\item [{(2)}] $(2,3)$, $(1,12)$, $(3,4)$, $(2,13)$ (or $(3,7)$), $(4,5)$).
\item [{(3)}] $(2,7)$, $(1,10)$, $(1,5)$ (or $(1,7)$), $(2,15)$.
\item [{(4)}] $(3,4)$, $(1,5)$, $(1,5)$ (or $(1,7)$), $(2,13)$ (or $(3,7)$), $(4,5)$).
\item [{(5)}] $(3,4)$, $(1,5)$, $(3,4)$, $(3,4)$.
\item [{(6)}] $(3,5)$, $(1,13)$, $(1,5)$ (or $(1,7)$), $(3,4)$.
\item [{(7)}] $(3,7)$, $(1,15)$, $(1,5)$ (or $(1,7)$), $(3,4)$.
\item [{(8)}] $(2,15)$, $(1,16)$, $(1,5)$ (or $(1,7)$), $(3,4)$.
\end{itemize}
All these cases can be eliminated by Lemma 2.1. 

{\bf (c):} $2-2g_\Sigma-k=-p$. In this case, $k=3$, $g_{|\Sigma/G|}=0$, and $g_\Sigma=8$. Moreover, either $C$ has $17$ nodal singularities or $17$ simple cusp singularities in the complement of the fixed points and at most cuspidal singularities at each fixed point $q_i$ with the sum of the $\delta$-invariants equaling $3$, or $C$ has only cuspidal singularities with the sum of the $\delta$-invariants equaling $20$. The former case is actually not possible because $C$ will be embedded near two different fixed points. The latter case can be ruled out by the fact that $d_C<0$. To see this, recall that with $k=3$, $d=6$,
$$
d_C=-\frac{18}{p}+(g_{|\Sigma/G|}-1)+3-\sum_{j=1}^3 \frac{m_{j,1}+m_{j,2}}{p}.
$$
Now with $g_{|\Sigma/G|}=0$, $d_C\geq 0$ is equivalent to $\sum_{j=1}^3 (m_{j,1}+m_{j,2})\leq 16$. With this 
understood, $d_C<0$ follows easily from the following observations: (i) there is at least one 
singularity of $C$ with $\delta$-invariant less than $11$, as $C$ cannot be embedded near more than one fixed points, and for such a singularity, one can check that the corresponding rotation numbers 
$(m_{j,1},m_{j,2})$ always obeys $m_{j,1}+m_{j,2}\geq 7$, (ii) if $(m_{j,1},m_{j,2})$ is at a fixed point near where $C$ is embedded, then $m_{j,1}+m_{j,2}\geq 6$, and finally, (iii) by Lemma 2.1, $m_{j,1}+m_{j,2}\geq 4$ for any $j$. This completes the analysis of Case (3).

\end{proof}

It follows from Lemmas 2.6 and 2.8 that the possible orders of the group $G=\Z_p$ are 
$p=3, 5, 7, 11, 13$, or $19$.

\subsection{Classification of the invariant curves}
In this subsection, we shall give a classification of all the possible invariant curves, along the way
determining the corresponding local representations of the group action at the fixed points, thus completing the proof of Theorem 1.3. Recall that by Lemma 2.3, if we choose $J\in\J$ to be generic, then each $J$-holomorphic curve $C\in\{C_i\}$, where $c_1(2K)=\sum_i n_i C_i$, must be invariant under the $G$-action. 

\vspace{2mm}

{\bf (1):} $p=3$. First, by Lemmas 2.2 and 2.6, $c_1(2K)=C$, where $C$ is a degree $6$ $J$-holomorphic curve containing all $3$ fixed points. With this understood, we look at the formula for $d_C$:
$$
d_C=-6 + (g_{|\Sigma/G|}-1)+ k-\sum_{j=1}^k \frac{m_{j,1}+m_{j,2}}{3}.
$$
Observing that for $p=3$, $m_{j,1}+m_{j,2}=3$ for any $j$, we see immediately that $d_C\geq 0$ is equivalent to
$g_{|\Sigma/G|}\geq 7$. Now the identity $2-2g_\Sigma-k=3(2-2g_{|\Sigma/G|}-k)$ implies that 
$g_\Sigma\geq k+ 19$, where $k\geq 3$ as $C$ contains all $3$ fixed points. With this understood, we note that by Lemma 2.7, $C$ cannot be embedded near any of the fixed points. On the other hand, a simple examination shows that any local branch of $C$ at a fixed point is either embedded, or has a cuspidal singularity with $\delta$-invariant at least $3$ (for $p=3$, $\delta$-invariant cannot be $1$ or $2$). It follows easily that $k=6$, $g_\Sigma=25$ and $g_{|\Sigma/G|}=7$ is the only possibility, where the adjunction inequality is not violated, with $C$ having a nodal singularity at each fixed point. 
Moreover, one can check easily that Lemma 2.7 and the Global Quotient Criterion are not violated. This finishes the case of $p=3$.

\vspace{2mm}

{\bf (2):} $p=5$. Again, by Lemmas 2.2 and 2.6, $c_1(2K)=C$, where $C$ is a degree $6$ $J$-holomorphic curve containing all $3$ fixed points, and by Lemma 2.7, $C$ can not be embedded near any of the fixed points. 
On the other hand, it is useful to observe (which can be easily verified) that if a singular local branch of $C$ at a fixed point has a $\delta$-invariant not greater than $8$, the associated rotation numbers (together with the $\delta$-invariant) at the fixed point must fall into one of the following cases; in particular, the $\delta$-invariant cannot be $2,7$ or $8$:

\vspace{2mm}

\centerline{
\begin{tabular}{lllllll}
\hline \vline\hspace{1mm}
$\delta$-invariant $\leq 8$ & \vline \hspace{1mm} rotation numbers \vline \\
\hline \vline\hspace{1mm}
1 &  \vline \hspace{1mm} $(2,3)$  \hspace{19mm} \vline \\ 
\hline \vline\hspace{1mm} 
3 &  \vline \hspace{1mm} $(3,4)$  \hspace{19mm}  \vline\\
\hline \vline\hspace{1mm} 
4 &  \vline \hspace{1mm} $(2,4)$  \hspace{19mm} \vline\\
\hline \vline\hspace{1mm} 
5 &  \vline \hspace{1mm} $(1,2)$  \hspace{19mm} \vline\\
\hline \vline\hspace{1mm} 
6 &  \vline \hspace{1mm} $(2,3)$  \hspace{19mm} \vline\\ \hline
\end{tabular}
}
\vspace{2mm}

\centerline{Table 2: possible $\delta$-invariants and rotation numbers at a fixed point, $p=5$.
}

\vspace{2mm}

With the preceding understood, our discussions will be divided into the following $4$ cases: (a)
 $2-2g_\Sigma-k=-11p$, (b)  $2-2g_\Sigma-k=-10p$, (c)  $2-2g_\Sigma-k=-9p$, and (d)  $2-2g_\Sigma-k=-xp$
 where $x\leq 8$. See Lemma 2.5.

{\bf (a):} $2-2g_\Sigma-k=-11p$. In this case, $k$ is odd and $3\leq k\leq 13$. The following adjunction inequality
$$
g_\Sigma\leq \frac{1}{2}(d^2+3d+2)-\frac{1}{2}\sum_{i=1}^3 k_i(k_i-1)-\sum_{i=1}^3 \delta_i,
$$
with $d=6$ and $2-2g_\Sigma-k=-11p$, is equivalent to 
$$
\sum_{i=1}^3 k_i=k\geq 1+\sum_{i=1}^3(k_i(k_i-1)+2\delta_i),
$$ 
where $\delta_i=0$ if $k_i>1$ and $\delta_i$ equals the $\delta$-invariant if $k_i=1$ (including the case where $C$ is embedded near the fixed point, in which case $\delta_i=0$). It follows easily that $C$ must be embedded near one of the fixed points, i.e., $\delta_i=0$ and $k_i=1$ for some $i$, which is not possible by Lemma 2.7. Hence (a) is eliminated. 

{\bf (b):} $2-2g_\Sigma-k=-10p$. In this case, $k$ is even and $4\leq k\leq 12$. One can check easily
that if $k\geq 8$,  the adjunction inequality is violated by a similar argument as in (a). If $k=4$, $C$ must have cuspidal singularities at two of the fixed points. The $\delta$-invariants of the singularities must be $1$ because $g_\Sigma=24$ when $k=4$. This implies that the rotation numbers at both fixed points are $(2,3)$ according to Table 2, which contradicts Lemma 2.1. 

Finally, let $k=6$, in which case $g_\Sigma=23$ and $g_{|\Sigma/G|}=3$. It follows easily from the adjunction inequality that $k_1=k_2=k_3=2$ must be true. Moreover, $C$ cannot have $3$ nodal singularities at the $3$ fixed points, and $C$ must have at least one nodal singularity at a fixed point. 

With this understood, let's rule out the possibility that $C$ has $2$ nodal singularities. Assume to the contrary that $C$ has $2$ nodal singularities, say at $q_1,q_2$. Then at $q_3$ there are two possibilities: either one of the local branch of $C$ is singular with a $\delta$-invariant $1$ and the other local branch being embedded, or both local branches are embedded, and are tangent with a tangency of order $3$. In the former case, the rotation numbers at $q_3$ are $(2,3)$ and $(1,4)$ associated to the two local branches, and in the latter case, they are both $(1,3)$. Using Lemma 2.1, one can then determine the rotation numbers at $q_1$ and $q_2$, and an easy calculation gives $d_C=-\frac{2}{5}$ and $d_C=-\frac{3}{5}$ respectively. This is a contradiction, hence the claim.

The only remaining possibility is that $C$ has one nodal singularity at a fixed point and at the remaining two fixed points, both local branches are embedded and are tangent with a tangency of order $2$. The corresponding rotation numbers $(m_{j,1},m_{j,2})$ must be both $(1,4)$ at
the fixed point with a nodal singularity, and must be $(1,2)$ at the remaining two fixed points, which one can check easily that Lemmas 2.1 and 2.7 are not violated. Moreover, with $g_{|\Sigma/G|}=3$ and $k=6$, one can check easily that 
$$
d_C=-\frac{18}{p}+(g_{|\Sigma/G|}-1)+k-\sum_{j=1}^k \frac{m_{j,1}+m_{j,2}}{p}=0,
$$
hence there is no contradiction so far.

We examine this possibility against the Global Quotient Criterion; in fact, its verification also allows us to determine the local representations of the $\Z_p$-action at the $3$ fixed points. More precisely, recall that the rotation numbers associated to the two local branches at the nodal singularity are 
both $(1,4)$. Let $g_1,g_2\in\Z_p$ be the corresponding generators, and let $g_3=g_4\in\Z_p$ and $g_5=g_6\in\Z_p$ be the corresponding generators at the other two fixed points. Now apply Lemma 2.1, with $g=g_1$ and $(a,b)=(1,4)$, and assume the rotation numbers at the other two fixed points are given by case (i), i.e., $(4,3)$ and $(1,2)$ respectively, and we write $g_j=g_1^{c_j}$ for $j=1,2,\cdots,6$. Since the rotation numbers $(m_{j,1},m_{j,2})$ at the other two fixed points are all $(1,2)$, it is easy to see that $c_1=1$, $c_2=4$, and $c_3=c_4=4$, $c_5=c_6=1$, which obeys the constraint $c_1+c_2+\cdots+c_6=0 \pmod{5}$ in  the Global Quotient Criterion.
Furthermore, if the rotation numbers are given by (ii) or (iii) in Lemma 2.1, the Global Quotient Criterion is actually violated. Finally, observe that (i) and (iv) in Lemma 2.1 are equivalent up to a change of the order of the two fixed points. This is the case stated in Theorem 1.3 for $p=5$. 

{\bf (c):} $2-2g_\Sigma-k=-9p$. Here $k$ must be odd and $3\leq k\leq 11$. First, let $k=3$. In this case, $C$ has $3$ cuspidal singularities at the $3$ fixed points, with the sum of $\delta$-invariants equaling $6$ as $g_\Sigma=22$. It follows easily from Table 2 that the only possible distribution of the $\delta$-invariants 
at the $3$ fixed points is $(1,1,4)$. Correspondingly, the rotation numbers at the fixed points are $(2,3)$, $(2,3)$, and $(2,4)$, which contradicts Lemma 2.1. Secondly, it is easy to see that if $k\geq 7$, the
adjunction inequality is violated. Hence it remains to check the case of $k=5$, where $g_\Sigma=21$. The following identity follows from the genus formula 
$$
\frac{1}{2}\sum_{k_i>1}\sum_{j\neq j^\prime}C_{j,i}\cdot C_{j^\prime,i}+\sum_{k_i>1}\sum_{j=1}^{k_i} \delta_{j,i}+\sum_{i=1}^3 \delta_i=7.
$$
Without loss of generality, there are $2$ scenarios, (i)-(ii), to examine. 

{\bf (i):} $k_1=3$, $k_2=k_3=1$. In this case, note that $C$ has a cuspidal singularity at
$q_2$ and $q_3$. On the other hand, observe that $\frac{1}{2}\sum_{j\neq j^\prime}C_{j,1}\cdot C_{j^\prime,1}\geq 4$,
which implies that the $\delta$-invariants of the singularities at $q_2,q_3$ must be both $1$. Correspondingly, the rotation numbers at $q_2,q_3$ are both $(2,3)$ by Table 2, which contradicts Lemma 2.1. 

{\bf (ii):} $k_1=k_2=2$, $k_3=1$. We note that $g_{|\Sigma/G|}=3$, so that $d_C\geq 0$ is equivalent to 
$\sum_{j=1}^5 (m_{j,1}+m_{j,2})\leq 17$ in the formula of $d_C$. On the other hand, $m_{j,1}+m_{j,2}\geq 3$
for each $j$. It follows that if we let $(m_{5,1},m_{5,2})$ be the rotation numbers at $q_3$, then $d_C<0$ whenever
$m_{5,1}+m_{5,2}>5$, and if $m_{5,1}+m_{5,2}=5$, then $d_C\geq 0$ forces $(m_{j,1},m_{j,2})=(1,2)$ for each
$j=1,2,3,4$. This in turn implies that the local branches of $C$ are embedded at $q_1,q_2$, and moreover, they
are tangent with a tangency of order $2$. As a consequence, the $\delta$-invariant of the singularity at $q_3$
must be $3$. But this contradicts the assumption that $m_{5,1}+m_{5,2}=5$, as $(m_{5,1},m_{5,2})=(3,4)$ if
the singularity at $q_3$ has $\delta$-invariant $3$ (see Table 2). Hence $m_{5,1}+m_{5,2}<5$, and the only possibility is $(m_{5,1},m_{5,2})=(1,2)$ with the $\delta$-invariant of the singularity at $q_3$ being $5$. Then the genus formula implies that $C$ has a nodal singularity at both $q_1,q_2$. With $(m_{5,1},m_{5,2})=(1,2)$, 
Lemma 2.1 implies that the rotation numbers at $q_1,q_2$ are $(1,2)$, $(1,3)$, $(1,4)$, $(1,4)$. It follows easily that $d_C<0$, so this case is also eliminated. 

{\bf (d):} $2-2g_\Sigma-k=-xp$ where $x\leq 8$. We observe that $x=2g_{|\Sigma/G|}-2+k$. As in the formula of $d_C$, $m_{j,1}+m_{j,2}\geq 3$ for all $j=1,2,\cdots,k$, it follows easily that $d_C\leq \frac{1}{10}(5x-36-k)$. Consequently, $d_C<0$ if $x<8$, and if $x=8$, $d_C\geq 0$ is equivalent to $k=4$ and $m_{j,1}+m_{j,2}=3$ for all 
$j=1,2,\cdots,k$. To see this is not possible, we note that if $x=8$ and $k=4$, we have $g_\Sigma=19$. It follows easily that $C$ has two cuspidal singularities at two of the fixed points, with the sum of the $\delta$-invariants 
being at most $8$. But with the rotation numbers at both fixed points being $(1,2)$, the $\delta$-invariant of each singularity must be $5$ by Table 2, which is a contradiction. 

This finishes the case of $p=5$.

\vspace{2mm}

{\bf (3):} $p=7$. There are two cases. First, assume that there is a $C\in \{C_i\}$, where $c_1(2K)=\sum_i n_i C_i$
from Lemma 2.2, such that $C$ is of degree $d\leq 3$. 

Suppose $d=1$. Then $d^2+3d+3=7=p$, which implies that $2-2g_\Sigma-k=-p$ is the only possibility, with $k=3$, $g_{|\Sigma/G|}=0$, and $g_\Sigma=3$. In particular, $C$ is embedded. On the other hand, $d_C\geq 0$ and $d_C\in\Z$ imply that $\sum_{j=1}^3 (m_{j,1}+m_{j,2})=11$ in the formula of $d_C$. It is clear that one of the rotation numbers $(m_{j,1},m_{j,2})$ has to be $(1,2)$. Then by Lemma 2.1, there must be another one which is $(1,6)$, which leads to a contradiction to $\sum_{j=1}^3 (m_{j,1}+m_{j,2})=11$. Hence $d=1$ is eliminated. 

Next, assume $d=2$. We note that $d^2+3d+3=13<2p$, so $2-2g_\Sigma-k=-p$ is also the only possibility. It is easy to check that $d_C<0$, so $d=2$ is not possible. 

Finally, let $d=3$. The first possibility in Lemma 2.5 is $2-2g_\Sigma-k=-3p$. It is clear that in this case, either $k=3$ 
or $k=5$; in fact it is easy to see that the adjunction inequality is violated when $k=5$. On the other hand, if $k=3$, then $g_{|\Sigma/G|}=1$ and $g_\Sigma=10$, in particular, $C$ is embedded. Furthermore, $d_C\geq 0$ and $d_C\in\Z$ imply that the rotation numbers in the formula of $d_C$ satisfy the equation 
$\sum_{j=1}^3 (m_{j,1}+m_{j,2})=12$. If one of the $(m_{j,1},m_{j,2})$'s is $(1,2)$, then by Lemma 2.1, there
must be another one which is $(1,6)$, which contradicts $\sum_{j=1}^3 (m_{j,1}+m_{j,2})=12$. With this understood,
it is easy to see that the only possibility for $\sum_{j=1}^3 (m_{j,1}+m_{j,2})=12$ to be true is 
$(m_{j,1},m_{j,2})=(1,3)$ for $j=1,2,3$. One can check easily that this is in line with Lemma 2.1. Furthermore, the Global Quotient Criterion is not violated if and only if the local representations of the $\Z_p$-action are in case
(i) of Lemma 2.1. More concretely, if we let $g_1,g_2,g_3$ be the corresponding generators at the $3$ fixed points,
each with rotation numbers $(1,3)$, and we let $g_j=g_1^{c_j}$ for $j=1,2,3$. Then if we apply Lemma 2.1(i)
to $g_1$ with $(a,b)=(1,3)$, the rotation numbers with respect to $g_1$ at the other two fixed points are $(6,2)$,
$(4,5)$. It follows easily that $c_2=4$, $c_3=2$, so that $c_1+c_2+c_3=7=0 \pmod{7}$ is satisfied. 
This case is stated as (3)(i) in Theorem 1.3.

To eliminate the remaining possibilities, we set $x:=2g_{|\Sigma/G|}-2+k$, where $x\leq 2$. Appealing to the fact that in the formula of $d_C$, $m_{j,1}+m_{j,2}\geq 3$ for each $j$, together with $k\leq x+2$, we see that 
$d_C\leq \frac{1}{14} (8x-16)$ holds true. It follows immediately that $d_C<0$ if $x<2$, and if $x=2$, $d_C\geq 0$
implies that $k=x+2=4$ and $(m_{j,1}, m_{j,2})=(1,2)$ for each $j$. But $(1,2)$ can be realized as rotation numbers
(for some generator of $G$) at each fixed point is in contradiction to Lemma 2.1. This finishes the first case
$d\leq 3$.

Secondly, assume that there is a $C\in\{C_i\}$ of degree $d=6$, i.e., $c_1(2K)=C$. We begin with the observation from Lemma 2.7 that if $C$ is embedded near a fixed point, then the rotation numbers at the fixed point has to be either $(1,2)$ or $(1,4)$. On the other hand, a quick examination yields the following list of rotation numbers at a cuspidal singularity with $\delta$-invariants less than $8$ (none of them violating Lemma 2.7):

\vspace{3mm}

\centerline{
\begin{tabular}{lllllll}
\hline \vline\hspace{1mm}
$\delta$-invariant $<8$ & \vline \hspace{1mm} rotation numbers  \hspace{5.8mm} \vline \\
\hline \vline\hspace{1mm}
0 &  \vline \hspace{1mm} $(1,2)$ or $(1,4)$ \hspace{11.6mm} \vline \\ 
\hline \vline\hspace{1mm}
1 &  \vline \hspace{1mm} $(2,3)$  \hspace{26.1mm} \vline \\ 
\hline \vline\hspace{1mm} 
2 &  \vline \hspace{1mm} $(2,5)$  \hspace{26.1mm} \vline\\
\hline \vline\hspace{1mm} 
3 &  \vline \hspace{1mm} $(3,4)$  \hspace{26.1mm} \vline\\
\hline \vline\hspace{1mm} 
4 &  \vline \hspace{1mm} $(3,5)$  \hspace{26.1mm} \vline\\
\hline \vline\hspace{1mm} 
5 &  \vline \hspace{1mm} $(2,4)$  \hspace{26.1mm} \vline\\
\hline \vline\hspace{1mm} 
6 &  \vline \hspace{1mm} $(2,6)$ or $(4,5)$ \hspace{11.6mm} \vline\\ 
\hline \vline\hspace{1mm}
7 &  \vline \hspace{1mm} $(1,2)$ or $(1,3)$ \hspace{11.6mm} \vline \\ \hline
\end{tabular}
}
\vspace{2mm}

\centerline{Table 3: $\delta$-invariants and rotation numbers at a cuspidal fixed point, $p=7$, $d=6$.
}

\vspace{2mm}

With the preceding understood, there are $4$ cases to examine: (a) $2-2g_\Sigma-k=-8p$, (b) $2-2g_\Sigma-k=-7p$, (c) $2-2g_\Sigma-k=-6p$, and (d) $2-2g_\Sigma-k=-xp$ where $x\leq 5$, see Lemma 2.5.

{\bf (a):} $2-2g_\Sigma-k=-8p$. In this case $k$ is even and $4\leq k\leq 10$. Suppose $k=4$, in which case $g_\Sigma=27$ and $g_{|\Sigma/G|}=3$. It follows easily from the genus formula that
$C$ must be immersed with exactly one nodal singularity at a fixed point. It follows that the rotation numbers at the other two fixed points have to be either $(1,2)$ or $(1,4)$ (cf. Table 3). With this understood, we apply Lemma 2.1 and find that the rotation numbers at the nodal singularity has to be $(1,6)$ for both local branches. Then $d_C\in\Z$ implies that the rotation numbers at the other two fixed points must be both $(1,4)$. One can check easily that these rotation numbers are consistent with Lemma 2.1, and Lemma 2.7 is not violated at the nodal singularity. Finally, the Global Quotient Criterion is not violated if and only if the local representations of the $\Z_p$-action are given by case (i) (or equivalently case (iv) after switching the two fixed points) in Lemma 2.1. More concretely, if we let $g_1,g_2$ be the corresponding generators at the nodal singularity, and $g_3,g_4$ the generators at the remaining two fixed points, with $g_j=g_1^{c_j}$ for $j=1,2,3,4$, and apply Lemma 2.1(i) to $g_1$ with $(a,b)=(1,6)$, then the rotation numbers at the other two
fixed points are $(6,5)$ and $(1,2)$ (with respect to $g_1$), which implies easily that $c_2=6$, $c_3=3$ and $c_4=4$. Thus the constraint 
$c_1+c_2+c_3+c_4=14=0\pmod{7}$ in the Global Quotient Criterion is satisfied. This case is stated as (3)(ii) in Theorem 1.3. Finally, if $k\geq 6$, one can check easily that the adjunction inequality is violated so it's not possible. 

{\bf (b):} $2-2g_\Sigma-k=-7p$. In this case, $k$ is odd and $3\leq k\leq 9$. Let $k=3$ first. In this case
we have $g_\Sigma=24$, and $C$ has only cuspidal singularities at the fixed points with total $\delta$-invariants $4$. The following is easy to see: if $C$ is embedded near one of the fixed points, then the distribution of $\delta$-invariants at the other two fixed points are either $2$, $2$, or $1$, $3$, or $0$, $4$, and if $C$ is singular at all $3$ fixed points, the distribution of $\delta$-invariants must be $1$, $1$, $2$. Using Table 3, one can check easily that 
in any of these cases, the corresponding rotation numbers are in contradiction with Lemma 2.1. Hence the case
$k=3$ is eliminated. 

Next, let $k=5$. Then $g_\Sigma=23$ and $g_{|\Sigma/G|}=2$. In this case, $d_C\geq 0$ and $d_C\in \Z$ imply that the rotation numbers in the formula of $d_C$ satisfy $\sum_{j=1}^5 (m_{j,1}+m_{j,2})=24$. Without loss of generality, there are $2$ scenarios, (i)-(ii), to examine. 

{\bf (i):} $k_1=3$, $k_2=k_3=1$. In this case, it is easy to see from the adjunction inequality that $C$ is either embedded near both $q_2,q_3$, or is embedded at $q_2$ but is singular at $q_3$ with $\delta$-invariant $1$. In the former case, the rotation numbers at $q_2,q_3$ are either $(1,2)$ or $(1,4)$. Then by Lemma 2.1, it is easy to see that the rotation numbers $(m_{j,1},m_{j,2})$ in the formula of $d_C$ which are associated to the $3$ local branches of $C$ at $q_1$ has the property that $m_{j,1}+m_{j,2}=p=7$. It follows immediately that $d_C<0$ in this case. In the latter case, the rotation numbers at $q_2,q_3$ are $(1,2)$ (or $(1,4)$) and $(2,3)$ respectively, which contradicts Lemma 2.1.  (See Table 3.)

{\bf (ii):} $k_1=k_2=2$, $k_3=1$. Note that in this case, it follows easily from the adjunction inequality that $C$ is either embedded near $q_3$ or has a cuspidal singularity with $\delta$-invariant at most $3$. Using Table 3, the following is easy to see: First, assume $C$ is embedded near $q_3$. Then the rotation numbers at $q_3$ is 
either $(1,2)$ or $(1,4)$. It follows from Lemma 2.1 that the rotation numbers $(m_{j,1},m_{j,2})$ in the formula of $d_C$ that are associated to the local branches of $C$ at either $q_1$ or $q_2$ has the property that $m_{j,1}+m_{j,2}=p=7$. With this understood, one can check easily that the condition 
$\sum_{j=1}^5 (m_{j,1}+m_{j,2})=24$ cannot be satisfied, hence this possibility is eliminated. Next assume $C$
is singular at $q_3$ with $\delta$-invariant $3$. Then the rotation numbers at $q_3$ is $(3,4)$, which implies, by Lemma 2.1, that the rotation numbers associated to the local branches of $C$ at $q_1,q_2$ are $(1,2)$, $(1,4)$, $(1,2)$, $(1,4)$, as $C$ must have a nodal singularity at both $q_1,q_2$. But this contradicts the condition 
$\sum_{j=1}^5 (m_{j,1}+m_{j,2})=24$, hence it is not possible. A similar argument rules out the possibility that $C$ is singular at $q_3$ with a $\delta$-invariant $2$. 

It remains to examine the case where $C$ is singular at $q_3$ with $\delta$-invariant $1$. In this case, the rotation numbers at $q_3$ is $(2,3)$ by Table 3. Applying Lemma 2.1 to it, it follows easily that $C$ must have a nodal singularity at, say, $q_2$, and at $q_1$ there are two possibilities: either both local branches of $C$ are embedded and are tangent with a tangency of order $3$, or one of the local branches is embedded while the other is singular with $\delta$-invariant $1$ and the intersection number of the two local branches is $2$. One can check easily that in the former case, the condition 
$\sum_{j=1}^5 (m_{j,1}+m_{j,2})=24$ is violated, while in the latter case, the associated rotation numbers are $(2,3)$, $(1,3)$, $(1,3)$, $(1,5)$, and $(2,3)$, and the condition $\sum_{j=1}^5 (m_{j,1}+m_{j,2})=24$ is satisfied. Furthermore, one can also verify that in the latter case, Lemma 2.7 is not violated. Finally, it turns out that the Global Quotient Criterion is not violated if and only if the local representations of the $\Z_p$-action are given by case (ii) in Lemma 2.1. 
More concretely, let $g_1$ be the generator at $q_3$, with rotation numbers $(2,3)$, $g_2,g_3$ be the generators
at the nodal singularity, which have rotation numbers $(1,3)$ and $(1,5)$, and let $g_4,g_5$ be the generators at the reducible singularity, with rotation numbers $(2,3)$ and $(1,3)$, and we write
$g_j=g_1^{c_j}$. Then if we apply Lemma 2.1(ii) to $g_1$ with $(a,b)=(2,3)$, the rotation numbers at the other
two fixed points are $(2,6)$ and $(4,6)$. It follows easily that $c_2=4$, $c_3=6$, $c_4=4$, and $c_5=6$, from
which $c_1+c_2+c_3+c_4+c_5=21=0\pmod{7}$ holds true. This case is stated in Theorem 1.3 as (3)(iii). 

Finally, if $k\geq 7$, the adjunction inequality is easily seen violated. This finishes the discussion for (b): $2-2g_\Sigma-k=-7p$.

{\bf (c):} $2-2g_\Sigma-k=-6p$. In this case, $k$ is even and $4\leq k\leq 8$. Let $k=4$, in which case $g_\Sigma=20$ and $g_{|\Sigma/G|}=2$. In particular, $d_C\geq 0$ and $d_C\in \Z$ imply that the rotation numbers in the formula of $d_C$ satisfy $\sum_{j=1}^4 (m_{j,1}+m_{j,2})=17$. Now let $k_1=2$, $k_2=k_3=1$ without loss of generality. Using Table 3, the following is easy to see: if $C$ is embedded near $q_2$ or $q_3$, then the rotation numbers at the fixed point must be $(1,2)$ or $(1,4)$. By Lemma 2.1, we find that at one of the other two fixed points, the rotation numbers has the property that its sum equals $p=7$. Clearly this can not happen at $q_1$, and it follows easily that the sum of rotation numbers at $q_1$ that are in the formula of $d_C$ is at most $7$
(because of the constraint $\sum_{j=1}^4 (m_{j,1}+m_{j,2})=17$), so they must be $(1,2)$ and $(1,3)$. But that's clearly not possible for $p=7$, so $C$ must be singular at both $q_2,q_3$, with $\delta$-invariant at each of the fixed point less than $7$. This leads easily to the observation that the sum of the rotation numbers in the formula of $d_C$ that are from $q_2,q_3$ is at least $10$, so that again the sum of rotation numbers at $q_1$ in the formula of $d_C$ is at most $7$, and as we have showed earlier, this is not possible. On the other hand, if $k=8$, the adjunction inequality is easily seen violated.

So it remains to examine the case of $k=6$. In this case, $g_\Sigma=19$ and $g_{|\Sigma/G|}=1$. In particular,  $d_C\geq 0$ and $d_C\in \Z$ imply that the rotation numbers in the formula of $d_C$ must obey the condition 
$\sum_{j=1}^6 (m_{j,1}+m_{j,2})=24$. Without loss of generality, there are $3$ scenarios, (i)-(iii), to consider. 

{\bf (i):} $k_1=4$, $k_2=k_3=1$. There is only one possibility where the adjunction inequality is not violated, i.e., 
the four local branches of $C$ at $q_1$ are all embedded, with $3$ of them being tangent of order $2$, and $C$ is embedded near both $q_2,q_3$. Correspondingly, the rotation numbers at $q_1$ are $(1,4)$, $(1,2)$,  $(1,2)$, $(1,2)$, and at $q_2,q_3$ they are either $(1,2)$ or $(1,4)$ (cf. Table 3). But this contradicts Lemma 2.1, hence is not possible.

{\bf (ii):} $k_1=3$, $k_2=2$, $k_3=1$. The condition $\sum_{j=1}^6 (m_{j,1}+m_{j,2})=24$, the genus formula, and Lemma 2.1 leave only one possibility for the rotation numbers associated to the local branches of $C$, i.e., $(1,4)$, $(1,2)$, $(1,2)$ at $q_1$, $(1,2)$, $(1,2)$ at $q_2$, and $(3,4)$ at $q_3$ (see Table 3). Furthermore, it is easy to check that Lemma 2.7 is not violated at any of the fixed points. Finally, we note that the Global Quotient Criterion is not violated if and only if the local representations of the $\Z_p$-action are given by case (i) in Lemma 2.1. More concretely, let $g_1,g_2,g_3$ be the generators associated to the three branches at $q_1$, and $g_4,g_5$ the generators at $q_2$ and $g_6$ the generator at $q_3$, and write 
$g_j=g_1^{c_j}$. Then applying Lemma 2.1(i) to $g_1$ with $(a,b)=(1,4)$, we get $(6,3)$, $(3,4)$ as the rotation numbers at $q_2,q_3$. It follows easily that $c_2=c_3=2$, $c_4=c_5=5$, and if we let $c_6=6$, then the constraint $c_1+c_2+c_3+c_4+c_5+c_6=0\pmod{7}$ is satisfied. This case is stated in Theorem 1.3 as (3)(iv). 

{\bf (iii):} $k_1=k_2=k_3=2$. Again, the condition $\sum_{j=1}^6 (m_{j,1}+m_{j,2})=24$, the genus formula, and Lemma 2.1 leave only one possibility for the rotation numbers associated to the local branches of $C$, i.e., $(1,3)$ at all $3$ fixed points and for all local branches. Moreover, one can check easily that Lemma 2.7 is not violated. Finally, we note that the Global Quotient Criterion is not violated if and only if the local representations of the $\Z_p$-action are in case (i) of Lemma 2.1. 
More concretely, if we let $g_1=g_2$, $g_3=g_4$, and $g_5=g_6$ be the corresponding generators at the $3$ fixed point, and write $g_j=g_1^{c_j}$, then applying Lemma 2.1(i) to $g_1$ with $(a,b)=(1,3)$, we get the rotation numbers at the other two fixed points to be $(6,2)$ and $(4,5)$. It follows easily that $c_2=1$, $c_3=c_4=4$, $c_5=c_6=2$, which obeys the constraint  $c_1+c_2+c_3+c_4+c_5+c_6=0\pmod{7}$. 
This case is stated in Theorem 1.3 as (3)(v). 

{\bf (d):} $2-2g_\Sigma-k=-xp$ where $x\leq 5$. Note that $x=2g_{|\Sigma/G|}-2+k$, from which it follows easily 
that $d_C\leq \frac{1}{14}(8x-34)$. So $d_C<0$ if $x<5$. If $x=5$, then $k=3,5$ or $7$. In each of these cases, $d_C<0$ by the following simple observation: for $d_C\geq 0$, one of the rotation numbers in the formula of $d_C$ has to be $(1,2)$, but by Lemma 2.1, there must be another rotation numbers (at another fixed point) whose sum equals $p=7$, and as a consequence, $d_C<0$. This completes the discussion for the case $p=7$.

\vspace{2mm}

{\bf (4):} $p=11$.  First, suppose there is a $C\in\{C_i\}$ of degree $d\leq 3$. Then $d^2+3d+3\geq p$ implies that $d\geq 2$. It is easy to see that for both $d=2$ and $3$, $2-2g_\Sigma-k=-p$ is the only possibility in Lemma 2.5,
in which case, $k=3$, $g_{|\Sigma/G|}=0$, and $g_\Sigma=5$. 

Let $d=2$. In this case, it follows from the adjunction inequality that $C$ must have a cuspidal singularity with $\delta$-invariant $1$ at one of the fixed points, and is embedded near the other two fixed points. The rotation numbers at the cuspidal singularity must be $(2,3)$, which, with Lemma 2.1 and the constraints $d_C\geq 0$ and $d_C\in\Z$, implies that the rotation numbers at the remaining fixed points must be $(1,5)$ and $(1,4)$. Furthermore, one can verify easily that the Global Quotient Criterion is not violated if and only if the local representations of the $\Z_p$-action are given by case (iv) in Lemma 2.1. More concretely, let  $g_1,g_2,g_3$ be the corresponding generators, with rotation numbers $(2,3)$, $(1,5)$ and $(1,4)$ respectively, and we write $g_j=g_1^{c_j}$. Applying Lemma 2.1(iv) to $g_1$ with $(a,b)=(2,3)$, we obtain $(2,10)$, $(3,1)$ for the rotation numbers at the other two fixed points. It follows easily that $c_2=6$, $c_3=4$, and with $c_1=1$, the constraint $c_1+c_2+c_3=0\pmod{11}$ is satisfied. This case is stated in Theorem 1.3 as (4)(i). 

If $d=3$, then $C$ has at most cuspidal singularities at the $3$ fixed points, with the sum of $\delta$-invariants equaling $5$. Now observe that for $p=11$, there is no cuspidal singularity with $\delta$-invariant $5$. On the other hand, we note that in this case, $d_C\geq 0$ and $d_C\in\Z$ is equivalent to $\sum_{j=1}^3 (m_{j,1}+m_{j,2})=13$ for the rotation numbers $(m_{j,1},m_{j,2})$ in the formula of $d_C$. With this understood, one can check easily
that for all the possible distributions of the $\delta$-invariants among the $3$ fixed points, we always have 
$d_C<0$. Hence $d=3$ is eliminated. This completes the discussion for the case $d\leq 3$.

Secondly, assume that there is a $C\in\{C_i\}$ of degree $d=6$. 
We begin with some quick but useful observations: if $C$ is embedded near a fixed point, then by Lemma 2.7 the rotation numbers at the fixed point must be either $(1,3)$ or $(1,4)$. As a corollary, $C$ cannot be embedded near more than one fixed points by Lemma 2.1. Secondly, it also follows from Lemma 2.7 that $C$ cannot have a cuspidal singularity with $\delta$-invariant $1$ or $2$ at any fixed point. With this understood, we note that 
there are $4$ cases to consider: (a) $2-2g_\Sigma-k=-5p$, (b) $2-2g_\Sigma-k=-4p$, (c) $2-2g_\Sigma-k=-3p$,
and (d) $2-2g_\Sigma-k=-xp$ where $x\leq 2$. See Lemma 2.5.

{\bf (a):} $2-2g_\Sigma-k=-5p$. In this case we have $2-2g_{|\Sigma/G|}-k=-5$, which implies that either $k=7$, or $k=5$, or $k=3$. Moreover, it follows easily from the adjunction inequality that $k=7$ and $k=3$ are not possible. To rule out $k=5$, we observe that in this case, $g_\Sigma=26$, so the adjunction inequality implies that $C$ must be immersed, with two nodal singularities at two of the fixed points. Furthermore, the rotation numbers at the fixed point
where $C$ is embedded must be either $(1,3)$ or $(1,4)$. Now using Lemma 2.1, we conclude that the rotation numbers at one of the other two fixed points where $C$ has a nodal singularity must be $(1,7)$ and $(1,8)$.
But when we apply Lemma 2.7 to the rotation numbers $(1,7)$ or $(1,8)$, we get $int_D(q_i)>2$, which is a contradiction as $C$ has a nodal singularity at the fixed point, so that  $int_D(q_i)=2$ for a generic $J$-holomorphic disk $D$. Hence $k=5$ is also eliminated. 

{\bf (b):} $2-2g_\Sigma-k=-4p$. In this case, either $k=6$ or $k=4$, with $g_\Sigma=20$ or $21$ respectively. 
Assume $k=6$. The following identity follows from the genus formula 
$$
\frac{1}{2}\sum_{k_i>1}\sum_{j\neq j^\prime}C_{j,i}\cdot C_{j^\prime,i}+\sum_{k_i>1}\sum_{j=1}^{k_i} \delta_{j,i}+\sum_{i=1}^3 \delta_i=8.
$$
With this understood, there are $3$ scenarios, (i)-(iii), to consider. 

{\bf (i):} $k_1=4$, $k_2=k_3=1$. Note that in this case, $\frac{1}{2}\sum_{j\neq j^\prime}C_{j,1}\cdot C_{j^\prime,1}\geq \frac{1}{2}\times 4\times 3=6$
and $\delta_2+\delta_3\geq 3$, which is a contradiction. 

{\bf (ii):} $k_1=3$, $k_2=2$, and $k_3=1$. In this case, we observe that either $\delta_3=0$ or 
$\delta_3\geq 3$, and $C_{j,2}\cdot C_{j^\prime,2}\geq 1$, $\frac{1}{2}\sum_{j\neq j^\prime}C_{j,1}\cdot C_{j^\prime,1}\geq 4$
as $k_2=2$ and $k_1=3$. It follows easily that if $\delta_3\geq 3$, then one must have $\delta_3=3$, 
$C_{j,2}\cdot C_{j^\prime,2}=1$, and $\frac{1}{2}\sum_{j\neq j^\prime}C_{j,1}\cdot C_{j^\prime,1}=4$. As a consequence, the rotation numbers at the fixed point $q_3$ is either $(2,7)$ or $(3,4)$, $C$ has a nodal singularity at $q_2$, and the three local branches of $C$ at $q_1$ are embedded where two of them are tangent with a tangency of order $2$. In particular, $(1,2)$ is the rotation numbers at $q_1$ which is associated to the two tangent branches. But the conclusion that the rotation numbers is $(1,2)$ at $q_1$ and $(2,7)$ or $(3,4)$ at $q_3$ are in contradiction with Lemma 2.1, hence 
$\delta_3\geq 3$ is ruled out.

Next consider the case $\delta_3=0$. With $\frac{1}{2}\sum_{j\neq j^\prime}C_{j,1}\cdot C_{j^\prime,1}\geq 4$ as $k_1=3$, one must have $\sum_{j=1}^{2} \delta_{j,2}+C_{j,2}\cdot C_{j^\prime,2}\leq 4$. We claim that this implies that 
both of the $\delta_{j,2}$'s must be $0$, i.e., the two local branches 
$C_{j,2}, C_{j^\prime,2}$ are both embedded. 
To see this, consider first the case where both of the $\delta_{j,2}$'s are nonzero. Then 
$\sum_{j=1}^{2} \delta_{j,2}+C_{j,2}\cdot C_{j^\prime,2}\geq 6$ because $C_{j,2}\cdot C_{j^\prime,2}\geq 4$, which is a contradiction. Next assume that only one of the $\delta_{j,2}$'s is nonzero.
Then one of the local branches $C_{j,2}$ or $C_{j^\prime,2}$ is embedded while the other is singular. We claim that the $\delta$-invariant $\delta_{j,2}$ of the singular branch must be either $1$ or $2$, as if otherwise, we would have 
$\sum_{j=1}^{2} \delta_{j,2}+C_{j,2}\cdot C_{j^\prime,2}\geq 3+2=5$, which is a contradiction. But $\delta_{j,2}=1$
or $2$ implies that the rotation numbers at $q_2$ is either $(2,3)$ or $(2,5)$. Applying Lemma 2.7 to the fixed point
$q_2$, it is easy to check that in either case, $int_D(q_2)\geq a_1+a_2\geq 4$ is always true. But for a generic 
$D$, $int_D(q_2)=3$, as one of the local branches at $q_2$, $C_{j,2}$ or $C_{j^\prime,2}$, is embedded and the other is singular with multiplicity $2$ at the singularity, which is a contradiction. 
Hence both of the $\delta_{j,2}$'s are $0$, and $C_{j,2}, C_{j^\prime,2}$ are both embedded. 

To proceed further, we claim $C_{j,2}\cdot C_{j^\prime,2}=1$. Suppose to the contrary that $m:=C_{j,2}\cdot C_{j^\prime,2}>1$. Then $(1,m)$ is the rotation numbers at $q_2$, and by Lemma 2.7, $m=3,4,5,9$, or $10$
must be true. But since the rotation numbers at $q_3$ is either $(1,3)$ or $(1,4)$ as $\delta_3=0$, 
$(1,3)$ or $(1,4)$ cannot be the rotation numbers at $q_2$, which implies $m\geq 5$. 
But $C_{j,2}\cdot C_{j^\prime,2}\geq 5$ and $\frac{1}{2}\sum_{j\neq j^\prime}C_{j,1}\cdot C_{j^\prime,1}\geq 4$ again form a contradiction. Hence $C_{j,2}\cdot C_{j^\prime,2}=1$, and therefore, $\frac{1}{2}\sum_{j\neq j^\prime}C_{j,1}\cdot C_{j^\prime,1}+\sum_{j=1}^3\delta_{j,1}=7$ must be true. With this understood, we note that with the rotation numbers at $q_3$ being $(1,3)$ or $(1,4)$, the rotation numbers at $q_2$ must be $(1,5)$ and $(1,9)$ for the two branches of $C$ at $q_2$ as a result of Lemma 2.1 and Lemma 2.7. Consequently, the rotation numbers 
at $q_1$ has to be $(1,7)$ or $(1,8)$ by Lemma 2.1. With these preparations, one can check easily that if one of the $\delta_{j,1}$'s is nonzero, $\frac{1}{2}\sum_{j\neq j^\prime}C_{j,1}\cdot C_{j^\prime,1}\geq 6$ must be true, 
and if all $\delta_{j,1}=0$, then $\frac{1}{2}\sum_{j\neq j^\prime}C_{j,1}\cdot C_{j^\prime,1}\geq 9$. It follows easily that the only possibility is $\sum_{j=1}^3\delta_{j,1}=1$ and $\frac{1}{2}\sum_{j\neq j^\prime}C_{j,1}\cdot C_{j^\prime,1}= 6$. From this, we can easily conclude that the rotation numbers at $q_1$ that are associated to the three local branches of $C$ are $(2,3)$, $(1,7)$, and $(1,8)$, and the rotation numbers at $q_2$ that are associated to the two local branches 
of $C$ are $(1,5)$ and $(1,9)$, and the rotation numbers at $q_3$ is either $(1,3)$ or $(1,4)$. 
With $g_{|\Sigma/G|}=0$ when $k=6$, we can easily see that $d_C<0$, which is a contradiction. 

{\bf (iii):} $k_1=k_2=k_3=2$. We begin with the observation from (ii) above that for any $i=1,2,3$, if one of 
$\delta_{1,i}, \delta_{2,i}$ is nonzero, then $C_{1,i}\cdot C_{2,i}+\delta_{1,i}+\delta_{2,i}\geq 5$. With 
$\sum_{i=1}^3 (C_{1,i}\cdot C_{2,i}+\delta_{1,i}+\delta_{2,i})=8$, it follows that there is at most one such $i$. On the
other hand, assume both $\delta_{1,i}=\delta_{2,i}=0$, i.e., $C_{1,i}, C_{2,i}$ are both embedded. If $m=C_{1,i}\cdot C_{2,i}>1$, then $m=3,4,5,9$ or $10$ by Lemma 2.7. With these two observations, it follows easily that the only possibility, without loss of generality, is that $\delta_{1,1}=3$, $\delta_{2,1}=0$ and 
$C_{1,1}\cdot C_{2,1}=3$,  and $C_{1,2}\cdot C_{2,2}=C_{1,3}\cdot C_{2,3}=1$. Correspondingly, the rotation numbers at $q_1$ are $(3,4)$ and $(1,9)$,  the rotation numbers at $q_2$ are $(1,3)$ and $(1,4)$, and
the rotation numbers at $q_3$ are $(1,7)$ and $(1,8)$. But with $(1,7)$ or $(1,8)$ at $q_3$, Lemma 2.7 implies that
$int_D(q_3)>2$, which is a contradiction. 

Thus the case of $k=6$ is eliminated. Next let $k=4$. There is only one possibility, i.e., $k_1=2$, $k_2=k_3=1$, and as such, $C_{1,1}\cdot C_{2,1}+\delta_{1,1}+\delta_{2,1}+\delta_2+\delta_3=8$. Furthermore, we observe that
$\delta_2+\delta_3\geq 3$, and $C_{1,1}\cdot C_{2,1}+\delta_{1,1}+\delta_{2,1}\geq 5$ if $\delta_{1,1}$ or
$\delta_{2,1}$ is nonzero. With this understood, as we argued before we can make a list of possible rotation numbers, where in each case the first two rotation numbers are associated with $C_{1,1}, C_{2,1}$ and the last two rotation numbers are at $q_2$, $q_3$ respectively:
\begin{itemize}
\item [{(1)}] $(2,7)$, $(1,5)$, $(1,3)$ (or $(1,4)$), $(2,7)$ (or $(3,4)$).
\item [{(2)}] $(1,5))$, $(1,5)$, $(1,3)$ (or $(1,4)$), $(2,7)$ (or $(3,4)$).
\item [{(3)}] $(1,4)$, $(1,4)$, $(1,3)$ (or $(1,4)$), $(2,9)$ (or $(3,5)$).
\item [{(4)}] $(1,3)$, $(1,4)$, $(1,3)$ (or $(1,4)$), $(3,8)$.
\item [{(5)}] $(1,3)$, $(1,4)$, $(2,7)$ (or $(3,4)$), $(2,9)$ (or $(3,5)$).
\item [{(6)}] $(1,5)$, $(1,9)$, $(1,3)$ (or $(1,4)$), $(3,8)$.
\item [{(7)}] $(1,5)$, $(1,9)$, $(2,7)$ (or $(3,4)$), $(2,9)$ (or $(3,5)$).
\item [{(8)}] $(1,10)$, $(1,10)$, $(1,3)$ (or $(1,4)$), $(3,8)$.
\item [{(9)}] $(1,10)$, $(1,10)$, $(2,7)$ (or $(3,4)$), $(2,9)$ (or $(3,5)$).
\end{itemize}
All these cases can be eliminated by Lemma 2.1. Hence $k=4$ is also ruled out. 

{\bf (c):} $2-2g_\Sigma-k=-3p$. Here we have either $k=3$ or $k=5$. Let $k=3$ first. In this case, 
$g_{|\Sigma/G|}=1$, from which it follows easily that $d_C\geq 0$ if and only if 
$\sum_{j=1}^3 (m_{j,1}+m_{j,2})\leq 15$ in the formula of $d_C$. With this understood, we observe that 
\begin{itemize}
\item [{(i)}]  By Lemma 2.1, $(1,3)$ cannot be realized as rotation numbers at two different fixed points, and if
$d_C\geq 0$, then $m_{j,1}+m_{j,2}\geq 4$ for any $j=1,2,3$.
\item [{(ii)}] As we mentioned earlier, by Lemma 2.7, the $\delta$-invariant of a singularity of $C$ is at least $3$.
Moreover, by Lemma 2.7 and Lemma 2.1, $C$ cannot be embedded at more than one fixed points.
With $g_\Sigma=16$, the adjunction formula implies that $C$ must have a singularity 
with $\delta$-invariant equaling $3,4,5$ or $6$. It is easy to check that in any of these cases, the corresponding rotation
numbers at the fixed point obeys $m_{j,1}+m_{j,2}\geq 7$.
\end{itemize}
It follows easily from (i) and (ii) above that $d_C<0$, which eliminates the case of $k=3$. 

For $k=5$, we will also prove $d_C<0$ except in one possibility, although the argument is a bit more delicate. Note that $g_{|\Sigma/G|}=0$ if $k=5$, so $d_C\geq 0$ and $d_C\in\Z$ imply $\sum_{j=1}^5 (m_{j,1}+m_{j,2})=26$ in the formula of $d_C$. 
The condition $\sum_{j=1}^5 (m_{j,1}+m_{j,2})=26$ leaves us with two cases to discuss: Case A and Case B.

\vspace{1mm}

{\bf Case A:} {\it There is an $j$ such that $(m_{j,1},m_{j,2})=(1,2)$.}

\vspace{1mm}

In this case, one has the following two scenarios without loss of generality.

{\bf A(i):} $k_1=k_2=2$, $k_3=1$. First, note that $g_\Sigma=15$, so the genus formula implies
$$
\sum_{i=1}^2 (C_{1,i}\cdot C_{2,i}+\delta_{1,i}+\delta_{2,i})+\delta_3=13.
$$
Now by Lemma 2.1, one of the fixed points has the property that any rotation numbers $(a,b)$ at the fixed point obey $a+b=p=11$. Clearly, this fixed point must be $q_3$; otherwise $d_C<0$. With this understood, it follows easily that $C$ must be singular at $q_3$, with rotation numbers $(2,9)$, $(3,8)$, $(4,7)$, or $(5,6)$. In particular, we note that $\delta_3\geq 4$. Furthermore, Lemma 2.1 also implies that at $q_1,q_2$, the rotation numbers of the form $(1,m)$, where $m>1$, must be either $(1,2)$ or $(1,6)$. This has the following consequence: note that for $i=1,2$, $C_{1,i}\cdot C_{2,i}+\delta_{1,i}+\delta_{2,i}\leq 8$ as $\delta_3\geq 4$, which implies that if $C_{1,i}$ or $C_{2,i}$ is singular, then the corresponding $\delta$-invariant must be at most $6$. Now here is the crucial observation: since only $(1,2)$ or $(1,6)$ can be the rotation numbers of the form $(1,m)$ at $q_1,q_2$, the only possible singularity $C_{1,i}$ or $C_{2,i}$ can have must be of $\delta$-invariant $6$, with rotation numbers $(3,7)$. But this implies $C_{1,i}\cdot C_{2,i}+\delta_{1,i}+\delta_{2,i}\geq 9$, which is a contradiction. Consequently,  $C_{1,i}$ and $C_{2,i}$ must be embedded for $i=1,2$. With this understood, we observe that the rotation numbers associated to each $C_{j,i}$ must be of the form $(1,m)$ where $m=2$ or $6$, and this fact will imply that $\sum_{j=1}^5 (m_{j,1}+m_{j, 2})$ is always an odd number, which is a contradiction. 

{\bf A(ii):} $k_1=3$, $k_2=k_3=1$. In this case, the genus formula gives rise to
$$
\frac{1}{2}\sum_{j\neq j^\prime} C_{j,1}\cdot C_{j^\prime,1}+\sum_{j=1}^3 \delta_{j,1}+\delta_2+\delta_3=13.
$$
By the same argument, we may assume without loss of generality that $q_3$ is the fixed point which has the property that any rotation numbers $(a,b)$ at the fixed point obey $a+b=p=11$. In particular, $\delta_3\geq 4$. 
On the other hand, note that at the fixed point $q_2$, the rotation numbers of the form $(1,m)$ must be either $(1,2)$ or $(1,6)$ by Lemma 2.1, and as we noted in A(i), $C$ must be singular at $q_2$, with $\delta_2=6$. 
But this would imply that $\frac{1}{2}\sum_{j\neq j^\prime} C_{j,1}\cdot C_{j^\prime,1}+\sum_{j=1}^3 \delta_{j,1}\leq 3$, 
which is a contradiction because $\frac{1}{2}\sum_{j\neq j^\prime} C_{j,1}\cdot C_{j^\prime,1}\geq 4$ as $k_1=3$.
So Case A is eliminated. 

\vspace{1mm}

{\bf Case B:} {\it There is an $j$ such that $(m_{j,1},m_{j,2})=(1,3)$ or $(1,4)$.}

\vspace{1mm}

Denote by $q$ the fixed point at which the rotation numbers $(m_{j,1},m_{j,2})=(1,3)$ or $(1,4)$.
Then by Lemma 2.1, neither $(1,3)$ nor $(1,4)$ can be realized as rotation numbers at any fixed point other than $q$. As a consequence, let $(a,b)$ be a rotation numbers at a fixed point $q^\prime\neq q$. Then $a+b\geq 7$ except $(a,b)=(2,3)$ or $(a,b)=(1,5)$. With this understood, there are two scenarios without loss of generality.

{\bf B(i):} $k_1=k_2=2$, $k_3=1$. Consider first the case where $(1,3)$ or $(1,4)$ is not at $q_3$, say it is at $q_1$ without loss of generality. Then with the observation that $(2,3)$ and $(1,5)$ cannot be rotation numbers at $q_3$ due to Lemma 2.7, we must have $(2,3)$ as rotation numbers at $q_2$; otherwise 
$\sum_{j=1}^5 (m_{j,1}+m_{j, 2})\geq 2(1+3)+2(1+5)+7=27$, which is a contradiction. What this means is that 
one of the local branches of $C$ at $q_2$,
i.e., $C_{1,2}$ or $C_{2,2}$ must have a simple cusp singularity. With this understood, it follows easily from
Lemma 2.7 that the other local branch must also be singular, and moreover, the rotation numbers associated to it must also be $(2,3)$; otherwise $\sum_{j=1}^5 (m_{j,1}+m_{j, 2})\geq 27$ as well.  As a consequence, $C_{1,2}\cdot C_{2,2}+\delta_{1,2}+\delta_{2,2}=8$ as $C_{1,2}\cdot C_{2,2}=2^2+1+1=6$. 
With $\delta_3\geq 3$, it follows easily that $C$ must have a
nodal singularity at $q_1$, i.e., $C_{1,1}$ and $C_{2,1}$ are embedded intersecting transversely. In particular,
the two rotation numbers at $q_1$ associated to $C_{1,1}$ and $C_{2,1}$ must be $(1,3)$ and $(1,4)$ respectively. 
On the other hand, observe that $\delta_3=4$, which implies that the rotation numbers at $q_3$ is either
$(2,9)$ or $(3,5)$. But this implies that $\sum_{j=1}^5 (m_{j,1}+m_{j, 2})\geq (1+3)+(1+4) + 2(2+3)+8=27$, 
which is also a contradiction. 

Hence $(1,3)$ or $(1,4)$ must be at $q_3$. Then it is easily seen that the only possibility to have $\sum_{j=1}^5 (m_{j,1}+m_{j, 2})=26$ hold true is that the rotation numbers at $q_1,q_2$ are $(2,3)$, $(2,3)$, $(1,5)$, $(1,5)$, with rotation numbers at $q_3$ being $(1,3)$. We claim that in this case,
Lemma 2.7, which is not violated at the fixed point $q_3$, is also not violated at $q_1$ and $q_2$.
To see this, we observe that at $q_1$, $int_D(q_1)=4$ for a generic $J$-holomorphic disk $D$, while
in Lemma 2.7 with $(m_1,m_2)=(2,3)$, we may take $(a_1,a_2)=(0,4)$, which obeys both
$$
(a_1+2)m_1+(a_2+2)m_2=0 \pmod{p}, \;\;\;\;  int_D(q_1)\geq a_1+a_2.
$$
Hence Lemma 2.7 is not violated at $q_1$. At the fixed point $q_2$, one has $int_D(q_2)=2$ and 
$(m_1,m_2)=(1,5)$ instead, and we may take $(a_1,a_2)=(0,2)$, which also shows that Lemma 2.7
is not violated. Finally, one can check that the Global Quotient Criterion is not violated if and only if the local representations of the $\Z_p$-action are given by case (i) in Lemma 2.1. 
More concretely, let $g_1$ be the generator at $q_3$, with rotation numbers $(1,3)$, and let $g_2=g_3$, $g_4=g_5$ be the generators at the other two fixed points, with rotation numbers $(1,5)$, $(2,3)$ respectively, and we write  $g_j=g_1^{c_j}$. Then applying Lemma 2.1(i) to $g_1$ with 
$(a,b)=(1,3)$, we obtain $(10,2)$, $(8,9)$ for the rotation numbers at the other two fixed points. It follows easily that $c_2=c_3=6$, $c_4=c_5=10$, and with $c_1=1$, the constraint $c_1+c_2+c_3+c_4+c_5=0\pmod{11}$ is satisfied. This case is stated as (4)(ii) in Theorem 1.3.

{\bf B(ii):} $k_1=3$, $k_2=k_3=1$. Consider first the case where $(1,3)$ or $(1,4)$ occurs as rotation numbers at
$q_2$ or $q_3$, say $q_3$ without loss of generality. Then by Lemma 2.7, neither $(2,3)$ nor $(1,5)$ can be rotation numbers at $q_2$. On the other hand, observe that $(2,3)$ and $(1,5)$ cannot both occur as rotation numbers at any fixed point (in particular, at $q_1$). It follows easily that $d_C<0$ unless $(2,3)$ is the rotation numbers at $q_1$ associated to all three local branches. But this is not possible by the adjunction inequality. 
It remains to examine the case where $(1,3)$ or $(1,4)$ occurs as rotation numbers at $q_1$. Then by Lemma 2.7, neither $(2,3)$ nor $(1,5)$ can be rotation numbers at $q_2$ and $q_3$. With this understood, it follows easily
that $d_C<0$ unless $(1,3)$ is the rotation numbers at $q_1$ associated to all three local branches. But this is also not possible by the adjunction inequality. This finishes the discussion on Case B.

{\bf (d):} $2-2g_\Sigma-k=-xp$ where $x\leq 2$. It is immediately clear that $d_C<0$. This completes the analysis for the case of $p=11$.

\vspace{2mm}

{\bf (5):} $p=13$. First, assume that there is a $C\in\{C_i\}$ of degree $d\leq 3$. It follows easily, as in the case
of $p=11$, that $d\geq 2$, and for both $d=2$ and $3$, $2-2g_\Sigma-k=-p$ is the only possibility in Lemma 2.5, in which case, $k=3$, $g_{|\Sigma/G|}=0$, and $g_\Sigma=6$. With this understood, consider first the case where $d=2$. It is easy to see that $C$ must be embedded. Consequently, in the formula 
$$
d_C=\frac{1}{p} c_1(TM)\cdot C+ (g_{|\Sigma/G|}-1)+3-\sum_{j=1}^3 \frac{m_{j,1}+m_{j,2}}{p},
$$
the rotation numbers $(m_{j,1},m_{j,2})$ must take the form of $(1,m_j)$ for each $j=1,2,3$. With $p=13$,
$d=2$, $g_{|\Sigma/G|}=0$, $1+m_j\geq 3$ for each $j$, and $d_C\geq 0$ and $d_C\in\Z$, we see easily that
$m_1+m_2+m_3=17$ must be true. In particular, at least one of the $m_j$'s must be greater than $5$. Without 
loss of generality, we assume $m_1>5$. We can apply Lemma 2.1 to get further informations on the rotation numbers, and as a result, the only possibility for $m_1+m_2+m_3=17$ to hold true is $(1,6)$, $(1,8)$, $(1,3)$. 
Furthermore, one can check easily that the Global Quotient Criterion is not violated if and only if the local representations of the $\Z_p$-action are in case (iv) of Lemma 2.1. More concretely, let
$g_1,g_2,g_3$ be the generators, with rotation numbers $(1,6), (1,8), (1,3)$ respectively, and we write $g_j=g_1^{c_j}$. Then applying Lemma 2.1(iv) to $g_1$ with $(a,b)=(1,6)$, we get $(1,8)$, $(6,5)$ as the rotation numbers at the other two fixed points. It follows easily that $c_2=1$, $c_3=11$, and 
with $c_1=1$, the constraint $c_1+c_2+c_3=0\pmod{13}$ is satisfied. This case is stated as (5)(i) in Theorem 1.3. 

If $d=3$, then $C$ has at most cuspidal singularities at the $3$ fixed points, with the sum of $\delta$-invariants equaling $4$. A case by case analysis shows that $d_C<0$, which rules out $d=3$. Here we point out a useful observation: for $p=13$, $(2,3)$ cannot be realized as rotation numbers at more than one fixed points 
by Lemma 2.1.

Secondly, assume that there is a $C\in\{C_i\}$ of degree $d=6$, i.e., $c_1(2K)=C$. We begin with the following observations: if $C$ is embedded near a fixed point, then by Lemma 2.7 the rotation numbers at the fixed point must be $(1,5)$ or $(1,8)$. As a corollary, $C$ cannot be embedded near more than one fixed points by Lemma 2.1. Furthermore, the following is a list of rotation numbers associated to a cuspidal singularity of $C$ (including $C$ is embedded) with a $\delta$-invariant less than $9$, subject to the constraints by Lemma 2.7; in particular, note that $\delta$-invariant $\neq 2,7$ and $8$.

\vspace{3mm}

\centerline{
\begin{tabular}{lllllll}
\hline \vline\hspace{1mm}
$\delta$-invariant $<9$ & \vline \hspace{1mm} rotation numbers  \hspace{5.8mm} \vline \\
\hline \vline\hspace{1mm}
0 &  \vline \hspace{1mm} $(1,5)$ or $(1,8)$ \hspace{11.6mm} \vline \\ 
\hline \vline\hspace{1mm}
1 &  \vline \hspace{1mm} $(2,3)$  \hspace{26.1mm} \vline \\ 
\hline \vline\hspace{1mm} 
3 &  \vline \hspace{1mm} $(3,4)$  \hspace{26.1mm} \vline\\
\hline \vline\hspace{1mm} 
4 &  \vline \hspace{1mm} $(2,9)$ or $(3,5)$ \hspace{11.6mm}  \vline\\
\hline \vline\hspace{1mm} 
5 &  \vline \hspace{1mm} $(2,11)$  \hspace{24.2mm} \vline\\
\hline \vline\hspace{1mm} 
6 &  \vline \hspace{1mm} $(3,7)$  \hspace{26.1mm} \vline\\ \hline
\end{tabular}
}
\vspace{2mm}

\centerline{Table 4: $\delta$-invariants and rotation numbers at a cuspidal fixed point, $p=13$, $d=6$.
}
\vspace{2mm}

With the preceding understood, by Lemma 2.5 there are $4$ different cases we need to examine: (a) $2-2g_\Sigma-k=-4p$, (b) $2-2g_\Sigma-k=-3p$, (c) $2-2g_\Sigma-k=-2p$, and (d) $2-2g_\Sigma-k=-p$. 

{\bf (a):} $2-2g_\Sigma-k=-4p$. In this case we have $2-2g_{|\Sigma/G|}-k=-4$, which implies that either $k=6$ or $k=4$, as $k\geq 3$. A quick examination of the genus formula for $g_\Sigma$ rules out all the possibilities (where we note that $g_\Sigma=24$ if $k=6$ and $g_\Sigma=25$ if $k=4$), except when $C$ has a nodal singularity at two of the fixed points, and at the third fixed point, $C$ has two embedded local branches which are tangent of order $2$. In this exceptional case, note that the rotation numbers at the last fixed point are $(1,2)$ for both local branches. It follows easily from Lemma 2.1 that the rotation numbers at the two nodal singularities must be $(1,12)$, $(1,12)$ and $(1,2)$, $(1,7)$. A simple calculation shows that $d_C$ is a non-integer, which is a contradiction. 

{\bf (b):} $2-2g_\Sigma-k=-3p$. In this case, either $k=3$ or $k=5$. We should point out that Table 4 will be frequently referred to in our discussions.

Let $k=3$ first. In this case, $g_{|\Sigma/G|}=1$ and $g_\Sigma=19$, and as $k=3$, $C$ has only cuspidal singularities, all occurring at the fixed points, with the sum of $\delta$-invariants equaling $9$. On the other hand, $d_C\geq 0$ and $d_C\in\Z$ imply that $\sum_{j=1}^3 (m_{j,1}+m_{j,2})=21$ for the rotation numbers in the formula of $d_C$. There are $2$ scenarios to consider.

{\bf (i):} $C$ is singular at all $3$ fixed points. First, if there is a singular point of $C$ with $\delta$-invariant $1$, then the other two singular points either both have $\delta$-invariant $4$ or one has $\delta$-invariant $3$ and the other $\delta$-invariant $5$ (cf. Table 4). Except for the following rotation numbers $(2,3)$, $(3,5)$, $(3,5)$, all other possibilities lead to $d_C<0$. But in the case of $(2,3)$, $(3,5)$, $(3,5)$, Lemma 2.1 is violated. 

Secondly, there is another possibility, i.e., all $3$ singular points have $\delta$-invariant $3$, and the rotation numbers at each singular point is $(3,4)$. This case is consistent with Lemma 2.1, and the condition $\sum_{j=1}^3 (m_{j,1}+m_{j,2})=21$ is satisfied. Moreover, note that Lemma 2.7 is also not violated as the rotation numbers $(3,4)$ is listed in Table 4. Finally, one can check that the Global Quotient Criterion is not violated if and only if the local representations of the $\Z_p$-action are given by case (i) in Lemma 2.1. More concretely, let $g_1,g_2,g_3$ be the generators, each with rotation numbers $(3,4)$, and we write $g_j=g_1^{c_j}$. Then if we apply Lemma 2.1(i) to $g_1$ with $(a,b)=(3,4)$, we get $(10,1)$, $(9,12)$ as the rotation numbers at the other two fixed points. From this it follows easily that $c_2=3$, $c_3=9$, and $c_1+c_2+c_3=13=0\pmod{13}$ is satisfied. This case is stated as (5)(ii) in Theorem 1.3. 

{\bf (ii):} $C$ is embedded near one of the fixed point. Note that there is no singularity of $\delta$-invariant $8$, so consequently, the other two singularities of $C$ are of $\delta$-invariants $3$ and $6$ or $4$ and $5$. Both cases lead to $d_C<0$ using Table 4. 

Next, let $k=5$. In this case, $g_{|\Sigma/G|}=0$ and $g_\Sigma=18$. Without loss of generality, there are $2$
scenarios, (iii)-(iv), to consider.

{\bf (iii):} $k_1=3$, $k_2=k_3=1$. First, we observe, due to Lemma 2.1, that if $C$ is embedded or singular with $\delta$-invariant $1$ at one of the fixed points $q_2,q_3$, it cannot be embedded or singular at the other fixed point with a $\delta$-invariant $1$ or $3$ (cf. Table 4). Consequently, in the following equation from the genus formula
$$
\frac{1}{2}\sum_{j\neq j^\prime} C_{j,1}\cdot C_{j^\prime,1}+\sum_{j=1}^3 \delta_{j,1}+\delta_2+\delta_3=10,
$$
one has $\delta_2+\delta_3\geq 4$, which implies that $\frac{1}{2}\sum_{j\neq j^\prime} C_{j,1}\cdot C_{j^\prime,1}+\sum_{j=1}^3 \delta_{j,1}\leq 6$. As a result,
the local branches of $C$ at $q_1$ must be all embedded. Moreover, we observe that $4\leq \frac{1}{2}\sum_{j\neq j^\prime} C_{j,1}\cdot C_{j^\prime,1}\leq 6$.
The following is a list of possible rotation numbers for the cases where the genus formula holds; in each case the first $3$ rotation numbers are at $q_1$ and the last two are at $q_2,q_3$ respectively:
\begin{itemize}
\item [{(1)}] $(1,2)$, $(1,2)$, $(1,7)$, $(1,5)$ (or $(1,8)$), $(3,7)$.
\item [{(2)}] $(1,2)$, $(1,2)$, $(1,7)$, $(2,3)$, $(2,11)$.
\item [{(3)}] $(1,2)$, $(1,2)$, $(1,7)$, $(3,4)$, $(3,4)$.
\item [{(4)}] $(1,3)$, $(1,3)$, $(1,9)$, $(1,5)$ (or $(1,8)$), $(2,11)$.
\item [{(5)}] $(1,3)$, $(1,3)$, $(1,9)$, $(2,3)$, $(2,9)$ (or $(3,5)$).
\item [{(6)}] $(1,4)$, $(1,4)$, $(1,10)$, $(1,5)$ (or $(1,8)$), $(2,9)$ (or $(3,5)$).
\item [{(7)}] $(1,4)$, $(1,4)$, $(1,10)$, $(2,3)$, $(3,4)$.
\item [{(8)}] $(1,2)$, $(1,2)$, $(1,2)$, $(1,5)$ (or $(1,8)$), $(2,9)$ (or $(3,5)$).
\item [{(9)}] $(1,2)$, $(1,2)$, $(1,2)$, $(2,3)$, $(3,4)$.
\end{itemize}
All cases except (5) can be eliminated by Lemma 2.1. Furthermore, case (5) with the last rotation numbers being $(3,5)$ can be ruled out as well,
due to the fact that $d_C$ is a non-integer in this case. However, in case (5) with the last rotation numbers being $(2,9)$, $d_C=0$. But this case is still not possible, as the Global Quotient Criterion is violated for all cases of the local representations listed in Lemma 2.1. 

{\bf (iv):} $k_1=k_2=2$, $k_3=1$. First, we observe that by Lemma 2.7, if one of the local branches of $C$ at $q_1$ or $q_2$ is singular, the $\delta$-invariant of the singularity can not equal $2$. On the other hand, by the adjunction inequality the $\delta$-invariant of the singularity can be at most $3$,
and if it is $3$, then the corresponding rotation numbers must be $(3,4)$ by Lemma 2.7. With this understood, our discussions will be divided into the
following cases, (A)-(G), without loss of generality.

{\bf (A):} One of the local branches at $q_1$ is singular with a $\delta$-invariant $3$. In this case, Lemma 2.1 and the genus formula imply easily that the following rotation numbers are the only possibility: $(3,4)$ and $(1,4)$ at $q_1$, $(1,4)$ and $(1,10)$ at $q_2$, and $(3,4)$ at $q_3$. However, this case is not possible as it leads to $d_C<0$.

{\bf (B):} One of the local branches at $q_1$ is singular with a $\delta$-invariant $1$. First, the adjunction inequality implies that the other local branch at
$q_1$ cannot be singular, so in this case, the rotation numbers associated to the two local branches at $q_1$ are either $(2,3)$ and $(1,5)$, or $(2,3)$ and
$(1,8)$, depending on whether the two branches are transversal to each other or not.  Furthermore, by Lemma 2.1 the local branches at $q_2$ cannot be
singular with $\delta$-invariant $1$, so they must be all embedded. With this understood, the following is a list of possible rotation numbers where Lemmas 2.1 and 2.7 are both being considered and the genus formula holds; in each case the first two pairs of rotation numbers are at $q_1$ and $q_2$ respectively.
\begin{itemize}
\item [{(1)}] $(2,3)$ and $(1,5)$, $(1,3)$ and $(1,3)$, $(2,9)$ (or $(3,5)$).
\item [{(2)}] $(2,3)$ and $(1,5)$, $(1,3)$ and $(1,9)$, $(3,7)$.
\end{itemize}
In (1), one can check that $d_C$ is always a non-integer, and in (2), $d_C<0$.

{\bf (C):} The local branches of $C$ are embedded at $q_1,q_2$ and the $\delta$-invariant of the singularity at $q_3$ equals $6$ (which is maximal among what are possible). The rotation numbers at $q_3$ must be $(3,7)$. Using Lemma 2.1, one can determine the possible rotation numbers at $q_1,q_2$.
The following is the only possibility where the genus formula holds true: $(1,5)$ and $(1,8)$ at $q_1$, $(1,3)$ and $(1,3)$ at $q_2$, and $(3,7)$ at $q_3$. 
It can be eliminated as $d_C$ is a non-integer in this case. 

{\bf (D):} The local branches of $C$ are embedded at $q_1,q_2$ and the $\delta$-invariant of the singularity at $q_3$ equals $5$. Using Lemma 2.1 (and
by the fact that the local branches at $q_1,q_2$ are all embedded), the only possible rotation numbers for each of the local branches is either $(1,2)$
or $(1,7)$. One can check easily that the genus formula will never hold true in this case.

{\bf (E):} The local branches of $C$ are embedded at $q_1,q_2$ and the $\delta$-invariant of the singularity at $q_3$ equals $4$. Using Lemma 2.1, 
the possible rotation numbers for each of the local branches at $q_1,q_2$ are $(1,3)$, $(1,9)$, $(1,5)$, or $(1,8)$. One can check easily that the genus formula will never hold true in this case as well.

{\bf (F):} The local branches of $C$ are embedded at $q_1,q_2$ and the $\delta$-invariant of the singularity at $q_3$ equals $3$. Using Lemma 2.1, 
the possible rotation numbers for each of the local branches at $q_1,q_2$ is either $(1,4)$ or $(1,10)$. 
One can check easily that the genus formula will never hold true in this case. 

{\bf (G):} The local branches of $C$ are embedded at $q_1,q_2$ and the $\delta$-invariant of the singularity at $q_3$ equals $1$ or $0$. 
Using Lemma 2.1, the possible rotation numbers for each of the local branches at $q_1,q_2$ are $(1,3)$, $(1,9)$, $(1,6)$, or $(1,11)$.  The following are
the only possibilities for which the genus formula holds:
\begin{itemize}
\item [{(1)}] $(1,3)$ and $(1,3)$ at $q_1$, $(1,6)$ and $(1,6)$ at $q_2$, and $(2,3)$ at $q_3$. 
\item [{(2)}] $(1,6)$ and $(1,11)$ at $q_1$, $(1,9)$ and $(1,9)$ at $q_2$, and $(1,5)$ (or $(1,8)$) at $q_3$. 
\end{itemize}
Both cases are not possible as $d_C$ is a non-integer in (1) and $d_C<0$ in (2). 

{\bf (c):} $2-2g_\Sigma-k=-2p$.  In this case, $g_{|\Sigma/G|}=0$, $k=4$, and $g_\Sigma=12$. Without loss of
generality, we assume $k_1=2$, $k_2=k_3=1$. 

We observe that $d_C\geq 0$ and $d_C\in\Z$ imply that $\sum_{j=1}^4 (m_{j,1}+m_{j,2})=21$ for the rotation numbers in the formula of $d_C$. It follows immediately from Lemma 2.1 that $(1,2)$ cannot be realized as rotation numbers at any of the fixed points. On the other hand, it is easy to see that if for any $j$, $m_{j,1}+m_{j,2}\geq 5$, then $d_C<0$, except when the following is true, i.e., three of the rotation numbers $(m_{j,1}, m_{j,2})$ are either $(2,3)$ or $(1,4)$, with the fourth one being $(1,5)$. But this is not possible by Lemma 2.1. Consequently, one of the rotation numbers $(m_{j,1}, m_{j,2})$ must be $(1,3)$. Furthermore, we note that by Lemma 2.1, $(1,3)$ cannot occur at two different fixed points. There are $2$ scenarios to consider.

{\bf (i):} The rotation numbers $(1,3)$ occurs at $q_1$. By Lemma 2.1, $(1,4)$ cannot be realized as rotation numbers at any of the fixed points, and 
both $(2,3)$ and $(1,5)$ cannot occur at $q_1$. It follows easily that if the rotation numbers associated to the other local branch at $q_1$ is not $(1,3)$,
then the sum of the rotation numbers from $q_1$ is at least $11$, so that the sum of the rotation numbers from $q_2$ and $q_3$ is at most $10$. It is easy
to see that this is not possible by Lemma 2.1. Thus the rotation numbers associated to both local branches at $q_1$ must be $(1,3)$. This implies that
the sum of the $\delta$-invariants at $q_2,q_3$ must be $13$, as $g_\Sigma=12$. Moreover, the sum of the rotation numbers $m_{j,1}+m_{j,2}$ from $q_2$ and $q_3$ 
must be $21-8=13$. To proceed further, we note that by Lemma 2.1, $(3,4)$ cannot be realized as rotation numbers at any of the fixed points. With this
understood, it follows easily that if the $\delta$-invariant at $q_2$ or $q_3$ is not $0$ or $1$, the sum of the rotation numbers from $q_2$ and $q_3$ 
is at least $14$, which is a contradiction. If the $\delta$-invariant at $q_2$ or $q_3$ is $0$ or $1$, the $\delta$-invariant at the other fixed point must be
$13$ or $12$. One can show that the sum of the rotation numbers from $q_2$ and $q_3$ is also at least $14$, a contradiction. 

{\bf (ii):} The rotation numbers $(1,3)$ occurs at $q_2$ or $q_3$. Without loss of generality, assume $(1,3)$ occurs at $q_2$. Then $C$ must be singular at $q_2$ with a $\delta$-invariant $13$ whose characteristic sequence is $(3,14)$. It follows easily that $C$ must be embedded at $q_3$, with rotation numbers at $q_3$ being either $(1,5)$ or $(1,8)$. On the other hand, the genus formula implies that $C_{1,1}\cdot C_{2,1}+\delta_{1,1}+\delta_{2,1}=3$. It follows easily that the only possibility is that one of $C_{1,1},C_{2,1}$ is singular with a $\delta$-invariant $1$ and the other is embedded intersecting transversely. Correspondingly, the rotation numbers at $q_1$ must be $(2,3)$ and $(1,5)$. But this contradicts Lemma 2.1, since $(1,5)$ or $(1,8)$ is assumed to
be rotation numbers at $q_3$ already. Hence case (c) is eliminated.

{\bf (d):} $2-2g_\Sigma-k=-p$. In this case, $g_{|\Sigma/G|}=0$ and $k=3$. One always has $d_C<0$.
This completes the analysis for the case of $p=13$.

\vspace{2mm}

{\bf (6):} $p=19$. By Lemmas 2.2 and 2.6, $c_1(2K)=C$, where $C$ is a degree $6$ $J$-holomorphic curve containing all $3$ fixed points. We note that by Lemma 2.7, if $C$ is embedded near a fixed point, the associated rotation numbers at the fixed point must be either $(1,8)$ or $(1,12)$. With this understood, there are three cases we need to examine: (a): $2-2g_\Sigma-k=-3p$, (b): $2-2g_\Sigma-k=-2p$, and (c): $2-2g_\Sigma-k=-p$. See Lemma 2.5.

{\bf (a):} $2-2g_\Sigma-k=-3p$. In this case, either $k=3$ or $k=5$.

Let $k=3$ first. In this case, note that $g_{|\Sigma/G|}=1$. Moreover, $C$ must be embedded by the adjunction inequality, and consequently, the rotation numbers at the $3$ fixed points must be either $(1,12)$ or $(1,8)$ by Lemma 2.7. With this understood, it follows easily that $d_C\geq 0$ and $d_C\in \Z$ imply that the rotation numbers at the $3$ fixed points $q_1,q_2,q_3$ must be all $(1,12)$. Furthermore, one can check that the Global Quotient Criterion is not violated if and only if the local representations of the $\Z_p$-action are given by case (i) in Lemma 2.1. More
concretely, let $g_1,g_2,g_3$ be the generators, each with rotation numbers $(1,12)$, and we write $g_j=g_1^{c_j}$. Then if we apply Lemma 2.1(i) to $g_1$ with $(a,b)=(1,12)$, we get $(18,11)$, $(7,8)$ as the rotation numbers at the other two fixed points. From this it follows easily that $c_2=7$, $c_3=11$, and $c_1+c_2+c_3=19=0\pmod{19}$ is satisfied. This case is stated as (6) in Theorem 1.3.

Next, let $k=5$. In this case, $g_\Sigma=27$, which clearly violates the following adjunction inequality 
$$
g_\Sigma\leq \frac{1}{2}(d^2+3d+2)-\frac{1}{2}\sum_{i=1}^3 k_i(k_i-1)-\sum_{i=1}^3 \delta_i. 
$$

{\bf (b):} $2-2g_\Sigma-k=-2p$. In this case, $k=4$ and 
$g_{|\Sigma/G|}=0$. Furthermore, $g_\Sigma=18$. For a closer examination, we assume without loss of generality
that $k_1=2$, and $k_2=k_3=1$. Let $C_{1,1}, C_{2,1}$ be the two local branches of $C$ at the fixed point $q_1$, 
let $\delta_{1,1}, \delta_{2,1}$ be the corresponding $\delta$-invariants of the singularities, and let 
$\delta_2,\delta_3$ be the $\delta$-invariants of $C$ at the fixed points $q_2,q_3$. 
Then it follows easily that
$$
C_{1,1}\cdot C_{2,1}+\delta_{1,1}+\delta_{2,1}+\delta_2+\delta_3=10.
$$
Note that $\delta_{1,1},\delta_{2,1},\delta_2,\delta_3\leq 9$. The rotation numbers at $q_2,q_3$, which are subject to constraints from Lemma 2.7, are one from the following list; in particular, note that 
$\delta$-invariant $\neq 1, 3$ and $5$.

\vspace{3mm}

\centerline{
\begin{tabular}{lllllll}
\hline \vline\hspace{1mm}
$\delta$-invariant $\leq 9$ & \vline \hspace{1mm} rotation numbers  \hspace{5.8mm} \vline \\
\hline \vline\hspace{1mm}
0 &  \vline \hspace{1mm} $(1,8)$ or $(1,12)$ \hspace{9.7mm} \vline \\ 
\hline \vline\hspace{1mm}
2 &  \vline \hspace{1mm} $(2,5)$  \hspace{26.1mm} \vline \\ 
\hline \vline\hspace{1mm} 
4 &  \vline \hspace{1mm} $(3,5)$  \hspace{26.1mm} \vline\\
\hline \vline\hspace{1mm} 
6 &  \vline \hspace{1mm} $(4,5)$  \hspace{26.1mm} \vline\\
\hline \vline\hspace{1mm} 
7 &  \vline \hspace{1mm} $(2,15)$ or $(3,8)$ \hspace{9.7mm} \vline\\
\hline \vline\hspace{1mm} 
8 &  \vline \hspace{1mm} $(2,17)$ \hspace{24.3mm}  \vline\\ 
\hline \vline\hspace{1mm}
9 &  \vline \hspace{1mm} $(4,7)$ \hspace{26.1mm} \vline \\ \hline
\end{tabular}
}
\vspace{2mm}

\centerline{Table 5: possible $\delta$-invariants and rotation numbers at $q_2,q_3$, $p=19$, Case  (b).
}
\vspace{2mm}

With the preceding understood, there are $4$ scenarios, (i)-(iv), to consider. 

{\bf (i):} Both $\delta_{1,1},\delta_{2,1}$ are nonzero. One can show that $C_{1,1}\cdot C_{2,1}+\delta_{1,1}+\delta_{2,1}>10$, except for the following case where both $C_{1,1},C_{2,1}$ are singular at $q_1$ with a simple cusp singularity, and in this case, we have instead
$C_{1,1}\cdot C_{2,1}+\delta_{1,1}+\delta_{2,1}=8$. Moreover, it is clear that $C$ must be embedded near one of $q_2,q_3$ and has a cuspidal singularity at the other with $\delta$-invariant $2$. It follows easily that the rotation numbers must be $(2,3)$, $(2,3)$, $(1,8)$ (or $(1,12)$), and $(2,5)$ (cf. Table 5). But one can check that $d_C$ is a non-integer, which is a contradiction. 

{\bf (ii):} $C$ has a nodal singularity at $q_1$, i.e., $\delta_{1,1}=\delta_{2,1}=0$ and $C_{1,1}\cdot C_{2,1}=1$. To rule out this case, we note that $\delta_2+\delta_3=9$, so that either $(\delta_2,\delta_3)=(2,7)$ or 
$(\delta_2,\delta_3)=(0,9)$ by Table 5. Correspondingly, the rotation numbers at $q_2,q_3$ are either $(2,5)$, $(2,15)$ (or $(3,8)$), or $(1,8)$ (or $(1,12)$), $(4,7)$. It is straightforward to check that each of these possibilities is in contradiction to Lemma 2.1. 

{\bf (iii):} $\delta_{1,1}=\delta_{2,1}=0$ but $C_{1,1}\cdot C_{2,1}>1$. In this case, the two
local branches $C_{1,1}, C_{2,1}$ have the same tangent plane at $q_1$, thus they determined the same rotation numbers $(1,m)$, where $m=C_{1,1}\cdot C_{2,1}>1$. To proceed further, we use Lemma 2.7 to show that
$m=8$ must be true. To see this, observe that $int_D(q_1)=2$ for a generic $D$. Then Lemma 2.7 implies that
$$
(2+a_1)+(2+a_2)m=0 \pmod{19}
$$
for some non-negative integers $a_1,a_2$ such that $0<a_1+a_2\leq 2$. It follows easily that $m=8, 9, 12, 17$ or $18$. With $m=C_{1,1}\cdot C_{2,1}\leq 10$ and $\delta_2+\delta_3\neq 1$, one gets $m=8$, and $\delta_2=2$,
$\delta_3=0$. The rotation numbers at $q_2,q_3$ are $(2,5)$ and $(1,8)$ (or $(1,12)$) according to Table 5. 
But for any of the choices, one can check easily that the virtual dimension $d_C$ is always a non-integer. 

{\bf (iv):} $\delta_{1,1}>0$ but $\delta_{2,1}=0$. Since $\delta_{1,1}\leq 9$, the rotation numbers at $q_1$ that is associated to the local branch $C_{1,1}$ will be of the form $(\alpha,\beta_1)$ where $(\alpha,\beta_1)$ is the characteristic sequence of the singularity and the multiplicity 
$\alpha=2,3$ or $4$, such that  $C_{1,1}\cdot C_{2,1}=\beta_1$ if $C_{1,1}, C_{2,1}$ are tangent at $q_1$ and 
$C_{1,1}\cdot C_{2,1}=\alpha$ if otherwise. Furthermore, $(\alpha,\beta_1)$ is subject to a constraint from 
Lemma 2.7, i.e., if for any non-negative integers $a_1,a_2$ with $a_1+a_2>0$ such that 
$$
(2+a_1)\alpha +(2+a_2)\beta_1=0 \pmod{19},
$$
one always has $a_1+a_2>\alpha+1=int_D(q_1)$ (here $D$ is generic), then $(\alpha,\beta_1)$ can be eliminated. Finally, we point out that
the rotation numbers at $q_1$ that is associated to the local branch $C_{2,1}$ is of the form $(1,m)$ which can be
determined completely from $(\alpha,\beta_1)$ and the fact whether $C_{1,1}, C_{2,1}$ are tangent at $q_1$ or not. 
With this understood, we make a list of possible rotation numbers below where in each case the first two rotation numbers are associated with $C_{1,1}, C_{2,1}$ and the last two rotation numbers are at $q_2$, $q_3$ respectively:
\begin{itemize}
\item [{(1)}] $(2,3)$, $(1,7)$, $(1,8)$ (or $(1,12)$), $(2,15)$ (or $(3,8)$); 
\item [{(2)}] $(2,3)$, $(1,11)$, $(1,8)$ (or $(1,12)$), $(4,5)$;
\item [{(3)}] $(2,3)$, $(1,11)$, $(2,5)$, $(3,5)$;
\item [{(4)}] $(2,5)$, $(1,8)$, $(1,8)$ (or $(1,12)$), $(4,5)$;
\item [{(5)}] $(2,5)$, $(1,8)$, $(2,5)$, $(3,5)$;
\item [{(6)}] $(4,5)$, $(1,16)$, $(1,8)$ (or $(1,12)$), $(1,8)$ (or $(1,12)$);
\item [{(7)}] $(3,8)$, $(1,17)$, $(1,8)$ (or $(1,12)$), $(1,8)$ (or $(1,12)$);
\item [{(8)}] $(2,17)$, $(1,18)$, $(1,8)$ (or $(1,12)$), $(1,8)$ (or $(1,12)$);
\end{itemize}
All the cases can be eliminated by Lemma 2.1 except (5). To rule out (5), we use the fact that $d_C$ is a non-integer in this case. 

{\bf (c):} $2-2g_\Sigma-k=-p$, where $k=3$ and $g_{|\Sigma/G|}=0$ must be true. Note that $g_\Sigma=9$. It follows easily that either $C$ has $19$ nodal or simple cusp singularities in the complement of the fixed points, or has only cuspidal singularities at the fixed points with the sum of the $\delta$-invariants equaling $19$. To rule out this possibility, we simply observe that 
$$
d_C=-\frac{18}{19}+(g_{|\Sigma/G|}-1)+3-\sum_{j=1}^3 \frac{m_{j,1}+m_{j,2}}{19}<0.
$$
(Note that the possible $\delta$-invariants of a cuspidal singularity of $C$, subject to the constraint from Lemma 2.7, together with the corresponding rotation numbers $(m_{j,1},m_{j,2})$ at the fixed point, are listed in Table 5 if the 
$\delta$-invariant is less than $10$, where we can easily see that $m_{j,1}+m_{j,2}\geq 7$. If the $\delta$-invariant is outside this range, we can check each case individually.)

The proof of Theorem 1.3 is completed.

\section{Minimal resolution of the quotient orbifold}
Let $X$ denote the minimal symplectic resolution of the quotient orbifold $M/\Z_p$, where $M$ is a symplectic $\Q$-homology $\C\P^2$ from Theorem 1.3. It is clear that $X$ is a $\Q$-homology 
$\C\P^2\#N\overline{\C\P^2}$, and when $M$ is simply connected, so is $X$ (but not the other way around). 

First of all, to each of the $12$ cases of symplectic $\Z_p$-actions listed in Theorem 1.3, we associate a symplectic divisor with normal crossings which is canonically embedded in the minimal resolution $X$. After slightly perturbing it if necessary (cf. \cite{Gompf}), a symplectic neighborhood of the divisor is given by a symplectic plumbing according to a weighted graph (cf. \cite{McR, GS}), a complete list of which is given in Figure 1. Furthermore, we note that the symplectic surfaces in the symplectic divisor form a basis of $H^2(X,\Q)$, which allows us to express the canonical class $c_1(K_X)$ in terms of the associated symplectic divisor. This is then used to study the minimality and symplectic Kodaira dimension of the symplectic resolution $X$. Finally, we examine the prime order automorphisms of a fake projective plane (cf. \cite{CS, PY, Keum1, Keum2}) from the point of view of symplectic $\Z_p$-actions. 

\subsection{The symplectic divisor associated to each symplectic $\Z_p$-action}
To each of the $12$ cases in Theorem 1.3, with the invariant curve $C$ described therein, we associate a plumbing of symplectic surfaces which is canonically embedded in the minimal symplectic resolution $X$. More concretely, the configuration of symplectic surfaces consists of the exceptional divisors from the minimal resolution of the singularities of $M/\Z_p$ plus the proper transform of the quotient curve $C/\Z_p$, which connects the exceptional divisors from different singular points in $M/\Z_p$. (The key fact here is that the invariant curve $C$ contains all $3$ fixed points, so that the plumbing graph is connected.) In order to determine the plumbing graphs, we first resolve the singularities of the invariant curve $C$ from Theorem 1.3 by a sequence of successive equivariant blowing-ups at the relevant fixed points, then apply the following two lemmas, i.e., Lemmas 3.1 and 3.2,  to the quotient curves in the quotient orbifolds.

\begin{lemma}
Let $Z$ be a symplectic $4$-orbifold with only isolated singularities, and let $S$ be an embedded, $2$-dimensional symplectic suborbifold in $Z$, containing singularities $x_1,x_2,\cdots,x_k$. Let $g$ be the genus of the underlying Riemann surface of $S$, let $n_i$ be the order of the isotropy group at $x_i$ and $(1,b_i)$, where $0<b_i<n_i$ and $gcd(n_i,b_i)=1$, be the rotation numbers at $x_i$ where $b_i$ is the weight of the local action in the normal direction of $S$, and let $e\in\Q$ be the Euler number (or equivalently, the first Chern number) of the orbifold normal bundle of $S$ in $Z$. 
Then there is an $e_0\in\Z$ such that $e=e_0+\sum_{i=1}^k \frac{b_i}{n_i}$. Moreover, let $U$ be a regular neighborhood of $S$ in $Z$ and let $\tilde{U}$ be the pre-image of $U$ under the minimal symplectic resolution $\pi: \tilde{Z}\rightarrow Z$. Then $\tilde{U}$ is diffeomorphic to a plumbing of symplectic surfaces according to the weighted graph in Figure 2(a), where the weights $e_{i,j}$ are determined from the continued fractions 
$-n_i/b_i=[e_{i,1},e_{i,2},\cdots,e_{i,s_i} ]$, for $i=1,2,\cdots,k$. 
\end{lemma}

\begin{proof}
Since $S$ is an embedded $2$-dimensional symplectic suborbifold in $Z$, its orbifold normal bundle is well-defined, which is naturally an orbifold complex line bundle. With this understood, by Lemma 3.6 in \cite{C00}, we have $e=e_0+\sum_{i=1}^k \frac{b_i}{n_i}$ for some $e_0\in\Z$. It follows easily that
the boundary of the regular neighborhood $U$ of $S$ in $Z$ is the Seifert manifold $M(g; -e_0-k, (n_1,n_1-b_1), \cdots, (n_k,n_k-b_k))$ (cf. \cite{N}). On the other hand, $M(g; -e_0-k, (n_1,n_1-b_1), \cdots, (n_k,n_k-b_k))$ is also the boundary of a plumbing according to the weighted graph in Figure 2(a) (e.g. see \cite{NR}, Theorem 5.1). The lemma follows immediately. 

\end{proof}

\begin{picture}(500,100)
\put(10,60){\circle*{5}}
\put(40,90){\circle*{5}}
\put(40,75){\circle*{5}}
\put(40,45){\circle*{5}}
\put(40,68){\circle*{2}}
\put(40,63){\circle*{2}}
\put(40,58){\circle*{2}}
\put(10,60){\line(1,1){30}}
\put(10,60){\line(2,1){30}}
\put(10,60){\line(2,-1){30}}
\put(70,90){\circle*{5}}
\put(85,90){\circle*{2}}
\put(90,90){\circle*{2}}
\put(95,90){\circle*{2}}
\put(110,90){\circle*{5}}
\put(40,90){\line(1,0){30}}
\put(110,90){\line(-1,0){10}}
\put(70,90){\line(1,0){10}}
\put(70,75){\circle*{5}}
\put(85,75){\circle*{2}}
\put(90,75){\circle*{2}}
\put(95,75){\circle*{2}}
\put(110,75){\circle*{5}}
\put(40,75){\line(1,0){30}}
\put(110,75){\line(-1,0){10}}
\put(70,75){\line(1,0){10}}
\put(70,45){\circle*{5}}
\put(85,45){\circle*{2}}
\put(90,45){\circle*{2}}
\put(95,45){\circle*{2}}
\put(110,45){\circle*{5}}
\put(40,45){\line(1,0){30}}
\put(110,45){\line(-1,0){10}}
\put(70,45){\line(1,0){10}}
\put(15,55){\makebox(0,0)[tr]{$[g], e_0$}}
\put(40,40){\makebox(0,0)[tl]{$e_{k,1}$}}
\put(70,40){\makebox(0,0)[tl]{$e_{k,2}$}}
\put(110,40){\makebox(0,0)[tl]{$e_{k,s_k}$}}
\put(40,93){\makebox(0,0)[bl]{$e_{1,1}$}}
\put(70,93){\makebox(0,0)[bl]{$e_{1,2}$}}
\put(110,93){\makebox(0,0)[bl]{$e_{1,s_1}$}}
\put(40,78){\makebox(0,0)[bl]{$e_{2,1}$}}
\put(70,78){\makebox(0,0)[bl]{$e_{2,2}$}}
\put(110,78){\makebox(0,0)[bl]{$e_{2,s_2}$}}
\put(70,30){\makebox(0,0)[tl]{$(a)$}}
\put(180,60){\circle*{5}}
\put(195,60){\circle*{2}}
\put(200,60){\circle*{2}}
\put(205,60){\circle*{2}}
\put(220,60){\circle*{5}}
\put(250,60){\circle*{5}}
\put(280,60){\circle*{5}}
\put(310,60){\circle*{5}}
\put(340,60){\circle*{5}}
\put(355,60){\circle*{2}}
\put(360,60){\circle*{2}}
\put(365,60){\circle*{2}}
\put(380,60){\circle*{5}}
\put(180,60){\line(1,0){10}}
\put(220,60){\line(-1,0){10}}
\put(220,60){\line(1,0){30}}
\put(250,60){\line(1,0){30}}
\put(280,60){\line(1,0){30}}
\put(310,60){\line(1,0){30}}
\put(340,60){\line(1,0){10}}
\put(380,60){\line(-1,0){10}}
\put(180,55){\makebox(0,0)[tl]{$e_{1,s_1}$}}
\put(220,55){\makebox(0,0)[tl]{$e_{1,2}$}}
\put(250,55){\makebox(0,0)[tl]{$e_{1,1}$}}
\put(280,55){\makebox(0,0)[tl]{$-1$}}
\put(380,55){\makebox(0,0)[tl]{$e_{2,s_2}$}}
\put(340,55){\makebox(0,0)[tl]{$e_{2,2}$}}
\put(310,55){\makebox(0,0)[tl]{$e_{2,1}$}}
\put(280,30){\makebox(0,0)[tl]{$(b)$}}
\put(190,10){\makebox(0,0)[t]{Figure 2: Weighted graphs in Lemmas 3.1 and 3.2}}
\end{picture}

The next lemma describes the local models of an equivariant blowing-up at an isolated fixed point of a symplectic
$\Z_p$-action.

\begin{lemma}
Consider a linear $\Z_p$-action on $\C^2$ for an odd prime $p$, with rotation numbers $(a,b)$ at the origin with respect to some generator $g\in\Z_p$, and let $B\subset \C^2$ be a $4$-ball centered at origin equipped with some $\Z_p$-invariant K\"{a}hler form. Let $\pi: \tilde{B}\rightarrow B$ be the complex blowing-up at the origin and denote by $E\subset \tilde{B}$ the exceptional $(-1)$-sphere. Note that the $\Z_p$-action lifts to a holomorphic $\Z_p$-action on $\tilde{B}$, which is symplectic with respect to some invariant symplectic form on $\tilde{B}$. Let $U:=\tilde{B}/\Z_p$ denote the corresponding symplectic $4$-orbifold.
\begin{itemize}
\item [{(1)}] Assume $a<b$. Then $U$ has $2$ isolated singularities, and its minimal symplectic resolution 
$\tilde{U}$ is diffeomorphic to a plumbing of symplectic surfaces according to the weighted graph in Figure 2(b), where $0<\tilde{a},\tilde{b}<p$ are determined as follows: let $s\in\Z$ such that $(b-a)s=1\pmod{p}$, then $\tilde{a}=as\pmod{p}$ and $\tilde{b}=-bs\pmod{p}$, and $-p/\tilde{a}=[e_{1,1},e_{1,2}, \cdots, e_{1,s_1}]$,  
$-p/\tilde{b}=[e_{2,1},e_{2,2}, \cdots, e_{2,s_2}]$ as continued fractions. 
\item [{(2)}] Assume $a=b=1$. Then $U$ is a smooth manifold with a symplectic structure such that the exceptional $(-1)$-sphere $E\subset \tilde{B}$ descends to an embedded symplectic $(-p)$-sphere in $U$. Moreover, for any embedded holomorphic disc $D\subset B$ through the origin, let $\tilde{D}$ be the proper transform in $\tilde{B}$. Then $\tilde{D}$ descends to an embedded symplectic disc in $U$ intersecting transversely and positively with the symplectic $(-p)$-sphere in $U$ at a single point. 
\end{itemize}
\end{lemma}

\begin{proof}
Let $z_1,z_2$ be the complex coordinates on $\C^2$, where the action of $g\in\Z_p$ is given by $g\cdot (z_1,z_2)=(\mu_p^a z_1, \mu_p^b z_2)$,
$\mu_p:=\exp(2\pi i/p)$. Then the induced action of $g$ on the blowing-up $\tilde{B}$ is given by $g\cdot ((z_1,z_2), [z_1:z_2])=
((\mu_p^a z_1, \mu_p^b z_2),[\mu_p^a z_1:\mu_p^b z_2])$. 

First, assume $a<b$. Then the $\Z_p$-action on $\tilde{B}$ has $2$ isolated fixed points, both lying on the exceptional $(-1)$-sphere, i.e., $[1:0]$ and $[0:1]$ in $E$. Near the fixed point $[1:0]$, the local chart of  $\tilde{B}$ is given by $((z_1,z_2), [z_1:z_2])\rightarrow (z_1,z_2/z_1)$, which implies that the rotation numbers at $[1:0]$ is $(a,b-a)$ with respect to $g$. Similarly, the rotation numbers at the fixed point $[0:1]$ is $(b,p+a-b)$ with respect to $g$. With this understood, we apply Lemma 3.1 to the $2$-dimensional symplectic suborbifold $S:=E/\Z_p$ in the symplectic $4$-orbifold $U:=\tilde{B}/\Z_p$. It is easy to see that $S$ contains $2$ isolated singularities with rotation numbers $(1,b_1)$, $(1,b_2)$, where $b_1=\tilde{a}$, $b_2=\tilde{b}$. To finish up, we observe that 
$\tilde{a}+\tilde{b}=s(a-b)=-1 \pmod{p}$, which implies easily that $\tilde{a}+\tilde{b}=p-1$. On the other hand,
it is clear that the Euler number $e$ of the orbifold normal bundle of $E/\Z_p$ in $\tilde{B}/\Z_p$ is $-1/p$. It follows immediately that $e_0=-1$, and by Lemma 3.1, the minimal symplectic resolution $\tilde{U}$ is diffeomorphic to a plumbing of symplectic surfaces according to the weighted graph in Figure 2(b). 

Next, assume $a=b=1$. Then the $\Z_p$-action on $\tilde{B}$ fixes the exceptional $(-1)$-sphere $E$ and has no other fixed points. Consequently, the underlying space of the symplectic $4$-orbifold $U:=\tilde{B}/\Z_p$ is a smooth manifold, and according to \cite{C11}, one can de-singularize the orbifold symplectic structure on 
$U:=\tilde{B}/\Z_p$, so that $U$ becomes a symplectic $4$-manifold. Furthermore, the singular set $E/\Z_p$ in $U$ becomes an embedded symplectic $(-p)$-sphere in the symplectic $4$-manifold $U$. Other claims in part (2) follows from the construction in \cite{C11} easily as well. The proof of Lemma 3.2 is completed.

\end{proof}

With these preparations, we now state the main theorem of this section.

\begin{theorem}
(1) To each group action listed in Theorem 1.3, there is an associated plumbing of symplectic surfaces in the minimal symplectic resolution $X$, whose plumbing graph is listed in Figure 1. The symplectic surfaces in the plumbing consist of the exceptional divisors from the minimal resolutions of the singularities in $M/\Z_p$ plus the proper transform of the quotient of the invariant curve $C/\Z_p$ which connects them. The graphs are labelled either as 
$\Gamma_p$ or $\Gamma_{p,i}$, according to the following convention: $p$ is the order of the group
action and the index $i$ indicates the subcase in the listing in Theorem 1.3, e.g., $\Gamma_{19}$ is the graph for the $\Z_{19}$-action in Case (6), and $\Gamma_{7,2}$ is the graph for the $\Z_7$-action listed in Case (3)(ii). 

(2) The symplectic surfaces form a basis of $H^2(X,\Q)$, so the canonical class $c_1(K_X)$ can be written uniquely as a linear combination of the surfaces, which is given below. For each of the cases, listed in the same way as in Theorem 1.3, the surfaces in the formula of $c_1(K_X)$ are denoted by $F_0$ and $F_{i,j}$, where $F_0$ is the proper transform of $C/\Z_p$, and the index $i=1,2,3$ is the label for the three subgraphs in the plumbing graph which are the linear chains corresponding to the minimal resolution of the three singularities in the quotient 
orbifold $M/\Z_p$, ordered clockwise in the plumbing graph, and the index $j$ is the label for the surfaces in each linear chain, ordered either from left to right or from top to bottom, except when the linear chain is not ``straight" in the plumbing graph, like in $\Gamma_5$ and $\Gamma_{7,2}$, it is ordered clockwise.
\begin{itemize}
\item [{(1)}] $c_1(K_X)=\frac{1}{2}(F_0+F_{1,1}+F_{1,2}+F_{2,1}+F_{2,2}+F_{3,1}+F_{3,2})$.
\item [{(2)}] $c_1(K_X)=\frac{1}{2}(F_0+F_{2,1}+F_{2,2}+F_{2,3}+F_{2,4})$.
\item [{(3)}]
\item [{(i)}] $c_1(K_X)=F_0$.
\item [{(ii)}] $c_1(K_X)=\frac{1}{2}(F_0-F_{1,1}+F_{2,1}+F_{2,2}+F_{2,3}+F_{2,4}+F_{2,5}+F_{2,6}-F_{3,2})$.
\item [{(iii)}] 
$c_1(K_X)=\frac{1}{14}(7F_0-F_{1,1}+4F_{1,2}+2F_{1,3}+F_{2,1}+2F_{2,2}-4F_{2,3}-2F_{3,1}+F_{3,2}+4F_{3,3})$.
\item [{(iv)}] \begin{eqnarray*}
c_1(K_X) & = & \frac{1}{14}(7F_0-3F_{1,1}+2F_{1,2}+3F_{2,1}+ 6F_{2,2}+9F_{2,3}+12F_{2,4}\\
& &  \hspace{6mm}     + 8F_{2,5}+4F_{2,6}-4F_{3,1}-2F_{3,2}).\\
\end{eqnarray*}
\item [{(v)}] $c_1(K_X)=\frac{1}{2}F_0$.
\item [{(4)}]
\item [{(i)}] \begin{eqnarray*}
c_1(K_X) & = & \frac{1}{22}(33F_0+F_{1,1}+2F_{1,2}+3F_{1,3}+ 4F_{1,4}+5F_{1,5}+2F_{2,1}\\
& &  \hspace{6mm}     + 4F_{2,2}+6F_{2,3}+3F_{2,4}-11F_{3,2}).\\
\end{eqnarray*} 
\item [{(ii)}] $c_1(K_X)=\frac{1}{2}(F_0-F_{3,1}-F_{3,2})$.
\item [{(5)}]
\item [{(i)}] \begin{eqnarray*}
c_1(K_X) & = & \frac{1}{26}(39F_0+F_{1,1}+2F_{1,2}+3F_{1,3}+ 4F_{1,4}+5F_{1,5}+6F_{1,6}-3F_{2,1}\\
& &  \hspace{6mm}     -6F_{2,2}-9F_{2,3}+16F_{3,1}-7F_{3,2}-11F_{3,3}).\\
\end{eqnarray*} 
\item [{(ii)}] $c_1(K_X)=\frac{1}{2}(F_0-F_{1,1}-F_{2,4}-F_{3,1})$.
\item [{(6)}] $c_1(K_X)=\frac{1}{2}(F_0-F_{1,1}-F_{1,2}-F_{1,3}-F_{2,1}-F_{2,2}-F_{2,3}-F_{3,2}-F_{3,3}-F_{3,4})$.
\end{itemize}
\end{theorem}

\begin{proof}
The plumbing of symplectic surfaces is obtained by a repeated application of Lemmas 3.1 and 3.2, and then blow down all the obvious symplectic $(-1)$-spheres created in the process. The formula for $c_1(K_X)$ is obtained via the self-intersection matrix of the plumbing graph and the intersection number of $c_1(K_X)$ with each symplectic surface in the plumbing, which can be determined by the adjunction formula. 

\vspace{2mm}

{\it Case (1):} As described in Theorem 1.3(1), for $p=3$ the invariant curve $C$ is of degree $6$ and has a nodal singularity at each of the $3$ fixed points. Moreover, the rotation numbers at each fixed point is $(1,2)$, and from the proof of Theorem 1.3, we recall that $g_{|\Sigma/G|}=7$, which is the genus of $C/\Z_p$. With this understood, we blow up at each of the $3$ fixed points to obtain a symplectic $4$-manifold $\tilde{M}$ with a lifted $\Z_p$-action on it (for $p=3$). It is clear that the proper transform $\tilde{C}$ of $C$ in $\tilde{M}$ is embedded and is invariant under the $\Z_p$-action, so that $\tilde{C}/\Z_p$ is an embedded $2$-dimensional suborbifold in the symplectic $4$-orbifold $\tilde{M}/\Z_p$. On the other hand, by Lemma 3.2, the $\Z_p$-action on $\tilde{M}$ has $6$ isolated fixed points contained in the $3$ exceptional $(-1)$-spheres in $\tilde{M}$, with each containing two of the fixed points. Moreover, it is easy to see that the rotation numbers at each of the $6$ fixed points on $\tilde{M}$ is $(1,1)$. Now we note that the self-intersection $\tilde{C}\cdot \tilde{C}=6^2-3\cdot 2^2=24$, so that the Euler number of the orbifold normal bundle of $\tilde{C}/\Z_p$ in $\tilde{M}/\Z_p$ equals $24/p=8$. It follows easily from 
Lemma 3.1 that the integer $e_0=e-\sum_{i=1}^6\frac{b_i}{p}=8-6\cdot \frac{1}{3}=6$ as each $b_i=1$. 
Now we take the minimal resolution of $\tilde{M}/\Z_p$, where each singular point is replaced by a
symplectic $(-3)$-sphere. Blow down the $3$ symplectic $(-1)$-spheres in it, one from each of the exceptional $(-1)$-spheres of the blowing-up at each of the $3$ fixed points of the $\Z_p$-action on $M$ (see Figure 2(b)), we obtain the minimal resolution of $M/\Z_p$, with a plumbing of symplectic surfaces which consists of $3$ pairs of $(-2)$-spheres from the minimal resolution of $M/\Z_p$, plus the proper transform of $C/\Z_p$ which connects them. It is clear that the plumbing graph is given by the weighted graph $\Gamma_3$ in Figure 1. Finally, the formula for $c_1(K_X)$ follows by a direct verification. 

\vspace{2mm}

{\it Case (2):} In this case, the invariant curve $C$ is of degree $6$ and has a nodal singularity at one of the fixed points, and two embedded local branches with a tangency of order $2$ at each of the other two fixed points. Moreover, the rotation numbers are $(1,4)$, $(1,2)$, $(1,2)$ respectively, and we have $g_{|\Sigma/G|}=3$. With this understood, we blow up once at the nodal singularity, and blow up successively twice at each of the other two fixed points so that the proper transform $\tilde{C}$ of $C$ becomes embedded. If we let $E$ be the exceptional $(-1)$-sphere of the blow-up at the nodal singularity, and denote by $E_1,E_2$ the exceptional divisors of the successive blow-up at each of the other two fixed points (where $E_1$ is a $(-2)$-sphere), then it is easy to see from Lemma 3.2 that each $E_2$ is a fixed component under the lifted $\Z_p$-action, each $E_1$ contains an isolated fixed point of rotation numbers $(1,3)$, and $E$ contains two isolated fixed points of rotation numbers $(1,3)$. It follows easily that after taking the quotient, each $E_2$ descends to a symplectic $(-5)$-sphere, and $\tilde{C}/\Z_p$ is an embedded $2$-suborbifold with two singular points, intersecting each of the symplectic $(-5)$-sphere transversely at two points. Moreover, the two singular points 
of $\tilde{C}/\Z_p$ are contained in $E/\Z_p$, with normal weight $b_i=3$. The Euler number of the orbifold normal bundle of $\tilde{C}/\Z_p$ is 
$$
e=\tilde{C}\cdot \tilde{C}/p=(6^2-5\cdot 2^2)/5=16/5,
$$ 
so that $e_0=e-\sum_{i=}^2 \frac{b_i}{p}=16/5-2\cdot \frac{3}{5}=2$. On the other hand, it is easy to see that $E_1/\Z_p$ is an embedded $2$-suborbifold with one singular point, with $e=-2/5$ so that $e_0=e-b/p=-2/5-3/5=-1$. Now we take the minimal resolution and apply Lemma 3.1, and then blow down successively $5$ symplectic $(-1)$-spheres, one of which is the proper transform of $E/\Z_p$ and two from each $E_1/\Z_p$. As a result, we obtain a plumbing of symplectic surfaces in the minimal resolution $X$ of $M/\Z_p$. It is easy to see that the corresponding plumbing graph is given by $\Gamma_5$ in Figure 1, where after the blowing-down, each of the $(-5)$-sphere becomes a $(-3)$-sphere intersecting the proper transform of $C/\Z_p$ transversely at two points. Finally, the formula for $c_1(K_X)$ follows by a direct verification. 

\vspace{2mm}

{\it Case (3)(i):} In this case, the invariant curve $C$ is of degree $3$, embedded, and from the proof of Theorem 1.3 we have $g_{|\Sigma/G|}=1$, with rotation numbers at each of the $3$ fixed points being $(1,3)$. It follows easily that $C/\Z_p$ is embedded in $M/\Z_p$, with Euler number of the orbifold normal bundle $e=3^2/p=9/7$, and consequently, 
$e_0=e-\sum_{i=1}^3 \frac{b_i}{p}=9/7-3\cdot \frac{3}{7}=0$. Now it is easy to see from Lemma 3.1
that the proper transform of $C/\Z_p$ together with the exceptional divisors form a symplectic divisor with normal crossings in the minimal symplectic resolution $X$ of $M/\Z_p$, with a plumbing graph given by $\Gamma_{7,1}$ in Figure 1. The formula for $c_1(K_X)$ is straightforward. 

\vspace{2mm}

{\it Case (3)(ii):} In this case, the invariant curve $C$ is of degree $6$, embedded except at one of the fixed points where it has a nodal singularity. Moreover, from the proof of Theorem 1.3 we have 
$g_{|\Sigma/G|}=3$, with the rotation numbers being $(1,6)$ at the nodal singularity and $(1,4)$ at the other $2$ fixed points. Now we blow up at the nodal singularity so that the proper transform $\tilde{C}$ of $C$ is embedded. It follows easily that the $2$-suborbifold $\tilde{C}/\Z_p$ in $\tilde{M}/\Z_p$ has an orbifold normal bundle with Euler number $e=\tilde{C}\cdot \tilde{C}/p=(6^2-2^2)/7=32/7$. Moreover, it has $4$ orbifold singularities, two from the nodal singularity of $C$, both with rotation numbers $(1,5)$ by Lemma 3.2 (with normal weight $b_i=5$), and the other two from the fixed points of the $\Z_p$-action on $C$, each with rotation numbers $(1,4)$ (with normal weight $b_i=4$). It follows easily that 
$$
e_0=e-\sum_{i=1}^4\frac{b_i}{p}=32/7-2\cdot \frac{5}{7}-2\cdot \frac{4}{7}=2.
$$
Now we take the minimal symplectic resolution of $\tilde{M}/\Z_p$ and blow down the symplectic 
$(-1)$-sphere resulted from the blown-up at the nodal singularity (see Lemma 3.2, Figure 2(b)). It is clear from Lemma 3.1 that the proper transform of $C/\Z_p$ together with the exceptional divisors form a symplectic divisor with normal crossings in the minimal symplectic resolution $X$ of $M/\Z_p$, with a plumbing graph given by $\Gamma_{7,2}$ in Figure 1. The formula for $c_1(K_X)$ follows by a direct verification. (Note that $F_{1,1}$ and $F_{3,2}$ are the $(-4)$-spheres in the plumbing, and they have a negative coefficient in the expression of $c_1(K_X)$.)

\vspace{2mm}

{\it Case (3)(iii):} In this case, the invariant curve $C$ is of degree $6$, has a simple cusp at one of the fixed points, a nodal singularity at another, and for the third fixed point, there are two local branches, one containing a simple cusp singularity and the other embedded intersecting the tangent plane of the first local branch transversely. For convenience, we denote the fixed points by $q_1,q_2,q_3$ in the above order. The other relevant information from the proof of Theorem 1.3 is $g_{|\Sigma/G|}=2$. The basic idea is that we blow up successively at each fixed point until the proper transform of $C$ becomes embedded, with normal crossings with the exceptional set, and then consider the lifted 
$\Z_p$-action and apply Lemma 3.2. After blowing down successively the obvious symplectic $(-1)$-spheres in the resolution of the quotient orbifold, we should arrive at a linear chain of symplectic spheres which corresponds to the exceptional divisors in the minimal resolution of the corresponding singularity in $M/\Z_p$. One of the key issues is to keep track in order to see where the proper transform of $C/\Z_p$ intersects the linear chain of spheres in the exceptional set. Another is to keep track of the loss of the self-intersection of $C$ during the process and of the relevant local Seifert invariants in order to determine the self-intersection number of the proper transform $F_0$ of 
$C/\Z_p$, i.e., the number $e_0$ in Lemma 3.1. Note that the genus of $F_0$ is $g_{|\Sigma/G|}=2$.

With the preceding understood, we consider the fixed point $q_1$ first, where $C$ has a simple cusp singularity. Let $z_1,z_2$ be local holomorphic coordinates, such that $C$ is parametrized by $z_1=z^2,z_2=z^3$ and the $\Z_p$-action is given by 
$\mu_p\cdot (z_1,z_2)=(\mu_p^2 z_1,\mu_p^3 z_2)$, where $\mu_p=:\exp(2\pi i/p)$ for $p=7$. We blow up at the origin and let $E_1$ be the exceptional $(-1)$-sphere. Then by Lemma 3.2, the lifted 
$\Z_p$-action has two isolated fixed points contained on $E_1$, with rotation numbers $(1,2)$ and $(1,4)$ respectively, where the weight $1$ is in the tangent direction of $E_1$. We note that the proper transform of $C$ passes through the first fixed point, locally described by an equation $w_2=w_1^2$. We blow up at this fixed point and let $E_2$ be the exceptional $(-1)$-sphere, and continue to denote the proper transform of $E_1$ by the same notation, which is now a $(-2)$-sphere. It is easy to see that the lifted $\Z_p$-action has three isolated fixed points, one being the intersection of $E_1,E_2$ and the proper transform of $C$, with rotation numbers $(1,1)$, and for the other two fixed points, one is contained in $E_1$ with rotation numbers $(1,4)$, and the other is in $E_2$ with rotation numbers $(1,5)$. With this understood, it is a simple application of Lemma 3.2 and Lemma 3.1, that after blowing down successively all the obvious $(-1)$-spheres, we get a linear chain of three spheres, in the order of 
a $(-2)$-sphere,  a $(-2)$-sphere, and a $(-3)$-sphere, where the proper transform of $C/\Z_p$ intersects the middle $(-2)$-sphere transversely at one point. We should also note that the self-intersection number of $C$ is decreased by $5$ here, and the relevant local Seifert invariant for $C/\Z_p$ is $(7,1)$ (i.e. $b_i=1$). This linear chain is the middle (i.e., top) one in the weighted graph $\Gamma_{7,3}$ in Figure 1.

Next, we consider the fixed point $q_2$, where $C$ has a nodal singularity, with rotation numbers $(1,3)$ and $(1,5)$ at the two local branches. We blow up at the fixed point and let $E$ be the exceptional $(-1)$-sphere. Then by Lemma 3.2, the lifted $\Z_p$-action has two isolated fixed points, all contained in $E$, with rotation numbers $(1,2)$, $(1,4)$ respectively, where the proper transform of $C$ intersects $E$ transversely. It follows easily that, after taking the quotient, resolving the singularities and then blowing down successively all the obvious $(-1)$-spheres, we get
a linear chain of a $(-2)$-sphere,  a $(-2)$-sphere, and  a $(-3)$-sphere, such that the proper transform of $C/\Z_p$ intersects the $(-2)$-sphere and $(-3)$-sphere at the end of the linear chain transversely at one point. On the other hand, we note that here the self-intersection number of $C$ is decreased by $4$, and the relevant local Seifert invariants are $(7,2)$, $(7,4)$. This linear chain is the one on the right-hand side of the weighted graph $\Gamma_{7,3}$ in Figure 1.

Finally, we note that the discussion for the fixed point $q_3$ where $C$ has two local branches, one containing a simple cusp singularity and the other embedded, is, for most part, the same as the case for the fixed point $q_1$,
except that the part of the proper transform of $C/\Z_p$ which corresponds to the embedded local branch also
intersects the $(-3)$-sphere in the linear chain. Moreover, we note that in this case, the self-intersection number of $C$ is decreased by $10$, and the relevant local Seifert invariants are $(7,1)$, $(7,2)$. The corresponding linear chain is the one on the left-hand side of the weighted graph $\Gamma_{7,3}$ in Figure 1.

To finish up, it remains to determine the self-intersection number of the proper transform $F_0$ of $C/\Z_p$, which is computed through the number $e_0$ in Lemma 3.1,
$$
e_0=e-\sum_{i=1}^5\frac{b_i}{p}=\frac{1}{7}(6^2-5-4-10)-\frac{1}{7}-\frac{2}{7}-\frac{4}{7}-\frac{1}{7}-\frac{2}{7}=1.
$$
Finally, the formula for $c_1(K_X)$ follows by a direct verification.

\vspace{2mm}

{\it Case (3)(iv):} In this case, the invariant curve $C$ is of degree $6$, and has three embedded local branches at one of the fixed points, two embedded local branches at another, and at the third fixed point, it has a cuspidal singularity. For convenience we denote them by $q_1,q_2,q_3$ respectively. We also note that $g_{|\Sigma/G|}=1$
in this case.

Consider the fixed point $q_2$ first, where the two local branches of  $C$ are tangent of order $2$. We blow up
successively twice and let $E_1,E_2$ be the exceptional set, where $E_1$ is a $(-2)$-sphere and $E_2$ is a
$(-1)$-sphere. Then it is easy to see that the lifted $\Z_p$-action fixes $E_2$ and has an isolated fixed point
contained in $E_1$, which has rotation numbers $(1,5)$. Taking the quotient, $E_2$ descends to a symplectic 
$(-7)$-sphere which intersects the proper transform of $C/\Z_p$ transversely at two points and the proper
transform of $E_1/\Z_p$ transversely at one point (cf. Lemma 3.2(2)). Resolving the singularity on $E_1/\Z_p$ and blowing down successively all the obvious $(-1)$-spheres, we get a linear chain of a $(-4)$-sphere and $(-2)$-sphere. The proper transform of $C/\Z_p$ intersects with the $(-4)$-sphere, transversely at two points. This is exhibited in the weighted graph $\Gamma_{7,4}$ in Figure 1 as the linear chain on the right-hand side. 

The discussion at the fixed point $q_1$ is the same for most part except that, in this case the proper transform of $C/\Z_p$ also intersects the $(-2)$-sphere in the linear chain because there is a third embedded local branch
of $C$ which is transverse to the other two branches. This case is exhibited in the weighted graph
$\Gamma_{7,4}$ in Figure 1 as the linear chain on the left-hand side. 

Finally, we examine the situation at the fixed point $q_3$, where $C$ is locally given by $z_1=z^3$, $z_2=z^4$
in a holomorphic coordinates. We blow up at the origin and let $E$ be the exceptional $(-1)$-sphere. Then
note that the lifted $\Z_p$-action has two isolated fixed points contained in $E$, where at the first one, the proper transform of $C$ is given by the equation $w_2=w_1^3$ and $E$ is given by $w_2=0$, while at the other fixed point, the rotation numbers is $(1,3)$. With this understood, it is convenient to separate out the discussion for the first fixed point in the following

\vspace{2mm}

{\bf A local model:} Consider two embedded, $\Z_p$-invariant holomorphic discs $C_1,C_2$ intersecting at the origin with a tangency of order $3$. If we blow up at the origin and let $E_1$ be the exceptional $(-1)$-sphere, then the lifted $\Z_p$-action has two isolated fixed points contained in $E_1$, where at the first one, the proper transform 
of $C_1,C_2$ are given by the equation $w_2=0$, $w_2=w_1^2$ respectively,  and $E_1$ is given by $w_1=0$, 
while at the other fixed point on $E_1$, the rotation numbers is $(1,2)$. With this understood, we blow up at the
first fixed point and let $E_2$ be the exceptional $(-1)$-sphere. Then the lifted $\Z_p$-action has two isolated fixed points contained in $E_2$, one being the intersection of the proper transforms of $C_1,C_2$ with $E_2$, with rotation numbers $(1,1)$, and the other being the intersection of the proper transform of $E_1$ with $E_2$, with rotation numbers $(1,5)$. Now take the quotient and resolve the singularities, and then blow down successively all the obvious $(-1)$-spheres. We get a linear chain of a $(-3)$-sphere, $(-2)$-sphere, and $(-2)$-sphere. Moreover, the proper transform of $C_1/\Z_p,C_2/\Z_p$ intersects the $(-3)$-sphere transversely at one point. 

\vspace{2mm}

Now applying the above local model to the proper transform of $C$ and the exceptional $(-1)$-sphere $E$, as
$C_1,C_2$, it follows easily that the proper transforms of $C/\Z_p$ and $E/\Z_p$ intersect the $(-3)$-sphere
in a linear chain of $(-3)$-sphere, $(-2)$-sphere, and $(-2)$-sphere transversely at one point. Now resolve the
singularity on the proper transform of $E/\Z_p$. Then the proper transform of $E/\Z_p$ becomes a $(-1)$-sphere, intersecting a linear chain of a $(-3)$-sphere, $(-2)$-sphere, and $(-2)$-sphere from the resolution of the 
singularity. Blowing down the $(-1)$-sphere, we get a linear chain of six $(-2)$-spheres. Note that the proper transform of $C/\Z_p$ intersects the fourth $(-2)$-sphere in the linear chain. This is depicted as the linear chain in the middle (i.e., the top) of the weighted graph $\Gamma_{7,4}$ in Figure 1.

It remains to determine the self-intersection of $F_0$: at $q_2$, the self-intersection of $C$ decreased by 
$2^2+2^2$, at $q_1$, it decreased by $3^2+2^2$, with a Seifert invariant $(7,3)$, and at $q_1$, it decreased by
$3^2+3$. It follows immediately that the self-intersection of $F_0$ equals 
$$
e_0=e-\frac{b}{p}=\frac{1}{7}(6^2-2^2-2^2-3^2-2^2-3^2-3)-\frac{3}{7}=0.
$$
Finally, the formula for $c_1(K_X)$ follows by a direct verification.

\vspace{2mm}

{\it Case (3)(v):} In this case, the invariant curve $C$ is of degree $6$, and has two embedded local branches 
at each of the fixed points, intersecting with a tangency of order $3$. Note that the local model in the previous 
Case (3)(iv) applies here, which yields a symplectic plumbing with the plumbing graph $\Gamma_{7,5}$ in
Figure 1. The genus of $F_0$ is $1$, as $g_{|\Sigma/G|}=1$ in this case. The self-intersection of $F_0$ is
$e_0=e=\frac{1}{7}(6^2-3\cdot (2^2+2^2+2^2))=0$. Finally, the formula for $c_1(K_X)$ follows by a direct verification.

\vspace{2mm}

{\it Case (4)(i):}  In this case, the invariant curve $C$ is of degree $2$, and has a simple cusp singularity at one of the fixed points, and are embedded near the other two fixed points. Moreover, from the proof of Theorem 1.3 we know that the rotation numbers at the fixed points are $(2,3)$, $(1,5)$, and $(1,4)$ respectively, and $g_{|\Sigma/G|}=0$. It follows immediately that the exceptional divisors in the minimal resolution $X$ which are associated with the last two fixed points are given by the two linear chains in the weighted graph $\Gamma_{11,1}$ in Figure 1, one on the left-hand side and one on the right-hand side (cf. Lemma 3.1). To understand the situation at the first fixed point, we note that $C$ is given locally by the equations $z_1=z^2$, $z_2=z^3$ in some holomorphic coordinates. We blow up at the origin successively twice and let $E_1,E_2$ be the exceptional divisors, where $E_1$ is a $(-2)$-sphere. Then the lifted $\Z_p$-action has three isolated fixed points, such that the one where the proper transforms of $C, E_1,E_2$ intersect has rotation numbers $(1,1)$, the one that is contained in $E_1$ has rotation numbers $(1,8)$, and the one that is on $E_2$ has rotation numbers $(1,9)$ respectively (see Lemma 3.2). It follows easily that, after taking the quotient and resolving the singularities, and then blowing down
successively all the obvious $(-1)$-spheres, we get a linear chain of a $(-2)$-sphere,  a $(-2)$-sphere, a  $(-3)$-sphere, and a $(-2)$-sphere, where the proper transform of $C/\Z_p$ intersects the $(-3)$-sphere transversely at one point. This is depicted as the linear chain which is in the middle of the graph $\Gamma_{11,1}$ in Figure 1. It remains to determine the self-intersection of $F_0$, which can be easily seen to be $-1$. Finally, the formula for $c_1(K_X)$ follows by a direct verification.

\vspace{2mm}

{\it Case (4)(ii):}  In this case, the invariant curve $C$ is of degree $6$, and is embedded at one of the fixed points, has two embedded local branches at another, which are tangent of order $5$, and has two tangent local branches at the third fixed point, where each local branch contains a simple cusp singularity. The discussion at the fixed point where $C$ has a simple cusp singularity is the same as in the previous case, i.e., Case (4)(i), except that in the present case, $C$ has two local branches at the fixed point. Since the two local branches have the same tangent plane, in the resolution the proper transform of $C/\Z_p$ will intersect the $(-3)$-sphere transversely at two points, corresponding to the two local branches of $C$. The corresponding linear chain is depicted in the weighted graph 
$\Gamma_{11,2}$ in Figure 1 as the one in the middle. In the same way, the situation at the fixed point where $C$ has two embedded branches which are tangent of order $5$ is the same as at the fixed point where $C$ is embedded with rotation numbers $(1,5)$ in the previous case, Case (4)(i), and the corresponding linear chain is depicted in $\Gamma_{11,2}$ as the one on the left-hand side. Lastly, the description at the remaining fixed point where $C$ is embedded with rotation numbers $(1,3)$ follows straightforwardly from Lemma 3.1. Finally, the formula for $c_1(K_X)$ is by a direct verification.

\vspace{2mm}

{\it Case (5)(i):}  In this case, the invariant curve $C$ is of degree $2$, and is embedded. With the information on the rotation numbers at the fixed points, which are $(1,6)$, $(1,8)$, and $(1,3)$, and the fact that $g_{|\Sigma/G|}=0$ from the proof of Theorem 1.3, it is a straightforward application of Lemma 3.1. The formula for $c_1(K_X)$ is by a direct calculation.

\vspace{2mm}

{\it Case (5)(ii):}  In this case, the invariant curve $C$ is of degree $6$, and at each of the fixed points, $C$ has a cuspidal singularity with rotation numbers $(3,4)$. With this understood, to understand the symplectic plumbing in this case, it suffices to focus at any of the fixed points. Note that $C$ is locally given by the equations $z_1=z^3$, $z_2=z^4$ in some holomorphic coordinates. We blow up at the origin and let $E$ be the exceptional $(-1)$-sphere. Then it is easy to see that the lifted $\Z_p$-action has two isolated fixed points contained in $E$, with rotation numbers $(1,3)$ and $(1,9)$ respectively. Moreover, the first fixed point is where the proper transform of $C$ and the exceptional $(-1)$-sphere $E$ meet, and their intersection is tangent of order $3$. One can work out a local model for the resolution at this fixed point, like we did in Case (3)(iv). The end result is that the exceptional divisors form a linear chain of a $(-5)$-sphere, $(-2)$-sphere, and $(-2)$-sphere, where each of
the proper transforms of $C/\Z_p$ and $E/\Z_p$ intersects the $(-5)$-sphere transversely at one point. With this understood, if we resolve the singular point of rotation numbers $(1,9)$ on the proper transform of $E/\Z_p$, we get a linear chain of a $(-2)$-sphere, $(-2)$-sphere, and a $(-5)$-sphere, where the proper transform of $E/\Z_p$ becomes a $(-1)$-sphere, intersecting the linear chain at the first $(-2)$-sphere (cf. Lemma 3.2(1)). Now it follows easily that by successively blowing down all the $(-1)$-spheres, we arrive at a linear chain of a  $(-2)$-sphere, $(-2)$-sphere,  $(-2)$-sphere, and a $(-4)$-sphere, where the proper transforms of $C/\Z_p$ intersects the third $(-2)$-sphere transversely at one point, as shown in the graph $\Gamma_{13,2}$ in Figure 1. Observe that from the proof of Theorem 1.3, $g_{|\Sigma/G|}=1$ so that $F_0$ has genus $1$. The determination of the self-intersection of $F_0$ is a standard affair now, which is left to the reader. Finally, the formula for $c_1(K_X)$ is by a direct calculation.

\vspace{2mm}

{\it Case (6):}  In this case, the invariant curve $C$ is of degree $6$, and is embedded.  With the information on the rotation numbers at the fixed points, which are all $(1,12)$, and the fact that 
$g_{|\Sigma/G|}=1$ from the proof of Theorem 1.3, it is a straightforward application of Lemma 3.1. The formula for $c_1(K_X)$ is by a direct calculation.

\vspace{2mm}

This completes the proof of Theorem 3.3.

\end{proof}

\begin{remark}
(1) Using the formula of $c_1(K_X)$ in Theorem 3.3, one can compute $c_1(K_X)^2$ from the symplectic plumbing. One can check easily that $c_1(K_X)^2$ is always an integer and $c_1(K_X)^2=2\chi(X)+3\sigma(X)=10-b_2(X)$ is always true.

(2) In constructing the symplectic resolution of a symplectic $4$-orbifold with isolated singular points, or performing
a symplectic blowing-up, one can always arrange so that the exceptional divisors have arbitrarily small areas 
(see \cite{C11}, or \cite{C3}, Lemma 3.1). It follows easily that there are symplectic structures $\omega$ on the symplectic resolution $X$ such that $|c_1(K_X)\cdot [\omega]-\frac{3}{d} \omega(F_0)|$ is arbitrarily small
(while $\omega(F_0)\geq \delta_0$ for some $\delta_0>0$), where $d$ is the degree of the invariant curve $C$.
Since the surfaces $F_{i,j}$ can have arbitrarily small areas, it follows immediately that in the expression of 
$c_1(K_X)$ in Theorem 3.3, the coefficient before $F_0$ must be $\frac{3}{d}$, which is indeed the case. In particular, this implies that $c_1(K_X)\cdot [\omega]>0$.
\end{remark}

\subsection{Minimality and symplectic Kodaira dimension of the resolution}

We investigate the minimality and symplectic Kodaira dimension of the symplectic resolution $X$. The main tool is the formula of $c_1(K_X)$ in Theorem 3.3.

\begin{theorem}
(1) $X$ is minimal in Case (1), Case (2), Case (3)(i), and Case (3)(v) of Theorem 1.3. Moreover, $X$ has 
symplectic Kodaira dimension $2$ in Case (1) and Case (2), and has symplectic Kodaira dimension $1$ in 
Case (3)(i) and Case (3)(v).

(2) $X$ is a rational $4$-manifold in Case (4)(ii) and Case (5)(i) of Theorem 1.3.
\end{theorem}

\begin{proof}
First, observe that in the expression of $c_1(K_X)$ in Theorem 3.3 for Case (1), Case (2), Case (3)(i), and 
Case (3)(v), and only for these cases, the coefficients before the surfaces $F_0, F_{i,j}$ are all non-negative. 
With this understood, recall that by work of Taubes \cite{T}, a symplectic $4$-manifold is minimal if and only if 
there is no symplectic $(-1)$-sphere in it (for the case of $b_2^{+}=1$, see \cite{Li0}). Now assume to the contrary 
that $X$ is non-minimal in any of the above cases. Then there is a symplectic $(-1)$-sphere representing a 
$(-1)$-class $E$. In fact, more is true, i.e., the Gromov invariant associated to $E$ is nonzero. As a consequence, for any $\omega$-compatible almost complex structure $J$, the class $E$ is represented by a finite irreducible $J$-holomorphic curves $\{C_k\}$ with multiplicity $\{n_k\}$, i.e., $E=\sum_{k} n_k C_k$ where $n_k>0$. On the other hand, since the symplectic surfaces $F_0,F_{i,j}$ have transversal, positive intersections, after a small perturbation if necessary to make them meet $\omega$-orthogonally (cf. \cite{Gompf}), there is an $\omega$-compatible almost complex structure $J$, such that $F_0$ and each $F_{i,j}$ are $J$-holomorphic curves. With this understood, for any given $C_k$, if $C_k$ is not a component of the symplectic divisor, i.e., $F_0$ or $F_{i,j}$, then $F_0\cdot C_k\geq 0$, $F_{i,j}\cdot C_k\geq 0$, by positivity of intersection of $J$-holomorphic curves, which implies that $c_1(K_X)\cdot C_k\geq 0$ because the coefficients in the expression of $c_1(K_X)$ are all non-negative. If $C_k=F_0$ or $F_{i,j}$, then we also have $c_1(K_X)\cdot C_k\geq 0$, because $c_1(K_X)\cdot F_0\geq 0$ and $c_1(K_X)\cdot F_{i,j}\geq 0$ for each $F_{i,j}$. It follows immediately that $c_1(K_X)\cdot E\geq 0$, which is a contradiction because $c_1(K_X)\cdot E=-1$. This proves that $X$ is minimal in these cases.

Now by the classification of minimal symplectic $4$-manifolds (see Li \cite{Li}), $X$ has symplectic Kodaira dimension $2$ in Case (1) and Case (2) because in these two cases, $c_1(K_X)^2>0$, and in  Case (3)(i) and 
Case (3)(v), $X$ has symplectic Kodaira dimension $1$  because $c_1(K_X)^2=0$. (Here in both cases, we also used the fact that $c_1(K_X)\cdot [\omega]>0$, see Remark 3.4(2).)

Next, we show that $X$ is a rational $4$-manifold in Case (4)(ii) and Case (5)(i). Consider Case (4)(ii) first. Note
that $F_0$ is a symplectic $(-1)$-sphere. If we symplectically blow down $F_0$ and let $(X^\prime,\omega^\prime)$
be the resulting symplectic $4$-manifold, then we have $c_1(K_{X^\prime})\cdot [\omega^\prime]=
c_1(K_X)\cdot [\omega]-\omega(F_0)$. On the other hand, as we have explained in Remark 3.4(2), $|c_1(K_X)\cdot [\omega]-\frac{3}{d} \omega(F_0)|$ can be made arbitrarily small while keeping $\omega(F_0)\geq \delta_0$ for some $\delta_0>0$, where $d$ is the degree of the invariant curve $C$. One has $d=6$ in Case (4)(ii), so that 
$c_1(K_X)\cdot [\omega]$ is approximately $\frac{1}{2}\omega(F_0)$. It follows immediately that
$c_1(K_{X^\prime})\cdot [\omega^\prime]<0$, which implies that $X$ is a rational $4$-manifold (see \cite{Li}).
The argument for Case (5)(i) is similar, except that in this case, one has $d=2$ instead. Consequently, 
$c_1(K_X)\cdot [\omega]$ is approximately $\frac{3}{2}\omega(F_0)$, and 
$c_1(K_{X^\prime})\cdot [\omega^\prime]$ is approximately $\frac{1}{2}\omega(F_0)$. To see that $X$ is also a rational $4$-manifold in this case, we observe that $(X^\prime,\omega^\prime)$, the symplectic $4$-manifold resulted from blowing the $(-1)$-sphere $F_0$, also contains a symplectic $(-1)$-sphere, denoted by $E$, which is the descendant of the surface $F_{3,1}$ (see the weighted graph $\Gamma_{13,1}$ in Figure 1). Note that 
$\omega^\prime(E)=\omega(F_{3,1}+F_0)>\omega(F_0)$. Now if we 
symplectically blow down $E$ and let $(\tilde{X},\tilde{\omega})$ be the resulting symplectic $4$-manifold, we
have $c_1(K_{\tilde{X}})\cdot [\tilde{\omega}]=c_1(K_{X^\prime})\cdot [\omega^\prime]-\omega^\prime(E)<0$,
which implies that  $X$ is a rational $4$-manifold. This finishes the proof of Theorem 3.5.

\end{proof}

\begin{remark}
(1) For the remaining cases of Theorem 1.3, the symplectic Kodaira dimension of $X$ is unknown at the moment.
However, except for Case (3)(iii), it is clear that $X$ is non-minimal in these remaining cases, as it is easily seen that $c_1(K_X)^2<0$ in these cases (see Li \cite{Li}).

(2) Our conclusion in Theorem 3.5(1), for Case (1), Case (3)(i), and Case (3)(v), is consistent with the corresponding
result concerning prime order automorphism groups of fake projective planes, see Keum \cite{Keum1, Keum2}. The assertion concerning Case (2), as well as Theorem 3.5(2) in regard of Case (4)(ii) and Case (5)(i), do not have any algebraic analog and are new phenomena in the symplectic category.

\end{remark}

\subsection{Automorphisms of a fake projective plane as a symplectic $\Z_p$-action} Here we examine the only known examples of symplectic $\Z_p$-actions on a $\Q$-homology $\C\P^2$
with positive canonical class, i.e., a prime order $p$ automorphism group of a fake projective plane, from the point of view of Theorem 1.3. The case $p=3$ is trivial, as it is obvious that every $\Z_3$-automorphism belongs to Case (1), which is the only possibility by Theorem 1.3, although according to \cite{CS}, there are at least six fake projective planes admitting a $\Z_3$-automorphism. 

The situation for $p=7$ is more delicate. By the classification in Keum \cite{Keum2}, the minimal resolution $X$ of the quotient space of a fake projective plane modulo an order $7$ automorphism group is a minimal elliptic surface of Kodaira dimension $1$ with two multiple fibers, with the pair of the multiplicities falling into one of the following 
three cases: $(2,3)$, $(2,4)$, or $(3,3)$. Note that the three cases can be distinguished by the fundamental 
group of $X$, i.e., $\pi_1(X)=\{1\}$, $\Z_2$, $\Z_3$ respectively. In fact, in the first case $X$ is a Dolgachev surface, which was actually used by Keum in \cite{Keum1} to construct a fake projective plane with an order $7$ automorphism group. On the other hand, according to the computations in \cite{CS}, the second case is known to be realized by a fake projective plane. However, whether the third case (i.e., both multiple fibers in $X$ have multiplicity $3$) is realized is unknown. Finally, we note that as a symplectic $\Z_7$-action, the automorphism group belongs to either Case (3)(i) or Case (3)(v) by the local representations at the fixed points. 

Since the almost complex structure $J$ used in Theorem 1.3 is assumed to be generic, so that it may not apply to the holomorphic structure on the fake projective plane, we need the following technical lemma.

\begin{lemma}
Let $X$ be the minimal resolution of the quotient space of a fake projective plane $M$ modulo an order $7$ automorphism group. There exists a symplectic structure $\omega$ on $X$ such that the holomorphic structure is
$\omega$-tame. Moreover, there exists a sequence $J_n$ of $\omega$-tame almost complex structures on $X$,
together with a sequence of symplectic divisors as obtained in Theorem 3.3, where the symplectic surfaces are $J_n$-holomorphic, such that $J_n$ converges in $C^\infty$ to the holomorphic structure on $X$ and the symplectic surfaces $F_{0,n}$ in the symplectic divisors (i.e., the proper transform of the quotient of the invariant curve in $X$) converges in a Gromov limit to $mF$, where $F$ is a multiple fiber of the elliptic fibration on $X$. 
\end{lemma}

\begin{proof}
Let $\omega_0$ be a K\"{a}hler form on the fake projective plane $M$. At each fixed point of the order $7$ automorphism group, we fix a neighborhood $U$. Then in the complement of these neighborhoods we choose a
sequence $J_n^\prime$ of generic $\omega_0$-compatible almost complex structures as in Theorem 1.3, which
converges in $C^\infty$ to the holomorphic structure on $M$. By Theorem 1.3, there is an invariant $J_n^\prime$-holomorphic curve $C_n$, and moreover, by Theorem 3.3, there is a corresponding plumbing of symplectic 
surfaces in $X$. Let $\pi: X\rightarrow M/\Z_p$ be the resolution map. We note that $J_n^\prime$ descends to an almost complex structure on $M/\Z_p$ under the quotient map $M\rightarrow M/\Z_p$, which is the holomorphic structure on each neighborhood $U/\Z_p$ of the singular point, hence can be extended naturally to $X$. 
We denote this almost complex structure on $X$ by $J_n$. It is clear that $J_n$ converges in $C^\infty$ to the holomorphic structure on $X$ and the surfaces in the symplectic plumbings are $J_n$-holomorphic. 

To construct the symplectic structure $\omega$ such that each $J_n$ is $\omega$-tame, we fix a K\"{a}hler form 
$\Omega$ on each $\pi^{-1}(U/\Z_p)$. Let $E\subset \pi^{-1}(U/\Z_p)$ be the exceptional set. Then note that on $\pi^{-1}(U/\Z_p)\setminus E$, $\Omega=d\gamma$ for some $1$-form 
$\gamma$. We fix a cut-off function $\rho$ which equals $1$ in a neighborhood of $E$ and equals zero near the
boundary of $\pi^{-1}(U/\Z_p)$, and define $\omega:=\pi^\ast\omega_0+ \epsilon d(\rho\gamma)$
for a sufficiently small $\epsilon>0$ (cf. \cite{C11}). Then $\omega=\pi^\ast\omega_0+\epsilon \Omega$ near
$E$; in particular, $\omega$ is K\"{a}hler near $E$, and $\omega=\omega_0$ outside each $\pi^{-1}(U/\Z_p)$. Over these regions, $J_n$ is clearly $\omega$-compatible. On the other hand, on each $\pi^{-1}(U/\Z_p)\setminus E$, 
$\omega$ is a small perturbation of $\omega_0$, over which $J_n$ is $\omega$-tame. 

By Gromov compactness, the sequence $F_{0,n}$ of $J_n$-holomorphic curves converges to a Gromov limit,
$\sum_i n_i C_i$, where $C_i$ is an irreducible holomorphic curve in $X$ and $n_i>0$, such that
$\sum_i n_i C_i$ has the same homology class of $F_{0,n}$, which is a multiple of $c_1(K_X)$. It follows easily 
that $\sum_i n_i C_i$ must be a fiber of the elliptic fibration on $X$. By the formula of $c_1(K_X)$ for elliptic surfaces (cf. \cite{FM}, Proposition 3.21), $c_1(K_X)=\frac{1}{6} T$, $\frac{1}{4} T$, and $\frac{1}{3} T$ respectively in the three cases of $X$, where $T$ is a regular fiber of the elliptic fibration. Since by Theorem 3.3(2), $c_1(K_X)=F_{0,n}$ or $\frac{1}{2} F_{0,n}$, it follows easily that $\sum_i n_i C_i=mF$ where  $F$ is a multiple fiber of the elliptic fibration. 

\end{proof}

We observe that in Case (3)(i), $c_1(K_X)=F_0$, where $F_0$ is clearly a primitive class. Hence in this case, we must have $m=1$ in Lemma 3.7. Similarly, in Case (3)(v), we have either $m=1$ or $m=2$. It follows easily that
if an order $7$ automorphism group of a fake projective plane belongs to Case (3)(i), the minimal resolution $X$
must have a multiple fiber representing the canonical class $c_1(K_X)$. In particular, if $X$ has no such a
multiple fiber, the order $7$ automorphism group must belong to Case (3)(v), and in this case, $X$ must have a
multiple fiber representing $2c_1(K_X)$. On the other hand, if $X$ does have a multiple fiber which represents $c_1(K_X)$, then one can easily see a plumbing in $X$ which is according to the graph $\Gamma_{7,1}$ in 
Figure 1. Simply note that the $(-3)$-sphere in each linear chain of the exceptional spheres in $X$ is a multi-section of the elliptic fibration, intersecting the multiple fiber transversely at one point, and the two $(-2)$-spheres must lie in a reduced fiber hence are disjoint from the multiple fiber. However, the following lemma implies that 
the order $7$ automorphism group in this case could belong to either Case (3)(i) or Case (3)(v).

\begin{lemma}
In a symplectic plumbing according to the graph $\Gamma_{7,1}$ in Figure 1, there always exists a symplectic plumbing according to the graph $\Gamma_{7,5}$ in Figure 1. Moreover, the surface $F_0$ in the latter plumbing
can be chosen to be in an arbitrarily small neighborhood of the surface $F_0$ in the former plumbing.
\end{lemma}

\begin{proof}
By Weinstein neighborhood theorem, the surface $F_0$ in the symplectic plumbing according to $\Gamma_{7,1}$,
which is an embedded torus of self-intersection $0$, has a standard neighborhood given by 
$(T^2\times D^2, \omega_0)$, where $\omega_0:=d\theta_1\wedge d\theta_2+ rdr\wedge d\theta$.
Here $(\theta_1,\theta_2)$ is the standard coordinate on $T^2$ and $(r,\theta)$ is the polar coordinate on
the disc $D^2$. Moreover, the three $(-3)$-spheres in the symplectic configuration are given by $\{x_i\}\times D^2$
in the neighborhood for three distinct points $x_i\in T^2$, $i=1,2,3$.

With the preceding understood, for any $r_0>0$, consider the following embedded torus $T$ of self-intersection $0$, which is the image of the embedding
$$
\varphi: (s,t)\mapsto (\theta_1,\theta_2, r_0,\theta)=(s, 2t, r_0, t).
$$
Note that $\varphi^\ast \omega_0=2 ds\wedge dt>0$ so that $T$ is symplectic, and $T$ intersects each disc 
$\{x_i\}\times D^2$ transversely and positively at two points. It is easy to see that there is a symplectic plumbing 
with $T$ being the surface $F_0$, which is according to the graph $\Gamma_{7,5}$. Furthermore, note that $T$ can be chosen arbitrarily close to the surface $F_0$ in the original plumbing by taking a sufficiently small $r_0>0$. This finishes the proof.

\end{proof}

With the preceding understood, we note that if $X$ has two multiple fibers with multiplicities $(2,3)$, then 
there is no multiple fiber in $X$ representing the canonical class $c_1(K_X)$, but $2c_1(K_X)$ is represented by
the multiple fiber with multiplicity $3$. On the other hand, in the remaining two cases, i.e., $X$ has two multiple 
fibers with multiplicities $(2,4)$ or $(3,3)$, there is a multiple fiber in $X$ representing $c_1(K_X)$, i.e., 
the fiber with multiplicity $4$ and $3$ respectively. We summarize our discussions in this subsection in the 
following theorem. 

\begin{theorem}
(1) Any order $3$ automorphism group of a fake projective plane belongs to Case (1) in Theorem 1.3.

(2) For an order $7$ automorphism group of a fake projective plane, if the minimal resolution $X$ of its quotient space is simply connected, i.e., $X$ is a Dolgachev surface, then the automorphism group must belong to Case (3)(v) in Theorem 1.3. If $\pi_1(X)$ is nontrivial, then $X$ contains two symplectic plumbings, one is according to the graph $\Gamma_{7,1}$ and the other according to $\Gamma_{7,5}$. In this case, the automorphism group could belong to either Case (3)(i) or Case (3)(v) in Theorem 1.3.
\end{theorem}

\section{From symplectic fillings to symplectic $4$-manifolds with group actions}

In this section, we describe the reverse-engineering construction which recovers the symplectic 
$\Q$-homology $\C\P^2$ with positive canonical class, equipped with a pseudofree symplectic 
$\Z_p$-action of prime order $p>2$, from a symplectic filling of the corresponding contact plumbed manifold. Under the most general assumption, this procedure actually produces a symplectic $4$-manifold on the Bogomolov-Miyaoka-Yau line, which is of an independent interest. By combining the relevant results from the previous section, we give a proof of Theorem 1.1. 
 
\subsection{The Li-Mak contact structure of a plumbed manifold}
We first recall the relevant results of Li and Mak in \cite{LM}, which will be formulated slightly 
differently here. Let $Y$ be an oriented graph manifold. We consider the set of weighted graphs 
$\Gamma$ which has the following properties:
\begin{itemize}
\item The intersection matrix of $\Gamma$ is non-singular and not negative definite.
\item The plumbing $P_\Gamma$ according to $\Gamma$ is a plumbing of symplectic surfaces with positive normal crossings, and $Y=-\partial P_\Gamma$ as oriented manifolds.
\end{itemize}
Given any such a weighted graph $\Gamma$, Li and Mak \cite{LM} showed that there is a symplectic structure on $P_\Gamma$ such that the plumbing surfaces in $P_\Gamma$ form a symplectic divisor 
$D_\Gamma$ with normal crossings, and $Y$ is a concave contact boundary of $P_\Gamma$, where the contact structure on $Y$ depends only on the weighted graph $\Gamma$ up to a contact isotopy. 

Note that for a given graph manifold $Y$, if such a set of weighted graphs $\Gamma$ is nonempty, the plumbings $P_\Gamma$ are not necessarily uniquely determined by $Y$, but are related to each other via certain standard moves, although if the plumbing $P_\Gamma$ is in a certain normal form,
the graph $\Gamma$ is uniquely determined by the diffeomorphism class of $Y$ (see Neumann \cite{N1}). With this understood, in our consideration here we do not require $P_\Gamma$ to be in a normal form, but rather, in view of the work of Li and Mak \cite{LM}, we shall regard $P_\Gamma$ as an extra structure imposed on $Y$, i.e., a plumbing structure, and call the graph manifold $Y$ a {\bf plumbed manifold} and denote it by $Y_\Gamma$ (if the plumbing is done according to the graph 
$\Gamma$). With this understood, note that $Y_\Gamma$ is endowed with a canonical contact structure by Li and Mak \cite{LM}, i.e., the Li-Mak contact structure $\xi_{LM}$.  We shall call 
$(Y_\Gamma,\xi_{LM})$ a {\bf contact plumbed manifold}. 

In general, the Li-Mak contact structures are not
well understood, particularly regarding its tightness or fillability, because $(Y_\Gamma,\xi_{LM})$ is
a concave contact boundary of $P_\Gamma$ so it does not imply tightness and fillability of the contact structure. With this said, we note that under suitable assumptions, there is an open book decomposition supporting the Li-Mak contact structure, which is canonically constructed from the plumbing (cf. \cite{LMin}, see also \cite{G1,G2, EO, GS}). 

The Li-Mak contact structures are closely connected to the problem of embedding 
certain symplectic divisors with normal crossings into a closed symplectic $4$-manifold. More 
concretely, it is immediately evident that if the contact plumbed manifold $(Y_\Gamma,\xi_{LM})$ 
admits a strong symplectic filling $W_\Gamma$, then $X_\Gamma:=W_\Gamma\cup_{(Y_\Gamma, \xi_{LM})}
P_\Gamma$ is a closed symplectic $4$-manifold containing an embedding of the symplectic divisor 
$D_\Gamma$. On the other hand, if $X$ is a closed symplectic $4$-manifold which contains 
$D_\Gamma$ as a symplectic divisor with normal crossings, then according to Li and Mak \cite{LM},
after deforming the symplectic structure on $X$ if necessary, while keeping the divisor $D_\Gamma$ symplectic, with respect to the new symplectic structure on $X$, the divisor $D_\Gamma$ has a regular neighborhood $V_\Gamma\subset X$ with concave contact boundary 
$(Y_\Gamma,\xi_{LM})$, which can be taken to be arbitrarily small. With this understood, note that
the complement $W_\Gamma:=X\setminus int(V_\Gamma)$ is a strong symplectic filling of 
$(Y_\Gamma,\xi_{LM})$. In summary, the embeddability of the symplectic divisor $D_\Gamma$ in a closed symplectic $4$-manifold is equivalent to the strong symplectic fillability of the contact plumbed 
manifold $(Y_\Gamma,\xi_{LM})$.

With the preceding understood, we come back to the discussions on symplectic $\Z_p$-actions on a $\Q$-homology $\C\P^2$. For each weighted graph $\Gamma$ in Figure 1, from the way it was constructed, it is clear that the corresponding symplectic plumbing $P_\Gamma$ has an intersection matrix which is non-singular and not negative definite. Consequently, for each $\Gamma$ in Figure 1, there is an associated contact plumbed manifold 
$(Y_\Gamma,\xi_{LM})$. Note that here we do not assume that the graph $\Gamma$ is realized by a symplectic 
$\Z_p$-action on a $\Q$-homology $\C\P^2$. In other words, the contact plumbed manifolds 
$(Y_\Gamma,\xi_{LM})$ are those which could potentially arise from a symplectic $\Z_p$-action on a $\Q$-homology $\C\P^2$. 

In view of Theorems 3.3 and 3.9, the following is straightforward. 

\begin{theorem}
(1) For each weighted graph $\Gamma$ in Figure 1, if the corresponding symplectic $\Z_p$-action can be realized,
then the contact plumbed manifold $(Y_\Gamma,\xi_{LM})$ must be strongly symplectically fillable; in particular,
it must be tight. 

(2) For $\Gamma=\Gamma_3, \Gamma_{7,1}$ and $\Gamma_{7,5}$ in Figure 1,  the contact plumbed manifold 
$(Y_\Gamma,\xi_{LM})$ is strongly symplectically fillable.
\end{theorem}

We also note that by Lemma 3.8, any strong symplectic filling of $(Y_\Gamma,\xi_{LM})$, where 
$\Gamma=\Gamma_{7,1}$, canonically gives rise to a strong symplectic filling of $(Y_\Gamma,\xi_{LM})$ for $\Gamma=\Gamma_{7,5}$.

\subsection{A reverse-engineering procedure} 
Let $\Gamma$ be any given weighted graph in Figure 1, and let $P_\Gamma$ be the corresponding 
plumbing, which is a connected, oriented $4$-manifold with boundary.  We note that as oriented manifolds, $Y_\Gamma=-\partial P_\Gamma$. As a plumbing, $P_\Gamma$ is homotopy equivalent to a CW $2$-complex, so that $H_k(P_\Gamma)=H^k(P_\Gamma)=0$ for any $k\geq 3$. On the other hand, it follows easily that $H_1(P_\Gamma)$ and $H_2(P_\Gamma)$ are both free Abelian groups, 
with $H_1(P_\Gamma)$ being the direct sum of $H_1(F_0)$ and the group generated by the loops in the graph $\Gamma$, and $H_2(P_\Gamma)$ is generated by the surfaces $F_0,F_{i,j}$ which form
a basis of $H_2(P_\Gamma)$ (see Theorem 3.3 for the notations $F_0,F_{i,j}$). 

We observe the following consequence of the fact that the intersection matrix of $P_\Gamma$ is non-singular: through Poincare duality $H^2(P_\Gamma,\Q)$ is identified with $H_2(P_\Gamma,Y_\Gamma,\Q)$, so that the intersection matrix of $P_\Gamma$ being non-singular implies that $j_\ast: H_2(P_\Gamma,\Q)\rightarrow H_2(P_\Gamma,Y_\Gamma,\Q)$ is an isomorphism in the long exact sequence associated to the pair $(P_\Gamma,Y_\Gamma)$. 
Consequently, $i_\ast: H_2(Y_\Gamma,\Q)\rightarrow H_2(P_\Gamma,\Q)$ must be the zero map and 
$i_\ast: H_1(Y_\Gamma,\Q)\rightarrow H_1(P_\Gamma,\Q)$ must be injective. With 
$H_1(P_\Gamma,Y_\Gamma,\Q)=H^3(P_\Gamma,\Q)=0$, we conclude that 
$i_\ast: H_1(Y_\Gamma,\Q)\rightarrow H_1(P_\Gamma,\Q)$ is an isomorphism. 

\begin{lemma}
$i_\ast: H_2(Y_\Gamma,\Q)\rightarrow H_2(P_\Gamma,\Q)$ is the zero map and 
$i_\ast: H_1(Y_\Gamma,\Q)\rightarrow H_1(P_\Gamma,\Q)$ is an isomorphism. 
\end{lemma}

Next, let $W_\Gamma$ be any connected, oriented $4$-manifold with $Y_\Gamma=
\partial W_\Gamma$, and we set $X_\Gamma:=W_\Gamma\cup_{Y_\Gamma} P_\Gamma$ to be the
closed oriented $4$-manifold. 

\begin{lemma}
$H_1(X_\Gamma,\Q)=0$ if and only if $H_1(W_\Gamma,\Q)=0$, and $i_\ast: H_2(P_\Gamma,\Q)\rightarrow H_2(X_\Gamma,\Q)$ is an isomorphism if and only if $i_\ast: H_2(Y_\Gamma,\Q)\rightarrow H_2(W_\Gamma,\Q)$ is surjective.
\end{lemma}

We remark that if $H_1(W_\Gamma,\Q)=0$, then $H_3(W_\Gamma,Y_\Gamma,\Q)=H^1(W_\Gamma,\Q)=0$, which implies that $i_\ast: H_2(Y_\Gamma,\Q)\rightarrow H_2(W_\Gamma,\Q)$ is injective.
Hence $H_1(X_\Gamma,\Q)=0$ and $i_\ast: H_2(P_\Gamma,\Q)\rightarrow H_2(X_\Gamma,\Q)$ is an isomorphism if and only if $H_1(W_\Gamma,\Q)=0$ and $i_\ast: H_2(Y_\Gamma,\Q)\rightarrow H_2(W_\Gamma,\Q)$ is an isomorphism. 

\begin{proof}
By the fact that $i_\ast: H_1(Y_\Gamma,\Q)\rightarrow H_1(P_\Gamma,\Q)$ is an isomorphism, it follows easily from the Mayer-Vietoris sequence that $H_1(X_\Gamma,\Q)=0$ if and only if $H_1(W_\Gamma,\Q)=0$.

For the second assertion, we first note that $i_\ast: H_2(P_\Gamma,\Q)\rightarrow H_2(X_\Gamma,\Q)$ is always injective. This is because $\delta_\ast: H_3(X_\Gamma,P_\Gamma,\Q)\rightarrow H_2(P_\Gamma,\Q)$ in the long exact sequence has zero image. To see this, we note that 
by excision, $H_3(X_\Gamma,P_\Gamma,\Q)=H_3(W_\Gamma,Y_\Gamma,\Q)$, so that the image of 
$\delta_\ast: H_3(X_\Gamma,P_\Gamma,\Q)\rightarrow H_2(P_\Gamma,\Q)$ is the same as the image of the composition of maps $\delta_\ast: H_3(W_\Gamma,Y_\Gamma,\Q)\rightarrow H_2(Y_\Gamma,\Q)$ and $i_\ast:  H_2(Y_\Gamma,\Q)\rightarrow H_2(P_\Gamma,\Q)$. Our claim follows immediately from the fact that $i_\ast: H_2(Y_\Gamma,\Q)\rightarrow H_2(P_\Gamma,\Q)$ is the zero map (cf. Lemma 4.2).  

With the preceding understood, we note from the long exact sequence that $i_\ast: H_2(P_\Gamma,\Q)\rightarrow H_2(X_\Gamma,\Q)$ is surjective if and only if $\delta_\ast: 
H_2(X_\Gamma,P_\Gamma,\Q)\rightarrow H_1(P_\Gamma,\Q)$ is injective. By excision and the fact that $i_\ast: H_1(Y_\Gamma,\Q)\rightarrow H_1(P_\Gamma,\Q)$ is an isomorphism, the latter is equivalent to $\delta_\ast: H_2(W_\Gamma,Y_\Gamma,\Q)\rightarrow H_1(Y_\Gamma,\Q)$ is injective. Under Poincare duality, this is equivalent to $i_\ast: H_2(Y_\Gamma,\Q)\rightarrow H_2(W_\Gamma,\Q)$ is surjective.

\end{proof}

To proceed further, we shall next consider a smooth oriented $4$-orbifold with boundary 
$-Y_\Gamma$, to be denoted by $\overline{P}_\Gamma$, which is constructed from the plumbing $P_\Gamma$ as follows. Recall from the proof of Theorem 3.3 that in the configuration of surfaces in
$P_\Gamma$, there are three disjoint linear chains of spheres which correspond to the exceptional divisors in the minimal resolution of the three singularities of the quotient orbifold of the $\Z_p$-action the graph $\Gamma$ is associated to. If we remove a small regular neighborhood of the three linear chains of spheres from $P_\Gamma$, we get a $4$-manifold $\Omega_\Gamma$, which has boundary $-Y_\Gamma$ plus three lens spaces, each with a $\Z_p$-fundamental group. The 
$4$-orbifold $\overline{P}_\Gamma$ is obtained from $\Omega_\Gamma$ by capping off each of the lens spaces with an appropriate cone $B^4/\Z_p$. We remark that if we endow $P_\Gamma$ with a symplectic structure via symplectic plumbing, $\overline{P}_\Gamma$ is naturally a symplectic $4$-orbifold.

\begin{lemma}
There exists an oriented smooth $4$-manifold $\tilde{P}_\Gamma$ with a connected boundary 
$-\tilde{Y}_\Gamma$, which is equipped with a smooth $\Z_p$-action, such that $\overline{P}_\Gamma$ is the quotient orbifold of $\tilde{P}_\Gamma$ modulo the $\Z_p$-action. Moreover,
the homomorphism $i_\ast: \pi_1(\tilde{Y}_\Gamma)\rightarrow \pi_1(\tilde{P}_\Gamma)$ is surjective, and the map $i_\ast: H_1(\tilde{Y}_\Gamma,\Q)\rightarrow H_1(\tilde{P}_\Gamma,\Q)$ is an isomorphism. 
\end{lemma}

\begin{proof}
After removing a small regular neighborhood of the three linear chains of spheres from $P_\Gamma$, 
the ``central" surface $F_0$ in the configuration of surfaces in $P_\Gamma$ becomes an oriented surface with boundary in $\Omega_\Gamma$, to be denoted by $S$, such that $\Omega_\Gamma$ is given by the union of $S\times D^2$ with a disjoint union of three lens spaces times an interval $I$, under a suitable embedding of $\partial S\times D^2$ into the union of lens spaces. With this understood, recall that in the proof of Theorem 1.3, there is the so-called Global Quotient Criterion
the invariant curve $C$ must pass. By Lemma 3.1 of \cite{C2}, what this means, from the construction in Theorem 3.3, is that there is an oriented surface $\tilde{S}$ and a $p$-fold covering $\tilde{S}\rightarrow S$, such that for a suitable embedding of $\partial \tilde{S}\times D^2$ into a disjoint union of three copies of $\s^3$, the $p$-fold covering $\partial \tilde{S}\times D^2\rightarrow \partial S\times D^2$ extends to the standard $p$-fold coverings of $\s^3$ over the corresponding lens spaces. If we let $\tilde{\Omega}_\Gamma$ be the corresponding union of $\tilde{S}\times D^2$ with the disjoint union of three copies of $\s^3$ times the interval $I$, then there is a natural $p$-fold covering $\tilde{\Omega}_\Gamma\rightarrow \Omega_\Gamma$. With this understood, the $4$-manifold $\tilde{P}_\Gamma$ is simply 
$\tilde{\Omega}_\Gamma$ with the three copies of $\s^3$ in the boundary capped off by $B^4$, and the smooth $\Z_p$-action on $\tilde{P}_\Gamma$ is defined by extending the free $\Z_p$-action on 
$\tilde{\Omega}_\Gamma$ to each of the three copies of $B^4$ by an appropriate linear $\Z_p$-action. It is clear from the construction that $\overline{P}_\Gamma$ is the quotient orbifold of 
$\tilde{P}_\Gamma$ modulo the $\Z_p$-action. Moreover, it is clear that the boundary of 
$\overline{P}_\Gamma$ is connected. 

To see that $i_\ast: \pi_1(\tilde{Y}_\Gamma)\rightarrow \pi_1(\tilde{P}_\Gamma)$ is surjective, we note
that $\tilde{P}_\Gamma$ is a deformation retract to a singular oriented surface $C$ inside it. In particular, any loop in $\tilde{P}_\Gamma$ is homotopic to a loop in the boundary of 
$\tilde{P}_\Gamma$, which implies immediately that $i_\ast: \pi_1(\tilde{Y}_\Gamma)\rightarrow \pi_1(\tilde{P}_\Gamma)$ is surjective. Finally, to see that the map $i_\ast: H_1(\tilde{Y}_\Gamma,\Q)\rightarrow H_1(\tilde{P}_\Gamma,\Q)$ is an isomorphism, we note that $H_2(\tilde{P}_\Gamma)$ is the free Abelian group generated by the singular oriented surface $C$ inside $\tilde{P}_\Gamma$, and from the construction in Theorem 3.3, it is easy to see that the self-intersection $C\cdot C>0$. Then as we argued in the proof of Lemma 4.2, it follows that $H_2(\tilde{P}_\Gamma,\Q)\rightarrow H_2(\tilde{P}_\Gamma,\tilde{Y}_\Gamma,\Q)$ is an isomorphism, which implies that $i_\ast: H_1(\tilde{Y}_\Gamma,\Q)\rightarrow H_1(\tilde{P}_\Gamma,\Q)$ is an isomorphism. 

\end{proof}

We note that there is a natural $p$-fold covering $\tilde{Y}_\Gamma\rightarrow Y_\Gamma$. We denote the corresponding surjective homomorphism from $\pi_1(Y_\Gamma)$ to $\Z_p$ by
$\rho: \pi_1(Y_\Gamma)\rightarrow \Z_p$. 

Now let $W_\Gamma$ be an oriented $4$-manifold with boundary $Y_\Gamma$, such that there
is a homomorphism $\bar{\rho}: \pi_1(W_\Gamma)\rightarrow \Z_p$ satisfying 
$\bar{\rho}\circ i_\ast=\rho$, where $i_\ast: \pi_1(Y_\Gamma)\rightarrow \pi_1(W_\Gamma)$. 
Note that $\bar{\rho}: \pi_1(W_\Gamma)\rightarrow \Z_p$ is surjective. We let $\tilde{W}_\Gamma
\rightarrow W_\Gamma$ be the corresponding $p$-fold covering which extends the $p$-fold 
covering $\tilde{Y}_\Gamma\rightarrow Y_\Gamma$. In particular, $\partial \tilde{W}_\Gamma=
\tilde{Y}_\Gamma$. With this understood, we set $M_\Gamma:=\tilde{W}_\Gamma\cup_{\tilde{Y}_\Gamma} \tilde{P}_\Gamma$. We observe that $M_\Gamma$ comes with a natural smooth
$\Z_p$-action such that the quotient orbifold $M_\Gamma/\Z_p=W_\Gamma \cup_{Y_\Gamma}
\overline{P}_\Gamma$. 

\begin{lemma}
Assume that $H_1(W_\Gamma,\Q)=0$ and the map $i_\ast: H_2(Y_\Gamma,\Q)\rightarrow 
H_2(W_\Gamma,\Q)$ is isomorphic. Then $\chi(M_\Gamma)=3$ and $Sign(M_\Gamma)=1$. Moreover, $M_\Gamma$ is a $\Q$-homology $\C\P^2$ if and only if 
$H_1(\tilde{W}_\Gamma,\Q)=0$. 
\end{lemma}

\begin{proof}
First, it follows from Lemma 4.3 that the assumption that $H_1(W_\Gamma,\Q)=0$ and the map $i_\ast: H_2(Y_\Gamma,\Q)\rightarrow H_2(W_\Gamma,\Q)$ is isomorphic implies that
the Euler characteristic $\chi(M_\Gamma/\Z_p)=3$ and the signature $Sign(M_\Gamma/\Z_p)=1$,
as $M_\Gamma/\Z_p=W_\Gamma \cup_{Y_\Gamma}\overline{P}_\Gamma$. 

With the preceding understood, let $g: M_\Gamma\rightarrow M_\Gamma$ be a self-diffeomorphism 
which generates the $\Z_p$-action. Then $g$ has three isolated fixed points in $M_\Gamma$ which
correspond to the three singularities in the $4$-orbifold $\overline{P}_\Gamma$. By the Lefschetz fixed point theorem (cf. \cite{AS}), the Lefschetz number $L(g,M_\Gamma)=3$, where 
$$
L(g,M_\Gamma)=\sum_{k=0}^4 (-1)^k tr(g_\ast: H_k(M_\Gamma,\Q)\rightarrow H_k(M_\Gamma,\Q)). 
$$
With this understood, we observe that 
$$
\sum_{l=1}^{p-1} L(g^l,M_\Gamma)=p\cdot \chi(M_\Gamma/\Z_p)-\chi(M_\Gamma).
$$
With $\chi(M_\Gamma/\Z_p)=3$, and $L(g^l,M_\Gamma)=3$ for each $l=1,2,\cdots,p-1$ 
by the Lefschetz fixed point theorem, it follows easily that $\chi(M_\Gamma)=3$.

Similarly, by the $G$-signature theorem (cf. \cite{AS}), one has 
$$
p \cdot Sign(M_\Gamma/\Z_p)-Sign (M_\Gamma)= \sum_{k=1}^3 \sum_{1\neq \lambda\in\C,
\lambda^p=1} \frac{(1+\lambda^{a_k})(1+\lambda^{b_k})}{(1-\lambda^{a_k})(1-\lambda^{b_k})},
$$
where $(a_k,b_k)$, $k=1,2,3$, are the rotation numbers at the three fixed points of the 
$\Z_p$-action on $M_\Gamma$. Since the fixed points correspond to the orbifold points in 
$\overline{P}_\Gamma$, the rotation numbers $(a_k,b_k)$, up to a simultaneous change of sign
of any of the $(a_k,b_k)$'s, are realized by a linear $\Z_p$-action on $\C\P^2$ (cf. Lemma 2.1). It follows easily that 
$$
p \cdot Sign(M_\Gamma/\Z_p)-Sign (M_\Gamma)= p\cdot Sign(\C\P^2/\Z_p)-Sign(\C\P^2)=p-1.
$$
With $Sign(M_\Gamma/\Z_p)=1$, it follows immediately that $Sign (M_\Gamma)=1$. 

Finally, combining $Sign (M_\Gamma)=1$ with $\chi(M_\Gamma)=3$, we arrive easily at
$$
b_2^{+}(M_\Gamma)=1+b_1(M_\Gamma), \;\; b_2^{-}(M_\Gamma)=b_1(M_\Gamma).
$$
It follows immediately that $M_\Gamma$ is a $\Q$-homology $\C\P^2$ if and only if $b_1(M_\Gamma)=0$.
Since $i_\ast: H_1(\tilde{Y}_\Gamma,\Q)\rightarrow H_1(\tilde{P}_\Gamma,\Q)$ is an isomorphism
by Lemma 4.4, $b_1(M_\Gamma)=0$ if and only if $H_1(\tilde{W}_\Gamma,\Q)=0$ by the 
Mayer-Vietoris sequence.

\end{proof}

With these preparations, we now give a proof of Theorem 1.1.

\vspace{2mm}

{\bf Proof of Theorem 1.1:}

\vspace{2mm}

For part (1), assume the strong symplectic filling $W_\Gamma$ of $(Y_\Gamma,\xi_{LM})$ 
obeys (i) and (ii). By the work of Li and Mak \cite{LM}, we can equip $P_\Gamma$ with a symplectic structure via symplectic plumbing such that it has a concave contact boundary 
$(Y_\Gamma,\xi_{LM})$. With this understood, note that $\overline{P}_\Gamma$ is a symplectic 
$4$-orbifold with a concave contact boundary $(Y_\Gamma,\xi_{LM})$, and 
$W_\Gamma \cup_{(Y_\Gamma,\xi_{LM})}\overline{P}_\Gamma$
is a closed symplectic $4$-orbifold.  It follows that $M_\Gamma$ is symplectic equipped with a symplectic $\Z_p$-action such that $M_\Gamma/\Z_p=W_\Gamma \cup_{(Y_\Gamma,\xi_{LM})}\overline{P}_\Gamma$. By Lemma 4.5, one has $\chi(M_\Gamma)=3$ and $Sign(M_\Gamma)=1$,
and $M_\Gamma$ is a $\Q$-homology $\C\P^2$ if and only if $W_\Gamma$ obeys (iii). Finally, we observe that condition (iv) is equivalent to $M_\Gamma$ being simply connected by the
Seifert-Van Kampen theorem, as $i_\ast: \pi_1(\tilde{Y}_\Gamma)\rightarrow \pi_1(\tilde{P}_\Gamma)$ is surjective by Lemma 4.4.

It remains to show that $M_\Gamma$ has a positive canonical class. To see this, we recall that in the proof of Lemma 4.4, we showed that $H_2(\tilde{P}_\Gamma)$ is a free Abelian group generated by a singular oriented surface $C$ in the interior of $\tilde{P}_\Gamma$. In fact, it is easy to see that more is true: $C$ is invariant under the $\Z_p$-action on $\tilde{P}_\Gamma$ and is symplectic with respect to the symplectic structure on $\tilde{P}_\Gamma$. On the other hand, from the formula for the canonical class $c_1(K_X)$ in Theorem 3.3(2), it follows easily that the canonical class $c_1(K_{\overline{P}_\Gamma})$ of the symplectic $4$-orbifold $\overline{P}_\Gamma$ is represented by a positive multiple of the quotient surface $C/\Z_p$. As a consequence, the canonical class of $\overline{P}_\Gamma$ is positive. Furthermore, it also implies that the restriction of $c_1(K_{\overline{P}_\Gamma})$ to the boundary $Y_\Gamma$ is a torsion class. Note that the latter assertion implies that $c_1(K_{W_\Gamma})$ is a torsion class as well 
because $i_\ast: H_2(Y_\Gamma,\Q)\rightarrow H_2(W_\Gamma,\Q)$ is isomorphic by (ii). It follows easily that the symplectic $4$-orbifold $W_\Gamma \cup_{(Y_\Gamma,\xi_{LM})}\overline{P}_\Gamma$ has a positive canonical class, which in turn implies that $M_\Gamma$ has a positive canonical class. 

For part (2), assume $M$ is a symplectic $\Q$-homology $\C\P^2$ with a positive canonical class
which is equipped with a pseudofree symplectic $\Z_p$-action for an odd prime $p$. By Theorem 1.3 
and Theorem 3.3, the $\Z_p$-action must belong to one of the $12$ cases listed therein, with an associated weighted graph $\Gamma$ in Figure 1, where $\Gamma=\Gamma_p$ or 
$\Gamma_{p,i}$. Moreover, the minimal resolution of the symplectic quotient orbifold $X$ contains a symplectic divisor $D$ with normal crossings whose associated plumbing graph is $\Gamma$. 
By Li-Mak \cite{LM}, after deforming the symplectic structure on $X$ if necessary, $D$ has a regular neighborhood $V_\Gamma$ with a concave contact boundary $(Y_\Gamma,\xi_{LM})$, where $V_\Gamma$ is diffeomorphic to the plumbing $P_\Gamma$. Let $W_\Gamma:=X\setminus int(V_\Gamma)$. Then $W_\Gamma$ is a strong symplectic filling of $(Y_\Gamma,\xi_{LM})$. 
It remains to show that $W_\Gamma$ obeys (i)-(iii). Note that with $V_\Gamma$ identified to the plumbing $P_\Gamma$, $M$ can be identified with $M_\Gamma$. (We remark that under this identification, $\tilde{P}_\Gamma$ is simply a regular neighborhood of the invariant curve $C$ from Theorem 1.3, and $\tilde{W}_\Gamma$ is the complement of $\tilde{P}_\Gamma$ in $M$.)

First, (i) follows from the fact that $X$ is the minimal resolution of the quotient orbifold. Secondly, by Lemma 4.3, (ii) follows from the fact that $H_1(X,\Q)=0$ and $i_\ast: H_2(V_\Gamma,\Q)
\rightarrow H_2(X,\Q)$ is isomorphic. Finally, (iii) follows by Lemma 4.5. 
When $M$ is an exotic $\C\P^2$, (iv) is also obeyed by $W_\Gamma$ by the
Seifert-Van Kampen theorem.

The proof of Theorem 1.1 is completed. 

\vspace{2mm}

We note that by Lemma 4.4, the homomorphism $i_\ast: \pi_1(\tilde{Y}_\Gamma)\rightarrow \pi_1(\tilde{P}_\Gamma)$ is surjective. This implies that the condition $\tilde{W}_\Gamma$ being simply connected represents a special case of (iii) and (iv). With this understood, we observe the following corollary, where the assumptions on the symplectic filling $W_\Gamma$ take a simpler form. 

\begin{corollary}
Suppose there is a weighted graph $\Gamma$ in Figure 1, with the plumbed manifold $Y_\Gamma$ being equipped with the canonical surjective homomorphism $\rho: \pi_1(Y_\Gamma)\rightarrow \Z_p$, such that
$(Y_\Gamma,\xi_{LM})$ admits a strong symplectic filling $W_\Gamma$ with the following properties:
$i_\ast: H_2(Y_\Gamma,\Q)\rightarrow H_2(W_\Gamma,\Q)$ is isomorphic and there is an isomorphism $\bar{\rho}: \pi_1(W_\Gamma)\rightarrow \Z_p$ satisfying $\bar{\rho}\circ i_\ast=\rho$, where $i_\ast: \pi_1(Y_\Gamma)\rightarrow \pi_1(W_\Gamma)$. Then there exists a symplectic 
exotic $\C\P^2$ with a nontrivial finite order symplectomorphism. 
\end{corollary}

\begin{acknowledgments}
During the preparation of this work, no instruments, software, or any AI related tools were used.
Many thanks to R. Inanc Baykur for the careful reading of the paper and for pointing out the errors found in the earlier version. 
\end{acknowledgments}


\vspace{2mm}

{\Small University of Massachusetts, Amherst.\\
{\it E-mail:} wch@umass.edu

\end{document}